\documentclass[a4paper]{article}
\usepackage[english]{babel}
\usepackage[utf8x]{inputenc}
\usepackage[T1]{fontenc}
\usepackage{bm}
\usepackage{authblk}
\usepackage[a4paper,top=3cm,bottom=2cm,left=3cm,right=3cm,marginparwidth=1.75cm]{geometry}
\usepackage{bbm}
\usepackage[normalem]{ulem}
\usepackage{tikz}
\usepackage[sort&compress,numbers]{natbib}
\usepackage{cancel} 
\usepackage{amsmath}
\usepackage{amsthm}
\usepackage[shortlabels]{enumitem} 
\usepackage{graphicx}
\usepackage[colorinlistoftodos]{todonotes}
\usepackage[colorlinks=true, allcolors=blue]{hyperref}
\newtheorem{theorem}{Theorem}[section]

\newtheorem{lemma}[theorem]{Lemma}
\newtheorem{definition}[theorem]{Definition}
\newtheorem{prop}[theorem]{Proposition}
\usepackage{amssymb}
\usepackage{mathrsfs}
\usepackage{comment}
\newtheorem{assumption}{Assumption}[section]

\newtheorem{remark}{Remark}[section]

\usepackage{dsfont}

\def\XXint#1#2#3{{\setbox0=\hbox{$#1{#2#3}{\int}$ }
		\vcenter{\hbox{$#2#3$ }}\kern-.6\wd0}}

\usepackage{mathtools}

\graphicspath{ {images/} }

\usepackage{xcolor}

\newcommand{\tr}[0]{\text{tr}}

\newcommand{\h}{\hspace}
\newcommand{\p}{\partial}
\newcommand{\RN}[1]{%
  \textup{\uppercase\expandafter{\romannumeral#1}}%
}

\allowdisplaybreaks

\DeclareMathOperator*{\bbotimes}{\text{\raisebox{0.25ex}{\scalebox{0.75}{$\bigotimes$}}}}

\numberwithin{equation}{section}

\DeclareFontFamily{U}{mathx}{}
\DeclareFontShape{U}{mathx}{m}{n}{<-> mathx10}{}
\DeclareSymbolFont{mathx}{U}{mathx}{m}{n}
\DeclareMathAccent{\widehat}{0}{mathx}{"70}
\DeclareMathAccent{\widecheck}{0}{mathx}{"71}

\newcommand{\rcm}[1]{{\color{red}#1}}

\newcommand{\pcm}[1]{{\color{purple}#1}}

\newcommand{\mycomment}[1]{}

\usepackage{scalerel,stackengine}
\newcommand\pig[1]{\scalerel*[5pt]{\big#1}{%
		\ensurestackMath{\addstackgap[1.5pt]{\big#1}}}}

\newcommand\pigr[1]{\mathclose{\pig{#1}}}

\title{ Viscosity Solutions of Fully second-order HJB Equations in the Wasserstein  Space\thanks{{I. Ekren
is partially supported by the NSF grant DMS-2406240. H. Cheung and J. Qiu are partially supported by Discovery Grant from the Natural Sciences and Engineering Research Council of Canada (NSERC).}} }

\date{}

\author[,1]{\normalsize Erhan Bayraktar\footnote{E-mail: erhan@umich.edu}}
\author[,2]{\normalsize Hang Cheung\footnote{E-mail: hang.cheung@ucalgary.ca}}
\author[,1]{\normalsize Ibrahim Ekren\footnote{E-mail: iekren@umich.edu }}
\author[,2]{\normalsize Jinniao Qiu\footnote{E-mail: jinniao.qiu@ucalgary.ca}}
\author[,3]{\normalsize Ho Man Tai\footnote{E-mail: homan.tai@dcu.ie}}
\author[,4]{\normalsize Xin Zhang\footnote{E-mail: xz1662@nyu.edu}}
\affil[1]{\small\it Department of Mathematics, University of Michigan, USA}
\affil[2]{\small\it Department of Mathematics and Statistics, University of Calgary, Canada}
\affil[3]{\small\it School of Mathematical Sciences, Dublin City University, Ireland}
\affil[4]{\small\it Department of Finance and Risk Engineering, New York University, USA}

\begin{document}
	\maketitle
	\begin{abstract}
            
            In this paper, we show that the value functions of mean field control problems with common noise are the unique viscosity solutions to fully second-order Hamilton-Jacobi-Bellman equations, in a Crandall-Lions-like framework. We allow the second-order derivative in measure to be state-dependent and thus infinite-dimensional, rather than derived from a finite-dimensional operator, hence the term ``fully''. Our argument leverages the construction of smooth approximations from particle systems developed by Cosso, Gozzi, Kharroubi, Pham, and Rosestolato [Trans. Amer. Math. Soc., 2023], and the compactness argument via penalization of measure moments in Soner and Yan [Appl. Math. Optim., 2024]. Our work addresses unbounded dynamics and state-dependent common noise volatility, and to our knowledge, this is the first result of its kind in the literature.
	\end{abstract}

	

\section{Introduction}
This paper establishes the existence and uniqueness of viscosity solution to the following Hamilton-Jacobi-Bellman (HJB)  equation in the Wasserstein space arising from mean field control problems with common noise
\begin{align}
 \label{HJB_intro}
			\left\{\begin{aligned}
				&\partial_t u(t,\mu)+\int_{\mathbb{R}^d}
				\sup_{a\in A} \Bigg\{f(t,x,\mu,a)+b(t,x,\mu,a)\cdot \partial_\mu u(t,\mu)(x) \\
				&\h{90pt}+ \dfrac{1}{2}\text{tr}\Big((\sigma(t,x,a)\big[\sigma(t,x,a)\big]^\top+\sigma^0(t,x)[\sigma^0(t,x)]^\top)\nabla_x\partial_\mu u(t,\mu)(x)\Big)\Bigg\}\mu(dx) \\
				&+ \dfrac{1}{2} \int_{\mathbb{R}^d\times\mathbb{R}^d}\text{tr}\Big[\sigma^0(t,x)[\sigma^0(t,y)]^\top\partial_\mu^2
    u(t,\mu)(x,y)\Big] \, \mu^{\otimes 2}(dx,dy)=0,\h{5pt} \forall\, (t,\mu) \in [0,T) \times\mathcal{P}_2(\mathbb{R}^d),\\
				&u(T,\mu)=\int_{\mathbb{R}^d}g(x,\mu)\mu(dx) \h{10pt}\text{for $\mu\in \mathcal{P}_2(\mathbb{R}^d)$},
			\end{aligned}\right.
		\end{align}
where $A$ is the control space, and $b,\sigma,\sigma^0, f,g$ are coefficients of the associated optimization problem. Here $\partial_{\mu}u$ and $\partial^2_{\mu}u$ denote the respective first and second-order $L$-derivatives of $u$ \cite{carmona_probabilistic_2018_vol1,lions_annals}, and \eqref{HJB_intro} is a degenerate fully second-order equation in the Wasserstein space. By ``fully'' second-order, we mean that the second-order term is infinite-dimensional, rather than a finite-dimensional operator as in \cite{bayraktar_comparison_2023,daudin2023well,cheung2023viscosity}. Our manuscript is the first work in the literature to address the viscosity solution directly in the Wasserstein space and to allow for unbounded dynamics and state-dependent common noise volatility.

Over the past two decades, the mean field games and mean field control problems,  as models for strategic decision making among a large population of symmetric agents, have attracted considerable attention. This method was independently introduced by Huang, Malhamé, and Caines \cite{cis/1183728987}, and Lasry and Lions \cite{lasry_mean_2007}. In this framework, it is assumed that symmetric agents interact with each other through a medium known as the mean field term, which represents the collective influence of all agents’ decisions. For a comprehensive introduction to mean field theory, we refer readers to Bensoussan, Frehse, and Yam \cite{bensoussan_mean_2013}, Carmona and Delarue \cite{carmona_probabilistic_2018_vol1}, and Gomes, Pimentel, and Voskanyan \cite{gomes2016regularity}.

Due to the presence of the mean field term, the HJB equations of mean field control problems are defined in the Wasserstein space and involve derivatives of measures, making them infinite-dimensional problems. For the well-posedness of classical solutions to HJB equations in the Wasserstein space, we refer the reader to \cite{lions_annals} by Cardaliaguet, Delarue, Lasry, and Lions, \cite{CCD18} by Chassagneux, Crisan, and Delarue, \cite{Gangbo_Mou_Zhang_AOP} by Gangbo, M{\'e}sz{\'a}ros, Mou, and Zhang, and the references therein. However, the existence of classical solutions requires stringent regularity conditions on the coefficient functions and a monotonicity structure of the model. To analyze the HJB equations to mean-field control problems with general cost functions and state dynamics, the theory of viscosity solutions is indispensable, as it typically requires only continuity conditions and natural growth rates of the coefficient functions. 

To overcome the non-smooth structure of the Wasserstein space, one can lift equations in the Wasserstein space to Hilbert spaces and make use of the existing theory \cite{FGS17,bandini2019randomized,Pham_Wei_Dynamic_Programming}. However, the relation between the solutions to the lifted PDEs and the original ones is unclear; see \cite[Remark 3.6]{cosso_master_2022} for more details. The lifting of a smooth function in the Wasserstein space may not be second-order Fr\'{e}chet differentiable; see e.g. \cite[Example 2.3]{MR3630288}. Moreover, there is no suitable lift version of $\nabla_x \partial_{\mu}$ in the $L^2$ space. 

In this paper, we address the HJB equations directly in the Wasserstein space, establishing the existence and uniqueness of viscosity solution to equation \eqref{HJB_intro} under a Crandall-Lions-like Definition \ref{def. of vis sol}. A viscosity solution to equation \eqref{HJB_intro} is provided by the value function $v$ of the corresponding mean-field control problem; see Theorem~\ref{thm. existence of vis sol}. To prove uniqueness, we adopt the idea of smooth approximation $v_n \to v$, as $n \to \infty$ from \cite{cosso_master_2022,cheung2023viscosity}.  For any subsolution $u_1$, we aim to show that $u_1 \leq v$. Suppose the contrary, $u_1(t_0,\mu_0) > v(t_0,\mu_0)$ at some $(t_0,\mu_0)$ in the interior. Therefore, for sufficiently large $n$, $ \sup_{(t,\mu)} u_1(t,\mu)-v_n(t,\mu)>0$.  To show a contradiction, we want to find a maximum point of $(t,\mu) \mapsto u_1(t,\mu)-v_n(t,\mu)$ and apply the definition of viscosity solution. However, the Wasserstein space is not locally compact and the maximizer may not be obtained. Therefore, \cite{cosso_master_2022,cheung2023viscosity} applied the smooth variational principle so that perturbed functions achieve their global maximum. In \cite{cosso_master_2022,cheung2023viscosity, bayraktar_smooth_2023}, the perturbation terms are not second-order $L$-differentiable, and hence they were not able to prove uniqueness for fully second-order HJB equations in the Wasserstein space. In this paper, to overcome the compactness issue of the Wasserstein space, instead of applying the smooth variational principle, we adopt the idea of moment penalization from \cite{MeYa23}. For any $\mu \in \mathcal{P}_2(\mathbb R^d)$, we define $M_2(\mu):= \int x^2 \, \mu(dx)$. We can always find a small $\delta>0$ such that $u_1(t_0,\mu_0)-v_n(t_0,\mu_0)-\delta M_2(\mu_0)>0$. This implies that the maximum point of $u_1-v_n-\delta M_2$, if exists, must be attained in the set of measures with bounded second moments. As shown in Lemma \ref{lem. compact of V^p_K}, such a set is compact with respect to the finer topology induced by $1$-Wasserstein metric. Moreover, the coefficient functions are assumed to be Lipschitz continuous under the 1-Wasserstein metric in our framework, which implies the same property for the value function. Therefore, the maximum point of $u_1-v_n-\delta M_2$ is attained, and we derive the desired contradiction using the definition of viscosity solution. Another  advantage of this approach is that $M_2(\mu)$ possesses $L$-derivatives up to second-order in a very clean form, $\partial_{\mu} M_2(\mu)(x)=2x$ and $\partial_{\mu}^2 M_2(\mu)(x,y)=0$, which allows unbounded state dynamics and state dependent common noise in our estimate.

We note that the definition of viscosity subsolution in this paper follows the standard Crandall-Lions framework, but we need to modify the definition of viscosity supersolution in a stronger sense when dealing with equations with mean field terms. This modification aims to address the technical difficulty posed by the supremum in equation \eqref{HJB_intro}, which leads to scarce choice of test functions in the proof of the comparison for the supersolution.  Similarly to \cite{cosso_master_2022}, we construct smooth test functions for supersolutions using the cost functional with fixed control, but with the domain of $[0,T] \times \mathcal{P}_2(\mathbb{R}^d\times A)$ rather than $[0,T] \times\mathcal{P}_2(\mathbb{R}^d)$. The strong dependence on the initial random variable and the control makes it difficult to rely solely on the Crandall-Lions definition to draw the conclusion. We note that this modification of the definition is only required for the comparison theorem; however, the value function can still be shown to be the viscosity solution under the Crandall-Lions definition. Nonetheless, for consistency, we shall work on the Crandall-Lions-like definition for both the existence and the uniqueness theorem.  See Remark \ref{rmk modified supersol} for more details.

Let us mention other related references. Another possible approach to proving uniqueness is to apply the doubling variable technique as in the classical case \cite{user_guide}. Suppose $u_1,u_2$ are viscosity sub- and super-solutions respectively. By doubling variable, one compares the derivatives of $u_1,u_2$ at the maximum point of $(t,s,\mu,\nu) \mapsto u_1(t,\mu)-u_2(s,\nu)- \frac{\alpha}{2} (\rho^2(\mu,\nu)+|t-s|^2)$, where $\alpha$ is a large positive constant and $\rho$ is a penalization function that forces the maximizers $\mu_{\alpha},\nu_{\alpha}$ to be close. The main challenge is to construct proper penalization functions $\rho$ tailored to the choice of differentiability in the Wassertein space. In \cite{soner_viscosity_2022,MeYa23}, Soner and Yan used the Fourier-Wasserstein distance as the penalization function, which is smooth in terms of the $L$-differentiability, and proved the well-posedness for first-order HJB equations in the Wasserstein space. Using essentially the same penalization function, inspired by \cite{gangbo2021finite}, in \cite{bayraktar_comparison_2023} Bayraktar, Ekren, and Zhang considered a finite-dimensional second-order operator, what they called the partial Hessian $\mathcal{H}$, as the second-order derivative in the barycenter of probability distributions. Therefore, they were able to apply Ishii's lemma to obtain second-order jets, and proved a comparison principle for fully nonlinear degenerate partially second-order PDEs for Lipschitz continuous  functions. Using the same idea of partial Hessian, Daudin, Jackson, and Seeger \cite{daudin2023well} established the uniqueness of semilinear Hamilton-Jacobi equations for semicontinuous  functions by employing delicate estimates for a sequence of finite-dimensional approximation PDEs and smoothing techniques. Choosing the $2$-Wasserstein distance as the penalization function and using the intrinsic differentiability, in \cite{2023arXiv230604283B} Bertucci proved the well-posedness for a class of equations where the Hamiltonian appears  in the form of 
$ 
H(\mu,\partial_{\mu} u)+\frac{\sigma(t)^2}{2}  \mathcal{H}u(t,\mu)
$ 
for some proper choice of $H$ and deterministic function $\sigma$. Recently, in \cite{BeLi24} Bertucci and Lions proved a comparison principle for  equations involving $\partial_{\mu} u$ and $\nabla_x \partial_{\mu}u$ using the regularity property of sup-convolution.

As shown in \cite{cox2021controlled}, second-order PDEs in the Wasserstein space are also related to measure-valued martingale optimization problems. 
Relying on the specific martingale structure, that paper manages to reduce the problem to finitely supported measures where the usual viscosity theory can be applied. 
\cite{cox2021controlled} proved a general uniqueness result under a novel definition of a viscosity solution that might not enjoy the stability property of viscosity solutions. In \cite{zhou2024viscosity}, Touzi, Zhang, and Zhou tackled the HJB equations by lifting it to the process space, where they introduced a novel concept of viscosity solutions, demonstrating both existence and uniqueness under the assumption of Lipschitz continuity. PDEs in the space of measures also appear in mean-field optimal stopping problems \cite{MR4604196,MR4613226,2023arXiv230709278P} and control problems of occupied processes \cite{STZ24}. The convergence of particle systems in mean-field control problems was studied in \cite{2022arXiv221016004T,2022arXiv221100719T,MR4595996} based on the viscosity theory. The convergence rate for PDEs in the Wasserstein space was obtained in \cite{2023arXiv231211373C,2023arXiv230508423D,2022arXiv220314554C,CDJM24,bayraktar2024convergence}. Assuming the existence of smooth solution to mean-field PDEs, \cite{MR4507678} got the optimal convergence rate by a verification argument.

The rest of the paper is organized as follows. In Section 2, we present the problem formulation along with some preliminary results. Section 3 introduces the smooth finite-dimensional approximation of the value function and provides the related estimates. Section 4 contains the proof of the existence and uniqueness of the viscosity solution.

	\section{Preliminaries}
In this section, we introduce the framework for setting up the HJB equation in \eqref{HJB_intro} and recall some basic results. Before this, we shall introduce some notations which are frequently used in the article. For any random variable $X$, the law of $X$ is denoted by $\mathcal{L}(X)$. For any \(x \in \mathbb{R}^d\), we denote its Euclidean norm by \(|x|\), and its \(i\)-th component by \(x_i\) or \((x)_i\). The standard scalar product of \(x, y \in \mathbb{R}^d\) is written as \(\langle x, y \rangle\) or \(x \cdot y\).  Let \( n \in \mathbb{N} \) and \( x^1, x^2, \ldots, x^n \in \mathbb{R}^d \), we denote the vector \( \overline{x} \in \mathbb{R}^{dn} \) by \( \overline{x} = (x^1, x^2, \ldots, x^n)^\top \). For any matrix \( M \in \mathbb{R}^{d \times d} \), the trace of \( M \) is denoted by \( \textup{tr}(M) := \sum_{i=1}^d M_{ii} \), its transpose by \( M^\top \), and its Frobenius norm by \( |M| := \left[\operatorname{tr}(M M^\top)\right]^{1/2} \). If \( M^0 \in \mathbb{R}^{d \times d} \) is another matrix, its transpose is denoted by \( M^{0;\top} \). The identity matrix in \( \mathbb{R}^d \) is denoted by \( I_d \). For a scalar variable \( x \in \mathbb{R} \), \( \p_x h \in \mathbb{R} \) refers to the partial derivative of the scalar function \( h \) with respect to \( x \). For a vector \( x \in \mathbb{R}^d \), \( \nabla_x h \in \mathbb{R}^d \) denotes the gradient of \( h \).

\subsection{Compactness and Differentiation in the Wasserstein space}
We introduce over $\mathbb{R}^d$ the space of probability measures $\mathcal{P}(\mathbb{R}^d)$ and define the $p$-th moment of any $\mu \in \mathcal{P}(\mathbb R^d)$ by 
\begin{align}
\label{moment_V_def}
M_p(\mu):=\int_{\mathbb{R}^d} |x|^p \, \mu(dx),\quad\text{for any $p\geq1$,}
\end{align}if it is finite.
We also denote $\mathcal{P}_p(\mathbb{R}^d)$ the subset of $\mathcal{P}(\mathbb{R}^d)$ consisting of those with finite $p$-th moment for $p\geq 1$. The space $\mathcal{P}_p(\mathbb{R}^d)$ is typically equipped with the $q$-Wasserstein distance with $q\in[1,p]$ and
		$$
		\mathcal{W}_q(\mu, \nu):=\inf _{\pi \in \Pi(\mu, \nu)}\Bigg(\int_{\mathbb{R}^d \times \mathbb{R}^d}|x-y|^q \pi(dx, dy)\Bigg)^{\frac{1}{q}}, \quad  \text{for $\mu, \nu \in \mathcal{P}_p(\mathbb{R}^d)$},
		$$
    		where $\Pi(\mu, \nu)$ is the set of probability measures on $\mathbb{R}^d\times\mathbb{R}^d$ satisfying $\pi(\cdot \times \mathbb R^d) = \mu $ and $\pi(\mathbb{R}^d \times \cdot) = \nu$. We present a simple compactness criterion for some subsets of the Wasserstein space, which is crucial for our uniqueness result. 

\begin{lemma}
    Let $K>0$ and $p_2 > p_1 \geq 1$. The set $V^{p_1,p_2}_K:=\{ \mu\in\mathcal{P}_{p_1}(\mathbb{R}^d) : M_{p_2}(\mu) \leq K\}=\{ \mu\in\mathcal{P}_{p_2}(\mathbb{R}^d) : M_{p_2}(\mu) \leq K\}$ is compact in $(\mathcal{P}_{p}(\mathbb{R}^d),\mathcal{W}_{p_1})$ for any $p\in[p_1,p_2]$.
    \label{lem. compact of V^p_K}
\end{lemma}
\begin{proof}
    Consider the closed ball $\overline{B_R}$ centered at $0$ with radius $R>0$ in $\mathbb{R}^d$. Then for any $\mu \in V^{p_1,p_2}_K$ and any random variable $X^\mu$ with law $\mu$, we have
    \begin{align*}
        \mu(\overline{B_R}) = \mathbb{P}(X^\mu \in \overline{B_R}) = 1 - \mathbb{P}(|X^\mu|> R)\geq 1 - K/R^{p_2},
    \end{align*}
    by Markov's inequality. $V^{p_1,p_2}_K$ is thus tight. Therefore, by Prokhorov's theorem, for any sequence $\{\mu_n\}_{n\in\mathbb{N}}\subseteq V^{p_1,p_2}_K$, there exists a subsequence $\{\mu_{n_k}\}_{k\in\mathbb{N}}$ and $\mu^*\in \mathcal{P}_{p_1}(\mathbb R^d)$ such that $\mu_{n_k}\to\mu^*$ weakly in the sense that $\int_{\mathbb{R}^d}f(x)\mu_{n_k}(dx) \to \int_{\mathbb{R}^d}f(x)\mu^*(dx)$ as $k\to \infty$, for any bounded and continuous function $f:\mathbb{R}^d \to\mathbb{R}$. Hence, it holds that $\int_{\mathbb{R}^d}\pig(|x|^{p_2}\wedge M\pig)\mu^*(dx)\leq K$ as $\int_{\mathbb{R}^d}\pig(|x|^{p_2}\wedge M\pig)\mu_{n_k}(dx)\leq K$ for any $k \in \mathbb{N}$. Passing $M \to \infty$, the monotone convergence theorem implies that
    $\int_{\mathbb{R}^d} |x|^{p_2} \mu^*(dx)\leq K$, therefore $\mu^* \in V_K^{p_1,p_2}$. 
    
    We claim that also $\mathcal{W}_{p_1}(\mu_{n_k},\mu) \to 0$. This can be easily seen from the fact (see \cite[Theorem 7.12]{villani_topics_2003}) that $\mathcal{W}_{p_1}(\mu_{n_k},\mu)\to 0$ if and only if
    \begin{enumerate}
        \item[(1).] $\mu_{n_k}\to\mu$ in weak convergence;
        \item[(2).]  $
        \displaystyle\lim_{R\to\infty}\limsup_{k\to\infty}\int_{|x|\geq R}|x|^{p_1} \mu_{n_k}(dx) = 0.$
    \end{enumerate}
    Take a $L^{p_1}$-random variable $X^{n_k}$ with its law $\mathcal{L}(X^{n_k})=\mu_{n_k}$, the second condition in the above follows from the fact that 
    \begin{align*}
        \int_{|x|\geq R}|x|^{p_1} \mu_{n_k}(dx) = \mathbb{E}\Big[|X^{n_k}|^{p_1} \mathds{1}_{\{|X^{n_k}|\geq R\}}\Big]
        \leq&\left\{\mathbb{E}\Big[|X^{n_k}|^{p_2} \Big]\right\}^{p_1/p_2}
        \left\{\mathbb{E}\Big[ \mathds{1}_{\{|X^{n_k}|\geq R\}}\Big]\right\}^{1-p_1/p_2}\\
        \leq&\frac{K}{R^{p_2-p_1}},
    \end{align*}
    hence $V_K^{p_1,p_2}$ is compact in $(\mathcal{P}_{p}(\mathbb{R}^d),\mathcal{W}_{p_1})$ for any $p\in[p_1,p_2]$.
\end{proof}
\begin{remark}
\label{V_lower_semi}
Note that if $\mu_n,\mu\in\mathcal{P}_2(\mathbb{R}^d)$ for any $n \in \mathbb{N}$ such that $\mu_n\to\mu$ in the $\mathcal{W}_1$-topology as $n \to \infty$, then $\mu_n$ weakly converges to $\mu$, and hence with $\mathcal{L}(X_n) = \mu_n$, $\mathcal{L}(X) = \mu$, it holds that $\liminf_n\mathbb{E}f(X_n)\geq\mathbb{E}f(X)$, for all nonnegative and continuous functions $f$. In particular, for any $p \geq 1$ we choose $f(x): = |x|^p$ and see that $\liminf_n  M_p(\mu_n) \geq  M_p(\mu)$, i.e., $ M_p$ is lower semicontinuous with respect to $\mathcal{W}_1$.     
\end{remark}

 In the rest of this subsection, we fix a rich enough probability space $(\Omega, \mathcal{F},\mathbb{P})$ such that it supports all probability laws on $\mathbb{R}^d$, i.e., for any probability law $\mu$ on $\mathbb{R}^d$, there exists $X :\Omega \to \mathbb{R}^d $ such that the law of $X$, denoted by $\mathcal{L}(X)$, is $\mu$. We recall the calculus in the Wasserstein space. For a function $f:[0,T]\times\mathcal{P}_2(\mathbb{R}^d)\to\mathbb{R}$, we adopt the notion of $L$-derivatives (see \cite{lions_annals} for instance) which are recalled as follows:
	\begin{definition}
		The function $f:[0,T]\times\mathcal{P}_2(\mathbb{R}^d)\to\mathbb{R}$ is said to be first-order $L$-differentiable if its lifting $F:[0,T]\times L^2(\Omega,\mathcal{F},\mathbb{P};\mathbb{R}^d)\to\mathbb{R}$; $F(t,\xi) := f(t,\mathcal{L}(\xi)) $ admits a continuous Fr\'echet derivative $D_\xi F :[0,T] \times L^2(\Omega,\mathcal{F},\mathbb{P};\mathbb{R}^d) \to L^2(\Omega,\mathcal{F},\mathbb{P};\mathbb{R}^d)$. 
        \label{def. 1st L d}
	\end{definition}
	
	\begin{remark}
		By \cite[Proposition 5.25]{carmona_probabilistic_2018_vol1}, if $f$ is first-order $L$-differentiable, then it can be shown that there is a measurable function, denoted by $\partial_\mu f(t,\mu)(\cdot) :\mathbb{R}^d \to \mathbb{R}^d$, such that $D_\xi F(t,\xi)=\partial_\mu f(t,\mu)(\xi)$ for any $(t,\mu) \in [0,T] \times \mathcal{P}_2(\mathbb{R}^d)$ and $\xi \in L^2(\Omega,\mathcal{F},\mathbb{P};\mathbb{R}^d)$ with $\mathcal{L}(\xi)=\mu$. We say that $\partial_\mu f:[0,T]\times \mathcal{P}_2(\mathbb{R}^d)\times\mathbb{R}^d \to\mathbb{R}^d$ is the first-order $L$-derivative of $f$.
		\label{169}
	\end{remark}

	\begin{definition}
		The function $f:[0,T]\times\mathcal{P}_2(\mathbb{R}^d)\to\mathbb{R}$ is said to be second-order $L$-differentiable if $f$ is first-order $L$-differentiable and for any $x \in \mathbb{R}^d$, the function $\mu\mapsto\partial_\mu f(t,\mu)(x)$ is $L$-differentiable, i.e., the lifting $F':L^2(\Omega,\mathcal{F},\mathbb{P};\mathbb{R}^d)\to\mathbb{R}^d$ of $\mu\mapsto\partial_\mu f (t,\mu)(x)$ admits a continuous Fr\'echet derivative $D_\xi F' :[0,T] \times L^2(\Omega,\mathcal{F},\mathbb{P};\mathbb{R}^d) \to L^2(\Omega,\mathcal{F},\mathbb{P};\mathbb{R}^{d\times d})$. 
               \label{def. 2nd L d}
	\end{definition}
	\begin{remark}
		Similarly, if $f$ is second-order $L$-differentiable, then there is a measurable function, denoted by $\partial_\mu^2 f(t,\mu)(x,\cdot) :\mathbb{R}^d \to \mathbb{R}^{d\times d}$, such that $D_\xi F'(t,\xi')=\partial_\mu^2 f(t,\mu)(x,\xi')$ for any $(t,\mu,x) \in [0,T] \times \mathcal{P}_2(\mathbb{R}^d)\times \mathbb{R}^d$ and $\xi' \in L^2(\Omega,\mathcal{F},\mathbb{P};\mathbb{R}^d)$ with $\mathcal{L}(\xi')=\mu$. We say that $\partial^2_\mu f:[0,T]\times \mathcal{P}_2(\mathbb{R}^d)\times\mathbb{R}^d\times\mathbb{R}^d \to\mathbb{R}^{d\times d}$ is the second-order $L$-derivative of $f$.
	\end{remark}

	\subsection{Problem Formulation}
	\label{strong_form_2}
	
This article aims to provide the proof of existence and uniqueness of viscosity solution of the HJB equation \begin{align*}
			\left\{\begin{aligned}
				&\partial_t u(t,\mu)+\int_{\mathbb{R}^d}
				\sup_{a\in A} \Bigg\{f(t,x,\mu,a)+b(t,x,\mu,a)\cdot \partial_\mu u(t,\mu)(x) \\
				&\h{100pt}+ \dfrac{1}{2}\text{tr}\Big((\sigma(t,x,a)\big[\sigma(t,x,a)\big]^\top+\sigma^0(t,x)[\sigma^0(t,x)]^\top)\nabla_x\partial_\mu u(t,\mu)(x)\Big)\Bigg\}\mu(dx) \\
				&+ \dfrac{1}{2} \int_{\mathbb{R}^d\times\mathbb{R}^d}\text{tr}\Big[\sigma^0(t,x)[\sigma^0(t,y)]^\top\partial_\mu^2
    u(t,\mu)(x,y)\Big] \, \mu^{\otimes 2}(dx,dy)=0,\h{5pt} \text{for any $(t,\mu) \in [0,T) \times\mathcal{P}_2(\mathbb{R}^d)$};\\
				&u(T,\mu)=\int_{\mathbb{R}^d}g(x,\mu)\mu(dx) \h{10pt}\text{for $\mu\in \mathcal{P}_2(\mathbb{R}^d)$},
			\end{aligned}\right.
		\end{align*} with the coefficient functions 
	\begin{align*}
		b&:[0,T]\times\mathbb{R}^d\times\mathcal{P}_2(\mathbb{R}^d)\times A\to \mathbb{R}^d,&
		\sigma&:[0,T]\times\mathbb{R}^d\times A\to\mathbb{R}^{d\times d},&
		\sigma^0&:[0,T]\times\mathbb{R}^d \to\mathbb{R}^{d\times d},\\
		f&:[0,T]\times\mathbb{R}^d\times\mathcal{P}_2(\mathbb{R}^d)\times A\to\mathbb{R},&
		g&:\mathbb{R}^d\times\mathcal{P}_2(\mathbb{R}^d)\to\mathbb{R},
	\end{align*}
	where $A$ is a compact subset of the Euclidean space $\mathbb{R}^d$ equipped with the distance $d_A$. The coefficient functions satisfy the following assumptions:
	\begin{assumption}\label{assume:A}
The functions $b$, $\sigma$, $\sigma^0$, $f$ and $g$ are jointly continuous in $(t,x, \mu,a)\in ([0,T],|\cdot|) \times (\mathbb{R}^d,|\cdot|)\times (\mathcal{P}_{2}(\mathbb{R}^d),\mathcal{W}_1)\times  (A,|\cdot|)$. Moreover, there exist constants $K\geq 0$, $\rho \in[0,1)$ and $\beta \in (0,1]$ such that for any $a\in A$, $(t,x,\mu)$, $(t',x',\mu')\in[0,T]\times\mathbb{R}^d\times\mathcal{P}_2(\mathbb{R}^d)$, it holds that
			\begin{align*}
			{\it (1)}.\,\,&\big|b(t,x,\mu,a)-b(t',x',\mu',a)\big|+\big|\sigma(t,x,a) - \sigma(t',x',a)\big|+\big|\sigma^0(t,x)-\sigma^0(t',x')\big|\\
				&+\big|f(t,x,\mu,a)-f(t',x',\mu',a)\big|+\big|g(x,\mu) - g(x',\mu')\big|
				\leq K\big[|x-x'|+|t-t'|^\beta+\mathcal{W}_1(\mu,\mu')\big];\\[3mm]
				{\it (2)}.\,\,&\big|b(t,x,\mu,a)\big|+\big|\sigma(t,x,a)\big|+\big|\sigma^0(t,x)\big|\leq K(1+|x|^\rho);\\[3mm]
    {\it (3)}.\,\,&\big|f(t,x,\mu,a)\big|+\big|g(x,\mu)\big|\leq K.
			\end{align*}
	\end{assumption}
 
 \begin{assumption}\label{assume:B}
   For any $a \in A$, the functions $\sigma(\cdot, \cdot, a)$ and $\sigma^0(\cdot,\cdot)$ belong to $C^{1,2}\left([0, T] \times \mathbb{R}^d\right)$. Moreover, there exists a constant $K \geq 0$ such that
		\begin{align*}
			\left|\partial_t \sigma(t, x, a)\right|+\left|\nabla_x \sigma(t, x, a)\right|+\big|\nabla_{xx}^2 \sigma(t, x, a)\big|+\left|\partial_t \sigma^0(t,x)\right|+\left|\nabla_x \sigma^0(t,x)\right|+\left|\nabla_{xx}^2 \sigma^0(t,x)\right| \leq K,
		\end{align*}
		for all $(t, x, a) \in[0, T] \times \mathbb{R}^d \times A$.
	\end{assumption}
  \begin{remark}
     The growth rate condition $\rho \in [0,1)$ in Assumption~\ref{assume:A} will be used for proving the uniqueness of viscosity solution in Theorem~\ref{thm compar}. Specifically, when controlling the error term of the perturbation of solution in \eqref{1271}, we are unable to achieve a sufficient decay rate if $\rho = 1$. Furthermore, we impose the $\mathcal{W}_1$-Lipschitz continuity to obtain better convergence results in the finite-dimensional particle approximation in Lemma \ref{lem estimate of b^i_n,m...}. The $\mathcal{W}_2$-Lipschitz continuity is insufficient, as noted in \cite[Remark 2.5]{cheung2023viscosity}. See also the second part of Remark \ref{rmk. W1 cts for test functions} for a reason on the compactness issue.
 \end{remark}
 Before giving the definition of viscosity solution, we introduce the set of test functions adopted in the article.
\begin{definition}
		The set $C^{1,2}([0,T]\times \mathcal{P}_2(\mathbb{R}^d))$ consists of all $([0,T],|\cdot|) \times (\mathcal{P}_{2}(\mathbb{R}^d),\mathcal{W}_2)$ continuous functions $\varphi:[0,T]\times \mathcal{P}_2(\mathbb{R}^d) \to \mathbb{R}$ satisfying the following:
		\begin{enumerate}[(1).]
			
			\item the derivatives $\partial_t \varphi(t,\mu)$, $\partial_\mu \varphi(t,\mu)(x)$, $\nabla_x \partial_\mu \varphi (t,\mu)(x)$, $\partial_\mu^2 \varphi(t,\mu)(x,x')$ exist and are jointly continuous in $(t, \mu, x, x')\in ([0,T],|\cdot|) \times (\mathcal{P}_{2}(\mathbb{R}^d),\mathcal{W}_2)\times (\mathbb{R}^d,|\cdot|)\times (\mathbb{R}^d,|\cdot|)$;
			\item there is a constant $C_\varphi\geq 0$ such that for any $(t,\mu,x,x') \in [0,T] \times \mathcal{P}_2(\mathbb{R}^d)\times \mathbb{R}^d\times \mathbb{R}^d$, we have
\begin{align*}			
     |\p_t \varphi(t,\mu)|+\pig|\partial_\mu^2 \varphi(t,\mu)(x,x')\pig|+\pig|\nabla_x\partial_\mu \varphi(t,\mu)(x)\pig|\leq C_\varphi\quad \text{and }\quad
     \pig|\partial_\mu \varphi(t,\mu)(x)\pig|\leq C_\varphi\pig(1+|x|\pig).
\end{align*}
		\end{enumerate}
  \label{def. of C1,2}
	\end{definition}
 We now introduce the notion of viscosity solution used in this article. The subsolution part follows the standard Crandall-Lions' definition. However, for equations involving the mean field term, we have to modify the supersolution part.

 \begin{definition}
 \label{vis_def}
		A bounded $([0,T],|\cdot|) \times (\mathcal{P}_{2}(\mathbb{R}^d),\mathcal{W}_1)$-continuous function $u:[0, T] \times \mathcal{P}_2(\mathbb{R}^d) \rightarrow \mathbb{R}$ is called a viscosity subsolution of equation \eqref{HJB_intro} if
		\begin{enumerate}
			\item[(1a).] $u(T, \mu) \leq \displaystyle\int_{\mathbb{R}^d} g(x, \mu) \mu(d x)$, for every $\mu \in \mathcal{P}_2(\mathbb{R}^d) ;$
			\item[(1b).] for any $(t, \mu) \in[0, T) \times \mathcal{P}_2(\mathbb{R}^d)$ and $\varphi \in C^{1,2}([0,T]\times\mathcal{P}_2(\mathbb{R}^d))$ such that $u-\varphi$ attains a maximum with a value of $0$ at $(t, \mu)$ over $[0,T]\times\mathcal{P}_2(\mathbb{R}^d)$, the first equation of \eqref{HJB_intro} holds  with the inequality sign $\geq$ replacing the equality sign and with $\varphi$ replacing  $u$.
		\end{enumerate}
  For any function $h:\mathcal{P}_2(\mathbb{R}^d)\to \mathbb{R}$, we extend it to a function on $\mathcal{P}_2(\mathbb{R}^d\times A)$ by $h(\nu) = h(\mu)$, with $\mu$ being the marginal distribution of $\nu$ on $\mathbb{R}^d$ such that $\nu(\cdot \times A)=\mu(\cdot)$. We define the respective projections $\pi_d : \mathbb{R}^{2d} \to \mathbb{R}^{d}$ and $\pi_{d \times d} : \mathbb{R}^{2d \times 2d} \to \mathbb{R}^{d \times d}$ such that for any $y=(y_1,y_2,\ldots,y_{2d})^\top \in \mathbb{R}^{2d}$ and $M\in\mathbb{R}^{2d \times 2d}$, it holds that $\pi_d(y)=(y_1,y_2,\ldots,y_{d})^\top$ and $(\pi_{d \times d}(M))_{ij}=M_{ij}$ for $i,j=1,\dots, d$. A bounded $([0,T],|\cdot|) \times (\mathcal{P}_{2}(\mathbb{R}^d),\mathcal{W}_1)$-continuous function $u:[0, T] \times \mathcal{P}_2(\mathbb{R}^d) \rightarrow \mathbb{R}$ is called a viscosity supersolution of equation \eqref{HJB_intro} if
		\begin{enumerate}
			\item[(2a).] $u(T, \mu) \geq \displaystyle\int_{\mathbb{R}^d} g(x, \mu) \mu(d x)$, for every $\mu \in \mathcal{P}_2(\mathbb{R}^d) ;$
			\item[(2b).] for any $s\in (0,T]$, $(t,\nu)\in [0,s)\times\mathcal{P}_2(\mathbb{R}^d\times A)$ and $\varphi \in C^{1,2}([0,s]\times\mathcal{P}_2(\mathbb{R}^d\times A))$ such that $u-\varphi$ attains a minimum with a value of $0$ at $(t,\nu)$ over $[0,s]\times\mathcal{P}_2(\mathbb{R}^d\times A)$, the following inequality holds:
\begin{align}
 \label{HJB_supersolution}
				&\partial_t \varphi(t,\nu)+\int_{\mathbb{R}^d\times A}
				 \Bigg\{f(t,x,\mu,a)+b(t,x,\mu,a)\cdot \partial_\mu \varphi(t,\nu)(x,a) \nonumber\\
				&\h{40pt}+ \dfrac{1}{2}\textup{tr}\Big((\sigma(t,x,a)\big[\sigma(t,x,a)\big]^\top+\sigma^0(t,x)[\sigma^0(t,x)]^\top)\nabla_x\partial_\mu \varphi(t,\mu)(x,a)\Big)\Bigg\}\nu(dx,da)\nonumber\\
&+ \dfrac{1}{2} \int_{\mathbb{R}^d\times A\times\mathbb{R}^d\times A}\textup{tr}\Big[\sigma^0(t,x)[\sigma^0(t,y)]^\top\partial_\mu^2
    \varphi(t,\mu)(x,a,y,\alpha)\Big] \, \nu^{\otimes 2}(dx,da,dy,d\alpha)\nonumber\\
				&\leq 0, 
		\end{align}
    where $\mu$ is the marginal distribution of $\nu$ on $\mathbb{R}^d$, $u(t,\nu)$ is defined as $u(t,\mu)$, $\partial_\mu \varphi(t,\nu)(\cdot,\cdot):=\pi_d(\partial_\nu \varphi(t,\nu)(\cdot,\cdot))$, and $\partial_\mu^2 \varphi(t,\nu)(\cdot,\cdot):=\pi_{d\times d}(\partial_\nu \varphi(t,\nu)(\cdot,\cdot))$. Here $\partial_\nu \varphi(t,\nu)(\cdot,\cdot):\mathbb{R}^d\times A\to\mathbb{R}^{2d}$ and $\partial^2_\nu \varphi(t,\nu)(\cdot,\cdot,\cdot,\cdot):\mathbb{R}^d\times A\times\mathbb{R}^d\times A\to\mathbb{R}^{2d\times 2d}$ are similarly defined as in Definitions \ref{def. 1st L d}, \ref{def. 2nd L d}, respectively, but over $\mathbb{R}^d \times A$.
\end{enumerate}
A function $u:[0, T] \times \mathcal{P}_2(\mathbb{R}^d) \rightarrow \mathbb{R}$ is called a viscosity solution of \eqref{HJB_intro} if it is both a viscosity subsolution and a viscosity supersolution.
		\label{def. of vis sol}
	\end{definition}
\begin{remark}\label{rmk modified supersol}
     As noted in \cite[Remark 6.1]{cheung2023viscosity}, conditions (2a)-(2b) in Definition \ref{def. of vis sol} provide sufficient criteria for the standard Crandall-Lions' definition of a supersolution. This formulation is adopted because, in proving the comparison theorem for viscosity supersolutions, there are technical difficulties that the Crandall-Lions' definition alone cannot resolve. For further details, see \cite[Remark 6.2]{cheung2023viscosity}.  It is important to note that this modification is required only for the comparison theorem; the value function can still be shown to be a viscosity solution to the HJB equation \eqref{HJB_intro} under the standard Crandall-Lions' definition, using essentially the same proof as in Theorem \ref{existence_label}.
\end{remark}
\begin{remark}
There are two key reasons for adopting the $\mathcal{W}_1$-Lipschitz continuity in Definition \ref{def. of vis sol}. First, as highlighted in \cite[Remark 2.5]{cheung2023viscosity}, to apply the finite-dimensional particle approximation of the value function $v$ in the comparison theorem, it is inevitable to assume the $\mathcal{W}_1$-Lipschitz continuity of the coefficient functions in Assumption \ref{assume:A}. This assumption ensures that the value function is also $\mathcal{W}_1$-Lipschitz continuity (see Lemma \ref{v_W1_continuous}). Hence, it is natural to seek candidate solutions among $\mathcal{W}_1$-continuous functions. Second, the $\mathcal{W}_1$-continuity provides a suitable compactness property, replacing the smooth variational principle to guarantee the existence of extrema for functions in the Wasserstein space.  To effectively apply the smooth variational principle, a gauge function that matches the order of the HJB equations is required. However, constructing a gauge function that is smooth up to second-order is challenging, which complicates the extension of viscosity solutions to fully nonlinear second-order HJB equations. Fortunately, this compactness property allows us to bypass the need for a second-order \( L \)-differentiable gauge function.
\label{rmk. W1 cts for test functions}
\end{remark}
\subsection{Associated Mean Field Control Problems}\label{sec. Associated Mean Field Control Problems}
 We consider a probability space \( (\Omega, \mathcal{F}, \mathbb{P}) \) structured as \( (\Omega^0 \times \Omega^1, \mathcal{F}^0 \otimes \mathcal{F}^1, \mathbb{P}^0 \otimes \mathbb{P}^1) \). The space \( (\Omega^0, \mathcal{F}^0, \mathbb{P}^0) \) supports a $d$-dimensional Brownian motion \( W^0 \), which represents the common noise. The space \( (\Omega^1, \mathcal{F}^1, \mathbb{P}^1) \) is of the form \( (\tilde{\Omega}^1 \times \hat{\Omega}^1, \mathcal{G} \otimes \hat{\mathcal{F}}^1, \tilde{\mathbb{P}}^1 \otimes \hat{\mathbb{P}}^1) \). On \( (\hat{\Omega}^1, \hat{\mathcal{F}}^1, \hat{\mathbb{P}}^1) \), there is a $d$-dimensional Brownian motion \( W \), which represents the idiosyncratic noise. Meanwhile, \( (\tilde{\Omega}^1, \mathcal{G}, \tilde{\mathbb{P}}^1) \) supports the initial random variables. We further assume, without loss of generality, that the probability space \( (\tilde{\Omega}^1, \mathcal{G}, \tilde{\mathbb{P}}^1) \) is rich enough to support any probability law on \( \mathbb{R}^d \). Specifically, for any probability measure \( \mu \) on \( \mathbb{R}^d \), there exists a random variable \( X(\omega) :\tilde{\Omega}^1 \to \mathbb{R}^d \) such that the distribution of \( X \), denoted by \( \mathcal{L}(X) \), is \( \mu \).

We represent \( \omega \in \Omega \) as \( \omega = (\omega^0, \omega^1) \), where the Brownian motions are given by \( W(\omega) = W(\omega^1) \) and \( W^0(\omega) = W^0(\omega^0) \). Let \( \mathbb{E} \) denote the expectation under \( \mathbb{P} \), while \( \mathbb{E}^0 \) and \( \mathbb{E}^1 \) represent the expectations under \( \mathbb{P}^0 \) and \( \mathbb{P}^1 \), respectively. We define the filtrations as follows:
\begin{itemize}
    \item \( \mathbb{F} = (\mathcal{F}_s)_{s \geq 0} := (\sigma(W^0_r)_{0 \leq r \leq s} \vee \sigma(W_r)_{0 \leq r \leq s} \vee \mathcal{G})_{s \geq 0} \);
    \item \( \mathbb{F}^t = (\mathcal{F}_s^t)_{s \geq 0} := (\sigma(W^0_{r \vee t} - W_t^0)_{0 \leq r \leq s} \vee \sigma(W_{r \vee t} - W_t)_{0 \leq r \leq s} \vee \mathcal{G})_{s \geq 0} \);
    \item \( \mathbb{F}^{W^0} = (\mathcal{F}_s^{W^0})_{s \geq 0} := (\sigma(W_r^0))_{0 \leq r \leq s} \);
    \item \( \mathbb{F}^1 = (\mathcal{F}_s^1)_{s \geq 0} := (\sigma(W_s) \vee \mathcal{G})_{0 \leq r \leq s} \).
\end{itemize}
For simplicity, we assume they are \( \mathbb{P} \)-complete.

Recall that \( A \subset \mathbb{R}^d\) is a compact subset equipped with the distance \( d_A \). For \( t > 0 \), let \( \mathcal{A} \) (resp. \( \mathcal{A}_t \)) denote the set of \( \mathbb{F} \)-progressively measurable (resp. \( \mathbb{F}^t \)-progressively measurable) processes on \( \Omega \) that take values in \( A \). The set \( \mathcal{A} \) (resp. \( \mathcal{A}_t \)) is a separable metric space with the Krylov distance defined as \( \Delta(\alpha, \beta) := \mathbb{E}\left[\int_0^T d_A(\alpha_r, \beta_r) \, dr\right] \) (resp. \( \Delta_t(\alpha, \beta) := \mathbb{E}\left[\int_t^T d_A(\alpha_r, \beta_r) \, dr\right] \)). We denote by \( \mathcal{B}_{\mathcal{A}} \) (resp. \( \mathcal{B}_{\mathcal{A}_t} \)) the Borel \( \sigma \)-algebra on \( \mathcal{A} \) (resp. \( \mathcal{A}_t \)). Without loss of generality, let \( (\Omega^0, \mathcal{F}^0, \mathbb{P}^0) \) be the canonical space, i.e., \( \Omega^0 = C(\mathbb{R}_{+}, \mathbb{R}^d) \), the space of continuous functions from \( \mathbb{R}^+ \) to \( \mathbb{R}^d \).

We now introduce a mean field control problem related to the HJB equation \eqref{HJB_intro}. For every $t\in[0,T]$, $\xi\in L^2(\Omega^1,\mathcal{F}_t^1,\mathbb{P}^1;\mathbb{R}^d)$ and $\alpha \in \mathcal{A}_t$, we consider the solution $X^{t,\xi,\alpha}$ of the following state dynamics:
	\begin{align}
		\label{dynamics}
		X_s &= \xi + \int_{t}^s b(r,X_{r},\mathbb{P}_{X_{r}}^{W^0},\alpha_r)dr + \int_t^s \sigma(r,X_{r},\alpha_r)dW_r + \int_t^s \sigma^0(r,X_{r})dW^0_r,\h{5pt} \text{ for $s \in [t,T]$,}
	\end{align}
	where $\mathbb{P}_{X_{r}}^{W^0}$ denotes the conditional distribution of $X_{r}$ given $W^0$. We are subject to the cost functional:
	\begin{align}
		J(t,\xi,\alpha) := \mathbb{E}\Bigg[\int_t^T f\Big(s,X_s^{t,\xi,\alpha},\mathbb{P}^{W^0}_{X_s^{t,\xi,\alpha}},\alpha_s\Big)ds + g\Big(X_T^{t,\xi,\alpha},\mathbb{P}_{X_T^{t,\xi,\alpha}}^{W^0}\Big)\Bigg].
	\end{align}
    We define the value function $V$ to be
	\begin{align}
		V(t,\xi):= \sup_{\alpha\in \mathcal{A}_t}J(t,\xi,\alpha),\quad\text{for any }(t,\xi)\in [0,T]\times L^2(\Omega^1,\mathcal{F}_t^1,\mathbb{P}^1;\mathbb{R}^d).
		\label{def. value function with xi}
	\end{align}
 As in \cite[Appendix B]{cosso2023optimal} and \cite[Proposition 3.3]{cheung2023viscosity}, it can be shown that the value function is law invariant under Assumption \ref{assume:A}, i.e., for every $t \in [0,T]$ and $\xi$, $\eta\in L^2({\Omega}^1,\mathcal{F}^1_t,\mathbb{P}^1;\mathbb{R}^d)$, with $\mathcal{L}(\xi) = \mathcal{L}(\eta)$, it holds that
$
			V(t,\xi) = V(t,\eta)$. 
Therefore we can define a function $v(t,\mu):[0,T]\times\mathcal{P}_2(\mathbb{R}^d) \mapsto \mathbb{R}$ such that for any $t\in [0,T]$ and $\mu \in \mathcal{P}_2(\mathbb{R}^d)$, it holds that
	\begin{align}
		\label{value_function_after_law_invariance}
		v(t,\mu) := V(t,\xi),
	\end{align}
     for any $\xi\in L^2(\Omega^1,\mathcal{F}_t^1,\mathbb{P}^1;\mathbb{R}^d)$ such that $\mathcal{L}(\xi)=\mu$, where $V(t,\xi)$ is defined in \eqref{def. value function with xi}. The main goal of this article is to prove that the value function $v(t,\mu)$ is the unique viscosity solution to the HJB equation \eqref{HJB_intro}, under Definition \ref{def. of vis sol}. To address this, we need some preliminary results. We first give the regularity of the solution of the SDE \eqref{dynamics} and the value function $v$. Proofs are standard and therefore omitted, and readers are referred to \cite{yong1999stochastic}.
		\begin{prop}
  \label{prop. property of X}
		Suppose that Assumption \ref{assume:A} holds. For every $t\in [0,T]$, $\xi\in L^2(\Omega^1,\mathcal{F}^1_t,\mathbb{P}^1;\mathbb{R}^d)$ and $\alpha\in\mathcal{A}_t$, there exists a unique (up to $\mathbb{P}$-indistinguishability) continuous $\mathbb{F}$-progressively measurable solution $X^{t,\xi,\alpha} = (X^{t,\xi,\alpha}_s)_{s\in [t,T]}$ of equation (\ref{dynamics}). Moreover, there is a constant $C$ depending only on $K$, $T$, $d$ such that
			\begin{align*}
&\mathbb{E}\Big[\sup_{s\in[t,T]}|X_s^{t,\xi,\alpha}|^2\Big]\leq C \pig(1+\mathbb{E}|\xi|^2\pigr),\quad
\mathbb{E}\Big[\sup_{s\in [t,T]}|X_{s}^{t,\xi,\alpha}-X_{s}^{t,\xi',\alpha}|^2\Big]\leq C \mathbb{E}\pig(|\xi-\xi'|^2\pigr)\quad\text{and}\\
&\mathbb{E}\Big[\sup_{s\in[t,t+h]}|X^{t,\xi,\alpha}_s-\xi|^2\Big]\leq C h,
		\end{align*}
for any $t\in [0,T]$, $h \in [0,T-t]$, $\xi,\xi'\in L^2(\Omega^1,\mathcal{F}^1_t,\mathbb{P}^1;\mathbb{R}^d)$ and $\alpha\in\mathcal{A}_t$. 
	\end{prop}
	\begin{prop}
		\label{property_V}
		Suppose that Assumption \ref{assume:A} holds. There exists a constant $C\geq 0$ such that for any $t,t'\in[0,T]$, $\mu$, $\mu'\in \mathcal{P}_2(\mathbb{R}^d)$,
			\begin{align}
				\label{difference_V}
				|v(t,\mu)|\leq (1+T)K\quad \text{and}\quad|v(t,\mu)-v(t',\mu')| \leq C\Big[\mathcal{W}_2(\mu,\mu')
				+|t-t'|^{1/2}\Big].
			\end{align}
			The constant $C$ depends on $d$, $K$, $T$ and independent of $t$, $t'$, $\mu$, $\mu'$.
	\end{prop}
 We then state the dynamic programming principle which is essential to verify that the value function $v$ is the viscosity solution to the HJB equation \eqref{HJB_intro}.  The proof is standard and can be obtained following arguments similar to those in \cite{Pham_Wei_Dynamic_Programming}.
	\begin{theorem}
		[Dynamic Programming Principle] 
		\label{dpp_thm}  
			Suppose that Assumption \ref{assume:A} holds. The value function $v$ satisfies the dynamic programming principle: for every $t\in[0,T]$ and $\mu \in \mathcal{P}_2(\mathbb{R}^d)$, it holds that
		\begin{align*}
			v(t,\mu) = & \sup_{\alpha\in\mathcal{A}_t}\sup_{s\in[t,T]}\Bigg\{\mathbb{E}\Bigg[\int_t^s f\Big(r,X_r^{t,\xi,\alpha},\mathbb{P}^{W^0}_{X_r^{t,\xi,\alpha}},\alpha_r\Big)dr+v\pig(s,\mathbb{P}_{X^{t,\xi,\alpha}_s}^{W^0}\pig)\Bigg]\Bigg\}\\
   = & \sup_{\alpha\in\mathcal{A}_t}\inf_{s\in[t,T]}\Bigg\{\mathbb{E}\Bigg[\int_t^s f\Big(r,X_r^{t,\xi,\alpha},\mathbb{P}^{W^0}_{X_r^{t,\xi,\alpha}},\alpha_r\Big)dr+v\pig(s,\mathbb{P}_{X^{t,\xi,\alpha}_s}^{W^0}\pig)\Bigg]\Bigg\},
		\end{align*}
		for any $\xi\in L^2(\Omega^1,\mathcal{F}_t^1,\mathbb{P}^1;\mathbb{R}^d)$ such that $\mathcal{L}(\xi) = \mu$.
	\end{theorem}

  We recall the following version of It\^o's formula for functions in $C^{1,2}([0,T]\times\mathcal{P}_2(\mathbb{R}^d))$, as stated in \cite[Theorem 4.14]{carmona_probabilistic_2018_vol2}.
 \begin{theorem}
	\label{standard_ito}
		Let $(\widetilde{b}_t)_{t\geq 0}$, $(\widetilde{\sigma}_t)_{t\geq 0}$ and $(\widetilde{\sigma}_t^0)_{t\geq 0}$ be progressively measurable processes with respect to $\mathbb{F}$, with values in $\mathbb{R}^d$, $\mathbb{R}^{d\times d}$ and $\mathbb{R}^{d\times d}$ respectively, such that for any finite horizon $T>0$, 
  \begin{align*}
      \mathbb{E}\left[\int_0^T \big(|\widetilde{b}_t|^2 + |\widetilde{\sigma}_t|^4 + |\widetilde{\sigma}^0_t|^4\big)dt\right]<\infty.
  \end{align*}
  Consider the following $\mathbb{R}^d$-valued It\^o process:
		\begin{align*}
			X_t = \xi+ \int^t_0 \widetilde{b}_s ds + \int^t_0\widetilde{\sigma}_s dW_s+\int^t_0\widetilde{\sigma}_s^0dW_s^0,\quad \text{for }t\in[0,T],
		\end{align*}
        where $\xi\in L^2(\Omega,\mathcal{F},\mathbb{P};\mathbb{R}^d)$. It holds $\mathbb{P}^0$-a.s. that for $\varphi\in C^{1,2}([0,T]\times\mathcal{P}_2(\mathbb{R}^d))$,
        \begin{align*}
			&\h{-10pt}\varphi(t,\mathbb{P}_{X_t}^{W^0}) \\=\,& \varphi(0,\mathbb{P}_{X_0}^{W^0})  + \int_0^t \partial_t \varphi(s,\mathbb{P}_{X_s}^{W^0})ds + \int_0^t\mathbb{E}^1\big[\partial_\mu \varphi(s,\mathbb{P}_{X_s}^{W^0})(X_s)\cdot \widetilde{b}_s\big]ds\\
   &+\int_0^t \mathbb{E}^1\big[\widetilde{\sigma}_s^{0;\top} \partial_\mu \varphi(s,\mathbb{P}_{X_s}^{W^0})(X_s)\big]\cdot dW_s^0
			+\frac{1}{2}\int_0^t\mathbb{E}^1\Big\{\operatorname{tr}\pig[\nabla_x\partial_\mu \varphi(s,\mathbb{P}_{X_s}^{W^0})(X_s)\widetilde{\sigma}_s\widetilde{\sigma}_s^\top\pig]\Big\}ds\\
   &+\frac{1}{2}\int_0^t\mathbb{E}^1\Big\{\operatorname{tr}\pig[\nabla_x\partial_\mu \varphi(s,\mathbb{P}_{X_s}^{W^0})(X_s)\widetilde{\sigma}_s^0\widetilde{\sigma}_s^{0;\top}\pig]\Big\}ds+\frac{1}{2}\int_0^t \mathbb{E}^1\widecheck{\mathbb{E}}^1\Big\{\operatorname{tr}\Big[\partial_\mu^2 \varphi(s,\mathbb{P}_{X_s}^{W^0})(X_s,\widecheck{X}_s)\widetilde{\sigma}_s^0\pig(\widecheck{\widetilde{\sigma}_s^{0}}\pigr)^\top\Big]\Big\}ds,
		\end{align*}
  where $\widecheck{\mathbb{E}}^1$ is the expectation under $(\widecheck{\Omega}^1,\widecheck{\mathcal{F}}^1,\widecheck{\mathbb{P}}^1)$ which is a copy of $(\Omega^1,\mathcal{F}^1,\mathbb{P}^1)$; and $(\widecheck{X}_t)_{t\geq 0}$, $(\widecheck{\widetilde{b}}_t)_{t\geq 0}$, $(\widecheck{\widetilde{\sigma}}_t)_{t\geq 0}$, $(\widecheck{\widetilde{\sigma}_t^0})_{t\geq 0}$ are copies of $(X_t)_{t\geq 0}$, $(\widetilde{b}_t)_{t\geq 0}$, $(\widetilde{\sigma}_t)_{t\geq 0}$, $(\widetilde{\sigma}^0_t)_{t\geq 0}$ on $(\widecheck{\Omega}^1,\widecheck{\mathcal{F}}^1,\widecheck{\mathbb{P}}^1)$.
	\end{theorem}

\section{Smooth Finite-dimensional Approximations of Value Function}
\label{approximation}
In this section, we construct $C^{1,2}([0,T]\times \mathcal{P}_2(\mathbb{R}^d))$ approximations of the value function $v$. As similar results have been proved in \cite{cosso_master_2022,cheung2023viscosity} and the notations are quite heavy, we only provide main ideas of the construction in this section and postpone the details to Appendix. The main idea is to first add a small term in the volatility to ensure that the  HJB equation \eqref{HJB_intro} is non-degenerate in the second-order term.  Then, we mollify the coefficient functions and use the empirical measure to approximate the measure arguments. It ensures that the approximated coefficient functions are smooth and on a finite-dimensional domain. These procedures allow us to construct a regular enough test function from the approximations of value function $v$.
\subsection{Approximation by Non-degenerate Problem}
	Fix a complete probability space $(\check{\Omega},\check{\mathcal{F}},\check{\mathbb{P}})$, also of the form $(\check{\Omega}^0\times\check{\Omega}^1,\check{\mathcal{F}}^0\otimes \check{\mathcal{F}}^1,\check{\mathbb{P}}^0\otimes \check{\mathbb{P}}^1)$. The space $(\check{\Omega}^0,\check{\mathcal{F}}^0,\check{\mathbb{P}}^0)$ supports a $d$-dimensional Brownian motion $\check{W}^0$. For $(\check{\Omega}^1,\check{\mathcal{F}}^1,\check{\mathbb{P}}^1)$, it is of the form $(\check{\tilde{\Omega}}^1\times \check{\hat{\Omega}}^1,\check{\mathcal{G}}\otimes\check{\hat{\mathcal{F}}}^1,\check{\tilde{\mathbb{P}}}^1\otimes\check{\hat{\mathbb{P}}}^1)$. On $(\check{\hat{\Omega}}^1,\check{\hat{\mathcal{F}}}^1,\check{\hat{\mathbb{P}}}^1)$, there lives $d$-dimensional Brownian motions $\check{W}$ and $\check{B}$. The space $(\check{\tilde{\Omega}}^1,\check{\mathcal{G}},\check{\tilde{\mathbb{P}}}^1)$ is where the initial random variables live. We assume that $(\check{\tilde{\Omega}}^1,\check{\mathcal{G}},\check{\tilde{\mathbb{P}}}^1)$ is rich enough to support all probability laws on $\mathbb{R}^d$, i.e., for any probability law $\mu$ on $\mathbb{R}^d$, there exists $X:\check{\tilde{\Omega}}^1\to\mathbb{R}^d$ such that $\mathcal{L}(X) = \mu$. The expectation $\mathbb{E}$ in this subsection is taken with respect to $\check{\mathbb{P}}=\check{\mathbb{P}}^0\otimes \check{\mathbb{P}}^1$.
 
	Set $\check{\mathbb{F}} = (\check{\mathcal{F}}_s)_{s\geq 0}:=\left(\sigma(\check{W}^0_{r})_{0\leq r\leq s}\vee\sigma(\check{W}_{r})_{0\leq r\leq s}\vee\sigma(\check{B}_{r})_{0\leq r\leq s}\vee\check{\mathcal{G}}\right)_{s\geq 0}$, $\check{\mathbb{F}}^t = (\mathcal{F}_s^t)_{s\geq 0}:=\big(\sigma(\check{W}_r^0-\check{W}_t^0)_{0\leq r\leq s}\vee\sigma(\check{W}_{r\vee t}-\check{W}_t)_{0\leq r\leq s}\vee\sigma(\check{B}_{r\vee t}-\check{B}_t)_{0\leq r\leq s}\vee\check{\mathcal{G}}\big)_{s\geq 0}$. Without loss of generality, we assume that they are $\check{\mathbb{P}}$-complete. Let $t>0$, denote by $\check{\mathcal{A}}$ (resp. $\check{\mathcal{A}}_t$) the set of $\check{\mathbb{F}}$-progressively measurable processes (resp. $\check{\mathbb{F}}^{t}$-progressively measurable processes) valued in $A$.

	Letting $\varepsilon>0$, $t\in [0,T)$, $\check{\xi}\in L^2(\check{\Omega}^1,\check{\mathcal{F}}^1,\check{\mathbb{P}}^1;\mathbb{R}^d)$ and $\check{\alpha}\in \check{\mathcal{A}}$, we consider the unique solution $\check{X}^{\varepsilon,t,\check{\xi},\check{\alpha}}=(\check{X}^{\varepsilon,t,\check{\xi},\check{\alpha}}_s)_{s\in [t,T]}$ of  the perturbed equation:
	\begin{align}
		X_s =\,& \check{\xi} + \int_{t}^s b(r,X_{r},\mathbb{P}_{X_{r}}^{\check{W}^0},\check{\alpha}_r)dr + \int_t^s \sigma(r,X_{r},\check{\alpha}_r)d\check{W}_r + \int_t^s \sigma^0(r,X_r)d\check{W}^0_r\nonumber\\
		&+\varepsilon(\check{B}_s - \check{B}_t).
		\label{eq. perturbed, mfc}
	\end{align}
	For any $t \in [0,T]$ and $\check{\xi}\in L^2(\check{\Omega}^1,\check{\mathcal{F}}^1,\check{\mathbb{P}}^1;\mathbb{R}^d)$, we  consider the value function:
	\begin{align}
		V_{\varepsilon}(t,\check{\xi})
		=\sup_{\check{\alpha}\in \check{\mathcal{A}}_t}J_{\varepsilon}(t,\check{\xi},\check{\alpha})
		:=\sup_{\check{\alpha}\in \check{\mathcal{A}}_t}
		\mathbb{E}\Bigg[\int_t^T f\Big(s,\check{X}^{\varepsilon,t,\check{\xi},\check{\alpha}}_s,
		\mathbb{P}^{\check{W}^0}_{\check{X}^{\varepsilon,t,\check{\xi},\check{\alpha}}_s},\check{\alpha}_s\Big)ds + g\Big(\check{X}^{\varepsilon,t,\check{\xi},\check{\alpha}}_T,
		\mathbb{P}^{\check{W}^0}_{\check{X}^{\varepsilon,t,\check{\xi},\check{\alpha}}_T}\Big)\Bigg].
		\label{1368}
	\end{align}
	By the law invariance property similar to \cite[Appendix B]{cosso2023optimal} and \cite[Proposition 3.3]{cheung2023viscosity}, we can define a function $v_{\varepsilon}(t,\mu):[0,T]\times\mathcal{P}_2(\mathbb{R}^d) \to \mathbb{R}$ such that
	\begin{align}
		v_{\varepsilon}(t,\mu) := V_{\varepsilon}(t,\check{\xi}),
		\label{def. v_e=V_e mfc}
	\end{align}
	for any $\check{\xi}\in L^2(\check{\Omega}^1,\check{\mathcal{F}}^1,\check{\mathbb{P}}^1;\mathbb{R}^d)$ such that $\mathcal{L}(\check{\xi})=\mu$.
	\begin{lemma}
	Suppose that Assumption \ref{assume:A} holds. There exists a constant $C_5=C_5(d,K,T)>0$ such that for any $\varepsilon \geq 0$ and $(t,\mu) \in [0,T] \times \mathcal{P}_2(\mathbb{R}^d)$, it holds that $|v_{\varepsilon}(t,\mu)-v_0(t,\mu)| \leq C_5\varepsilon.$
		\label{lem. |v_e-v_0|<C_5 e}
	\end{lemma}
 \begin{proof}
It is standard to see that $\mathbb{E}\pig[|\check{X}^{\varepsilon,t,\check{\xi},\check{\alpha}}_s
-\check{X}^{0,t,\check{\xi},\check{\alpha}}_s|^2\pig]\leq C \varepsilon^2$ for any $s\in [t,T]$ by using Assumption \ref{assume:A} and equation \eqref{eq. perturbed, mfc}. The result follows by 
\begin{align*}
\left| v_{\varepsilon}(t, \mu) - v_0(t, \mu) \right|^2 
\leq 2K^2 \mathbb{E} \left[ T\int_{t}^{T} \left| \check{X}^{\varepsilon,t,\check{\xi},\check{\alpha}}_s - \check{X}^{0,t,\check{\xi},\check{\alpha}}_s \right|^2 ds + \left| \check{X}^{\varepsilon,t,\check{\xi},\check{\alpha}}_T - \check{X}^{0,t,\check{\xi},\check{\alpha}}_T \right|^2 \right]\leq C_5^2 \varepsilon^2.
\end{align*}
 \end{proof}

\vspace{4pt}
\subsection{Smooth Finite-dimensional Approximation of Coefficient Functions}
\label{Phi_defined}

In this section, we illustrate the finite-dimensional approximation. Consider a complete probability space $(\overline{\Omega},\overline{\mathcal{F}},\overline{\mathbb{P}})$, which is also of the form $(\overline{\Omega}^0\times\overline{\Omega}^1,\overline{\mathcal{F}}^0\otimes \overline{\mathcal{F}}^1,\overline{\mathbb{P}}^0\otimes \overline{\mathbb{P}}^1)$. The space $(\overline{\Omega}^0,\overline{\mathcal{F}}^0,\overline{\mathbb{P}}^0)$ supports a $d$-dimensional Brownian motion $\overline{W}^0$. For $(\overline{\Omega}^1,\overline{\mathcal{F}}^1,\overline{\mathbb{P}}^1)$, it is of the form $(\overline{\tilde{\Omega}}^1\times \overline{\hat{\Omega}}^1,\overline{\mathcal{G}}\otimes\overline{\hat{\mathcal{F}}}^1,\overline{\tilde{\mathbb{P}}}^1\otimes\overline{\hat{\mathbb{P}}}^1)$. Let $n \in \mathbb{N}$. There lives $d$-dimensional Brownian motions $\overline{W}^1,\ldots,\overline{W}^n$ and $\overline{B}^1,\ldots,\overline{B}^n$ on $(\overline{\hat{\Omega}}^1,\overline{\hat{\mathcal{F}}}^1,\overline{\hat{\mathbb{P}}}^1)$. We require $\overline{W}^{1},\ldots,\overline{W}^n,\overline{B}^1,\ldots,\overline{B}^n$ to be mutually independent. The space $(\overline{\tilde{\Omega}}^1,\overline{\mathcal{G}},\overline{\tilde{\mathbb{P}}}^1)$ is where the initial random variables live. We assume that $(\overline{\tilde{\Omega}}^1,\overline{\mathcal{G}},\overline{\tilde{\mathbb{P}}}^1)$ is rich enough to support all probability laws on $\mathbb{R}^d$, i.e., for any probability law $\mu$ on $\mathbb{R}^d$, there exists $X:\overline{\tilde{\Omega}}^1\to\mathbb{R}^d$ such that $\mathcal{L}(X) = \mu$. 

	We define $\overline{\mathbb{F}} = (\overline{\mathcal{F}}_s)_{s\geq 0}:=\left(\sigma(\overline{W}^0_{r})_{0\leq r\leq s}\vee\sigma(\overline{W}_{r}^i)_{0\leq r\leq s,i=1,\ldots,n}\vee\sigma(\overline{B}_{r}^i)_{0\leq r\leq s,i=1,\ldots,n}\vee\overline{\mathcal{G}}\right)_{s\geq 0}$, $\overline{\mathbb{F}}^t = (\mathcal{F}_s^t)_{s\geq 0}:=\left(\sigma(\overline{W}_r^0-\overline{W}^0_t)_{0\leq r\leq s}\vee\sigma(\overline{W}_{r\vee t}^i-\overline{W}^i_t)_{0\leq r\leq s,i=1,\ldots,n}\vee\sigma(\overline{B}^i_{r\vee t}-\overline{B}^i_t)_{0\leq r\leq s,i=1,\ldots,n}\vee\overline{\mathcal{G}}\right)_{s\geq 0}$. Without loss of generality, we assume that they are $\overline{\mathbb{P}}$-complete. Let $t>0$ and denote by $\overline{\mathcal{A}}^n$ (resp. $\overline{\mathcal{A}}_t^n$) the set of $\overline{\mathbb{F}}$-progressively measurable processes (resp. $\overline{\mathbb{F}}^{t}$-progressively measurable processes) $\overline{\alpha} = (\overline{\alpha}^1,\ldots,\overline{\alpha}^n)$ valued in $A^n$.
    
 Now we introduce the smooth approximations of the coefficients. Let $m \in \mathbb{N}$, $\phi: \mathbb{R} \to \mathbb{R}^+$ and $\Phi: \mathbb{R}^d \to \mathbb{R}^+$ be two compactly supported smooth functions satisfying $\int_{\mathbb{R}}\phi(s)ds=1$, $\int_{\mathbb{R}^{d}}\Phi(y)dy=1$ and $\int_{\mathbb{R}^{d}}|y|^\rho\Phi(y)dy\leq C_{\Phi,\rho}$ for some constant $C_{\Phi,\rho} > 0$, where $\rho$ is the parameter given in Assumption \ref{assume:A}. For $n,m \in \mathbb{N}$ and  $i=1,\ldots,n$, we define $b^i_{n,m}:[0,T]\times\mathbb{R}^{dn}\times A \to \mathbb{R}^d$, 
	$f^i_{n,m}:[0,T]\times\mathbb{R}^{dn}\times A \to \mathbb{R}^d$, 
	$g^i_{n,m}:\mathbb{R}^{dn}\to \mathbb{R}^d$ by the smooth approximations of $b$, $f$ and $g$ respectively such that
		\begin{align}
		b^i_{n,m}(t,\overline{x},a)&:=m^{dn+1}\int_{\mathbb{R}^{dn+1}}
		b\left( (t-s)^+\wedge T,x^i-y^i,\dfrac{1}{n}\sum^n_{j=1}\delta_{x^j-y^j},a\right)\phi(ms)\prod^n_{k=1}\Phi(my^k)dy^kds;\label{def. approx of b}\\
		f^i_{n,m}(t,\overline{x},a)&:=m^{dn+1}\int_{\mathbb{R}^{dn+1}}
		f\left( (t-s)^+\wedge T,x^i-y^i,\dfrac{1}{n}\sum^n_{j=1}\delta_{x^j-y^j},a\right)\phi(ms)\prod^n_{k=1}\Phi(my^k)dy^kds;\label{def. approx of f}\\
		g^i_{n,m}(\overline{x})&:=m^{dn}\int_{\mathbb{R}^{dn}}
		g\left( x^i-y^i,\dfrac{1}{n}\sum^n_{j=1}\delta_{x^j-y^j}\right)\prod^n_{k=1}\Phi(my^k)dy^k.
  \label{def. approx of g}
	\end{align}
First, we establish some basic properties of these approximation functions. 
\begin{lemma}
		Suppose that Assumption \ref{assume:A} holds, and define $
			\widehat{\mu}^{n,\overline{x}}:=\dfrac{1}{n}\displaystyle\sum^n_{j=1} \delta_{x^j}$. For any $n,m \in \mathbb{N}$, $i=1,2,\ldots,n$ and $(t,\overline{x},\overline{z},a) \in [0,T] \times \mathbb{R}^{dn}\times \mathbb{R}^{dn}\times A$, we have the following:
		\begin{enumerate}[(1).]
			\item 
				$|f^i_{n,m}(t,\overline{x},a)|
				\vee|g^i_{n,m}(\overline{x})|
				\leq K$ \,and \,$
    |b^i_{n,m}(t,\overline{x},a)|\leq K\pig(1+C_{\Phi,\rho}m^{-\rho}+|x^i|^\rho\pig);$
			\item 
			$|b^i_{n,m}(t,\overline{x},a)
					-b(t,x^i,\widehat{\mu}^{n,\overline{x}},a)|\vee
					|f^i_{n,m}(t,\overline{x},a)
					-f(t,x^i,\widehat{\mu}^{n,\overline{x}},a)|$\\
					$\leq\, Km\displaystyle\int_{\mathbb{R}}\left|t-\pig[T\wedge(t-s)^+\pig]\right|^\beta \phi(ms)ds
					+Km^{dn}\int_{\mathbb{R}^{dn}}
					\left(|y^i|+\dfrac{1}{n}\sum^n_{j=1}|y^j|\right) \prod^n_{k=1}\Phi(my^k)dy^k$;\\ 
				$|g^i_{n,m}(\overline{x})
					-g(x^i,\widehat{\mu}^{n,\overline{x}})|
					\leq Km^{dn}\displaystyle\int_{\mathbb{R}^{dn}}
					\left(|y^i|+\dfrac{1}{n}\displaystyle\sum^n_{j=1}|y^j|\right) \prod^n_{k=1}\Phi(my^k)dy^k$;
			\item $|b^i_{n,m}(t,\overline{x},a)
				-b^i_{n,m}(t,\overline{z},a)|
				\vee
				|f^i_{n,m}(t,\overline{x},a)
				-f^i_{n,m}(t,\overline{z},a)|
				\vee
				|g^i_{n,m}(\overline{x})
				-g^i_{n,m}(\overline{z})|$\\
				$\leq K\left[|x^i-z^i|+\dfrac{1}{n}\displaystyle\sum^n_{j=1}\left|x^j
				-z^j\right|\right]$;
   	\item $\displaystyle\lim_{m\to\infty} b^i_{n,m}(t,\overline{x},a)
				=b(t,x^i,\widehat{\mu}^{n,\overline{x}},a),\h{10pt}
				\lim_{m\to\infty} f^i_{n,m}(t,\overline{x},a)
				=f(t,x^i,\widehat{\mu}^{n,\overline{x}},a)$ and \\
				$\displaystyle\lim_{m\to\infty} g^i_{n,m}(\overline{x})
			=g(x^i,\widehat{\mu}^{n,\overline{x}})$. These limits hold uniformly in $(t, \overline{x}, a) \in [0, T] \times \mathbb{R}^{dn} \times A$.
		\end{enumerate}
		\label{lem estimate of b^i_n,m...}
	\end{lemma}
    For the proof of Lemma \ref{lem estimate of b^i_n,m...}, please refer to Appendix \ref{Appendix_proof_Section_3}. Let $\overline{X}^{m,\varepsilon,t,\overline{x},\overline{\alpha}}_s
	=(\overline{X}^{1,m,\varepsilon,t,\overline{x},\overline{\alpha}}_s,\ldots,\overline{X}^{n,m,\varepsilon,t,\overline{x},\overline{\alpha}}_s)$ be the solution of 
	\begin{align}
		X_s^i =\,& x^i + \int_{t}^s b_{n,m}^i(r,X_{r},\overline{\alpha}_r^i)dr + \int_t^s \sigma(r,X_{r}^i,\overline{\alpha}_r^i)d\overline{W}_r^i + \int_t^s \sigma^0(r,X_{r}^i)d\overline{W}^0_r
  +\varepsilon(\overline{B}^i_s - \overline{B}^i_t),
		\label{eq. state perturbed by e BM and smoothing}
	\end{align}
	where $X_s=(X_s^1,\ldots,X^n_s)$ with $\mathbb{R}^d$-valued processes $X_s^i$ for $i = 1,\dots,n$. We define 
	\begin{align}
		\overline{v}_{\varepsilon,n,m}(t,\overline{x})
		:=&\sup_{\overline{\alpha}\in \overline{\mathcal{A}}^n_t}J_{\varepsilon,n,m}^*(t,\overline{x},\overline{\alpha})\nonumber\\
		:=&\sup_{\overline{\alpha}\in \overline{\mathcal{A}}^n_t} \dfrac{1}{n}\sum^n_{i=1}\mathbb{E}\Bigg[\int_t^T f^i_{n,m}\left(s,\overline{X}^{1,m,\varepsilon,t,\overline{x},\overline{\alpha}}_s,\ldots,
		\overline{X}^{n,m,\varepsilon,t,\overline{x},\overline{\alpha}}_s,
		\overline{\alpha}^i_s\right)ds\nonumber \\
		&\h{100pt}+ g^i_{n,m}\left(\overline{X}^{1,m,\varepsilon,t,\overline{x},\overline{\alpha}}_T,\ldots,
		\overline{X}^{n,m,\varepsilon,t,\overline{x},\overline{\alpha}}_T\right)\Bigg],
		\label{def. tilde v_e,n,m}
	\end{align}
	for any $t \in [0,T]$ and $\overline{x} \in \mathbb{R}^{dn}$. This approximation of the value function is defined on a finite-dimensional domain. Owing to the non-degeneracy and smoothness of the coefficients, the approximation belongs to the class $C^{1,2}([0,T] \times \mathbb{R}^{dn})$ and possesses the following properties:
	\begin{lemma}
 \label{finite_derivative_estimate}
		Suppose that Assumptions \ref{assume:A}-\ref{assume:B} hold. There are positive constants $C_4=C_4(d,K,T)$ and $C_{n,m}$ such that the function $\overline{v}_{\varepsilon,n,m} : [0,T] \times \mathbb{R}^{dn} \longmapsto \mathbb{R}$ satisfies the following:
  \begin{enumerate}[(1).]
      \item For any $(t,\overline{x})$ in $[0,T]\times \mathbb{R}^{dn}$ and $i=1,2,\ldots,n$, we have
      \begin{align}
      \label{ineq.|D v_e,n,m|}
      |\nabla_{x^i}\overline{v}_{\varepsilon,n,m}(t,\overline{x})| \leq \dfrac{C_4}{n};
      \end{align}
      \item Denoting by $\p_{\overline{x}_i\overline{x}_j}^2 v(t,\overline{x}) \in  \mathbb{R}$ the second-order derivative with respect to $\overline{x}_i$ and $\overline{x}_j$, it holds that 
      \begin{align}
        -C_{n,m}\leq\p^2_{\overline{x}_i\overline{x}_j}\overline{v}_{\varepsilon,n,m}(t,\overline{x}) \leq \dfrac{C_{n,m}}{\varepsilon^2},
         \label{ineq.|D^2 v_e,n,m|}
      \end{align}
      for any $(t,\overline{x})$ in $[0,T]\times \mathbb{R}^{dn}$ and every $i,j=1,2,\ldots,dn$;
      \item The function $\overline{v}_{\varepsilon,n,m} \in C^{1,2}([0,T] \times \mathbb{R}^{dn})$ is the unique classical solution of the Bellman equation 
		\begin{equation}
			\left\{
			\begin{aligned}
				&\partial_t u(t,\overline{x})
				+\sup_{\overline{a} \in A^n}\Bigg\{\dfrac{1}{n}\sum^n_{i=1}f^i_{n,m}(t,\overline{x},a^i)
				+\dfrac{1}{2}\sum^n_{i=1}\textup{tr}\left[\Big((\sigma\sigma^\top)(t,x^i,a^i)+\sigma^0\sigma^{0;\top}(t,x^i)+\varepsilon^2I_d\Big)\nabla_{x^ix^i}^2
				u(t,\overline{x})\right]\\ 
				&\h{80pt}+\sum^n_{i=1}\left\langle b^i_{n,m}(t,\overline{x},a^i)
				,\nabla_{x^i}u(t,\overline{x})\right\rangle
				+\frac{1}{2}\sum_{i,j=1,i\neq j}^n\textup{tr}\Big[\sigma^0(t,x^i)\sigma^{0;\top}(t,x^j)\nabla_{x^ix^j}^2u(t,\overline{x})\Big]\Bigg\}\\
				&=0\h{5pt} \text{in $[0,T) \times\mathbb{R}^{dn}$};\\
				&u(T,\overline{x})
				=\dfrac{1}{n}\sum^n_{i=1}g^i_{n,m}(\overline{x})
				\h{5pt} \text{ in $\mathbb{R}^{dn}$},
			\end{aligned}
			\right.   
			\label{eq. bellman bar v_e,n,m}
		\end{equation}
		where $\overline{a}=(a^1,\ldots,a^n)$ with $a^i\in A $ for $i=1,2,\ldots,n$. 
  \end{enumerate}

		\label{lem. classical sol. of smooth approx.}
	\end{lemma}
The proof of Lemma \ref{lem. classical sol. of smooth approx.} is delegated to Appendix \ref{Appendix_proof_Section_3}.

\subsection{Approximation of Value Function}

Recalling the definition of $J_{\varepsilon,n,m}^*$ in \eqref{def. tilde v_e,n,m}, we define
\begin{align*}
		\widetilde{v}_{\varepsilon,n,m}(t,\overline{\mu})
		:=&\sup_{\overline{\alpha}\in \overline{\mathcal{A}}^n_t}J_{\varepsilon,n,m}^*(t,\overline{\xi},\overline{\alpha})
\end{align*}
	for any $t \in [0,T]$ and $\overline{\mu} \in \mathcal{P}_2(\mathbb{R}^{dn})$ such that $\mathcal{L}(\overline{\xi})=\overline{\mu}$. It is clear that $\overline{v}_{\varepsilon,n,m}(t,\overline{x})=\widetilde{v}_{\varepsilon,n,m}(t,\delta_{x^1}\otimes\ldots\otimes\delta_{x^n})$. We define $v_{\varepsilon,n,m}(t,\mu):[0,T] \times \mathcal{P}_2(\mathbb{R}^d)\to\mathbb{R}$ by
\begin{align}
		v_{\varepsilon,n,m}(t,\mu) := \widetilde{v}_{\varepsilon,n,m}(t,\mu\otimes\ldots\otimes\mu).
		\label{def. v_e=V_e mfg}
\end{align}
\begin{theorem}
			Suppose that Assumptions \ref{assume:A}-\ref{assume:B} hold. For every $\varepsilon > 0$,
		$n,m \in \mathbb{N}$, the function $\overline{v}_{\varepsilon,n,m}$ defined in \eqref{def. tilde v_e,n,m} and the function $v_{\varepsilon,n,m}$ defined in \eqref{def. v_e=V_e mfg} satisfy the following:
		\begin{enumerate}[(1).]
			\item 
			For any $(t,\mu) \in [0,T] \times \mathcal{P}_2(\mathbb{R}^d)$, we have
			\begin{align}
				v_{\varepsilon,n,m}(t,\mu)
				=\int_{\mathbb{R}^{dn}}
				\overline{v}_{\varepsilon,n,m}
				(t,x^1,\ldots,x^n)
				\mu(dx^1)\otimes\ldots\otimes\mu(dx^n)\, ;
				\label{eq. v_e,n,m=int bar of v_e,n,m}
			\end{align}
			\item  $v_{\varepsilon,n,m}(t,\mu) \in C^{1,2}([0,T] \times \mathcal{P}_2(\mathbb{R}^d))$;
			\item There is a positive constant $\ell_2=\ell_2(K,T)$ such that for any $(t,\mu) \in [0,T] \times \mathcal{P}_2(\mathbb{R}^d)$, we have $|v_{\varepsilon,n,m}(t,\mu)|\leq \ell_2$;
			\item The function $v_{\varepsilon,n,m}(t,\mu)$ solves the following equation in the classical sense:
			\begin{equation*}
				\left\{
				\begin{aligned}
					&\partial_t u(t,\mu)\\
					&+\int_{\mathbb{R}^{dn}}
					\sup_{\overline{a} \in A^n}\Bigg\{\dfrac{1}{n}\sum^n_{i=1}f^i_{n,m}(t,\overline{x},a^i)
     +\sum^n_{i=1}\left\langle b^i_{n,m}(t,\overline{x},a^i),\nabla_{x^i}\overline{v}_{\varepsilon,n,m}(t,\overline{x})\right\rangle
					\\
					&\h{65pt}
     +\dfrac{1}{2}
			\sum^n_{i=1}\textup{tr}\left[\Big((\sigma\sigma^\top)(t,x^i,a^i)
					+(\sigma^0\sigma^{0;\top})(t,x^i)+\varepsilon^2I_d\Big)\nabla_{x^ix^i}^2
					\overline{v}_{\varepsilon,n,m}
	(t,\overline{x})\right]\\
 		&\h{65pt}+\frac{1}{2}\sum_{i,j=1,i\neq j}^n\textup{tr}\Big[\sigma^0(t,x^i)\sigma^{0;\top}(t,x^j)\nabla_{x^ix^j}^2\overline{v}_{\varepsilon,n,m}(t,\overline{x})\Big]\Bigg\}\mu(dx^1)\otimes\ldots\otimes \mu(dx^n)\\
					&=0\h{5pt} \text{for any $(t,\mu) \in [0,T) \times\mathcal{P}_2(\mathbb{R}^{d})$};\\
					& u(T,\mu)
					=\dfrac{1}{n}\sum^n_{i=1}\int_{\mathbb{R}^{dn}}g^i_{n,m}(\overline{x})\mu(dx^1)\otimes\ldots\otimes\mu(dx^n)
					\h{5pt} \text{for any $\mu \in \mathcal{P}_2(\mathbb{R}^{d})$}.
				\end{aligned}
				\right.    
			\end{equation*}
		\end{enumerate}
		\label{thm v_e,n,m}
	\end{theorem}
	\begin{proof} The proof of item (1) follows from Steps I-II of the proof of \cite[Theorem A.7]{cosso_master_2022}, together with Lemma \ref{lem. classical sol. of smooth approx.}. The function $\overline{v}_{\varepsilon,n,m}(t,\overline{x})$ is in $C^{1,2}([0,T] \times \mathbb{R}^{dn})$ from Lemma \ref{lem. classical sol. of smooth approx.}  and its boundedness is ensured by its definition. Then, item (2) follows by differentiating \eqref{eq. v_e,n,m=int bar of v_e,n,m} and item (3) is ensured by the boundedness of $\overline{v}_{\varepsilon,n,m}(t,\overline{x})$. Item (4) is easily obtained by integrating \eqref{eq. bellman bar v_e,n,m} with $u=\overline{v}_{\varepsilon,n,m}$.
 \end{proof}

	\begin{lemma}
		Suppose that Assumption \ref{assume:A} holds. There exists a positive constant $C_6=C_6(K,T,d)$ such that for any $n,m\in\mathbb{N}$ and $(t,\mu) \in [0,T] \times \mathcal{P}_2(\mathbb{R}^d)$, we have $|v_{\varepsilon,n,m}(t,\mu)-v_{0,n,m}(t,\mu)|\leq C_6 \varepsilon$.
		\label{lem. |v_e,n,m-v_0,n,m| < C_6e}
	\end{lemma}
	\begin{proof}
This proof follows the same approach as the proof of Lemma \ref{lem. |v_e-v_0|<C_5 e}, with the aid of Lemma \ref{lem estimate of b^i_n,m...}. The details are omitted here.
	\end{proof}

	\begin{lemma}
	Suppose that Assumption \ref{assume:A} holds. Let $\varepsilon>0$ and $(t,\mu) \in [0,T] \times \mathcal{P}_2(\mathbb{R}^d)$. If there exists $q>2$ such that $ \mu \in \mathcal{P}_q(\mathbb{R}^d)$, then $\displaystyle\lim_{n\to \infty}\lim_{m\to \infty} v_{\varepsilon,n,m}(t,\mu)=v_\varepsilon(t,\mu)$ where $v_\varepsilon(t,\mu)$ is defined in \eqref{def. v_e=V_e mfc}.
		\label{lem. conv. of v_e,n,m to v_e}
	\end{lemma}

	\begin{proof}
		The proof follows the same reasoning as \cite[Theorem A.6]{cosso_master_2022}, but with \cite[Theorems 3.1, 3.6]{limit_theory_tan} replacing the limit theory in \cite{cosso_master_2022} when passing to the limit $n \to \infty$ in $\lim_{m\to \infty} v_{\varepsilon,n,m}(t,\mu)$, because of the appearance of the common noise.

	\end{proof}

\section{Viscosity Solution Theory}
In this section, we prove the main result of the article, which is the existence and uniqueness of the viscosity solution of the HJB equation in \eqref{HJB_intro}. We first need the following lemma.
\begin{lemma}
        \label{v_W1_continuous}
            The value function $v(t,\mu)$ defined \eqref{value_function_after_law_invariance} and its finite-dimensional approximation $v_{\varepsilon,n,m}(t,\mu)$ defined \eqref{def. v_e=V_e mfg} satisfy
            $$|v_{\varepsilon,n,m}(t,\mu)-v_{\varepsilon,n,m}(t,\mu')|
                \leq C_4\mathcal{W}_1(\mu,\mu'),\quad
                |v (t,\mu)-v (t,\mu')|
                \leq C_4\mathcal{W}_1(\mu,\mu')$$
for any $(t,\mu,\mu')\in [0,T] \times \mathcal{P}_2(\mathbb{R}^d)\times \mathcal{P}_2(\mathbb{R}^d)$, where $C_4$ is given in Lemma \ref{lem. classical sol. of smooth approx.}.
        \end{lemma}
        \begin{proof}
            Let $\mu$, $\mu' \in \mathcal{P}_2(\mathbb{R}^d)$. Let $(\xi^1,\eta^1),\ldots,(\xi^n,\eta^n)$ be a sequence of independent and identically distributed random variables, with $\mathcal{L}(\xi^i) = \mu$, $\mathcal{L}(\eta^i) = \mu'$, for $i = 1,\ldots, n$. Theorem \ref{thm v_e,n,m} deduces that $v_{\varepsilon,n,m}(t,\mu) = \mathbb{E}\bar{v}_{\varepsilon,n,m}(t,\xi^1,\ldots,\xi^n)$ and $v_{\varepsilon,n,m}(t,\mu') = \mathbb{E}\bar{v}_{\varepsilon,n,m}(t,\eta^1,\ldots,\eta^n)$. Lemma \ref{lem. classical sol. of smooth approx.} implies that
            \begin{align*}
                |v_{\varepsilon,n,m}(t,\mu)-v_{\varepsilon,n,m}(t,\mu')|
                \leq \mathbb{E}|\bar{v}_{\varepsilon,n,m}(t,\xi^1,\ldots,\xi^n) - \bar{v}_{\varepsilon,n,m}(t,\eta^1,\ldots,\eta^n)|
                \leq &\frac{C_4}{n}\sum_{i=1}^n\mathbb{E}|\xi^i - \eta^i|\\
                =&C_4\mathbb{E}|\xi^1 - \eta^1|.
            \end{align*}
            Since $\xi^1$ and $\eta^1$ are arbitrary random variables such that $\mathcal{L}(\xi^1) = \mu$ and $\mathcal{L}(\eta^1) = \mu'$, we conclude that 
            \begin{align*}
                |v_{\varepsilon,n,m}(t,\mu)-v_{\varepsilon,n,m}(t,\mu')|\leq C_4\mathcal{W}_1(\mu,\mu').
            \end{align*}
            Now we consider $v$. Let $\mathcal{L}(\xi) = \mu$ and $\mathcal{L}(\eta) = \mu'$. By applying Proposition \ref{property_V}, together with Lemmas \ref{lem. |v_e,n,m-v_0,n,m| < C_6e} and \ref{lem. conv. of v_e,n,m to v_e}, we have
        \begin{align*}
                |v(t,\mu) - v(t,\mu')| =& |V(t,\xi) - V(t,\eta)|\\
                =& \lim_{k\to\infty} |V(t,\xi \mathds{1}_{\{|\xi|\leq k\}}) - V(t,\eta \mathds{1}_{\{|\eta|\leq k\}})|\\
                = &\lim_{k\to\infty} \lim_{\varepsilon\to0}\lim_{n\to\infty}\lim_{m\to\infty}|v_{\varepsilon,n,m}(t,\mathcal{L}(\xi \mathds{1}_{\{|\xi|\leq k\}})) - v_{\varepsilon,n,m}(t,\mathcal{L}(\eta \mathds{1}_{\{|\eta|\leq k\}}))|\\
                \leq &\,C_4\lim_{k\to\infty} \mathbb{E}|\xi \mathds{1}_{\{|\xi|\leq k\}} - \eta \mathds{1}_{\{|\eta|\leq k\}}|\\
                =&\,C_4 \mathbb{E}|\xi - \eta|.
            \end{align*}
            As $\xi$, $\eta$ are arbitrary, provided that $\mathcal{L}(\xi) = \mu$ and  $\mathcal{L}(\eta) = \mu'$, we conclude that $|v(t,\mu) - v(t,\mu')|\leq C_4\mathcal{W}_1(\mu,\mu')$.
        \end{proof}
\subsection{Existence}
We verify that the value function $v$ is a viscosity solution of \eqref{HJB_intro} by applying the It\^o's formula in Theorem \ref{standard_ito} and the dynamic programming in Theorem \ref{dpp_thm}. Similarly, the value function $v$ can be shown to be a viscosity solution under the standard Crandall-Lions definition by following the approach outlined in the theorem below. However, for consistency with the uniqueness theorem, we provide the proof under Definition \ref{def. of vis sol}.
\begin{theorem}
\label{existence_label}
	Suppose that Assumption \ref{assume:A} holds. The value function $v$ defined in \eqref{value_function_after_law_invariance} is a viscosity solution of \eqref{HJB_intro}.
		\label{thm. existence of vis sol}
\end{theorem}
\begin{proof}
 From Lemma \ref{v_W1_continuous}, the value function $v$ is $([0,T],|\cdot|) \times (\mathcal{P}_{2}(\mathbb{R}^d),\mathcal{W}_1)$-continuous, and its boundedness follows from the boundedness of $f$ and $g$. Now we show separately that it is a viscosity supersolution and a viscosity subsolution of \eqref{HJB_intro}, according to Definition \ref{vis_def}.\\
 \hfill\\
\noindent\textbf{Part 1. $v$ is a viscosity supersolution:} For any $s_0\in (0,T]$ and $\varphi \in C^{1,2}([0,s_0]\times\mathcal{P}_2(\mathbb{R}^d\times A))$, we assume that $v-\varphi$ attains a minimum at $(t_0,\nu_0)\in [0,s_0)\times\mathcal{P}_2(\mathbb{R}^d\times A)$ with a value of $0$, where $v(t_0,\nu_0)$ is defined as $v(t_0,\mu_0)$ with $\mu_0$ being the marginal of $\nu_0$ on $\mathbb{R}^d$. Recalling  that $\mathcal{M}_{t}$ is the set of $\mathcal{F}_{t}^{t}$-measurable random variables, we let $\alpha' \in \mathcal{M}_{t_0}$ and  $\xi \in L^2(\Omega^1,\mathcal{F}_t^1,\mathbb{P}^1;\mathbb{R}^d)$ such that $\mathcal{L}(\xi,\alpha') = \nu_0$. Letting $a\in A$ be arbitrary, we define $\alpha_s := a 
  \mathds{1}_{[0,t_0)}(s)
  +  \alpha' \mathds{1}_{[t_0,T]}(s)$, which belongs to $\mathcal{A}_{t_0}$. Let $(X_s^{t_0,\xi,\alpha})_{s\in [t_0, T]}$ be the solution of the dynamic \eqref{dynamics} with initial time $t_0$, initial data $\xi$ and control $\alpha$ defined in the above. Then for $h>0$ small enough, we use the dynamic programming in Theorem \ref{dpp_thm} and It\^o's formula in Theorem \ref{standard_ito} to obtain that
		\begin{align*}
			0\geq\,&\frac{1}{h}\mathbb{E}\Big[(v-\varphi)(t_0,\nu_0) - (v-\varphi)(t_0+h,\mathbb{P}^{W^0}_{(X_{t_0+h}^{t_0,\xi,\alpha},\alpha_{t_0+h})})\Big]\\
			\geq\, &\frac{1}{h}\mathbb{E}\Bigg[\int_{t_0}^{t_0+h}f(s,X_s^{t_0,\xi,\alpha},\mathbb{P}^{W^0}_{X_s^{t_0,\xi,\alpha}},\alpha_s) + \partial_t\varphi(s,\mathbb{P}_{(X_s^{t_0,\xi,\alpha},\alpha_s)}^{W^0})\\
   &\h{30pt}+ \mathbb{E}^1\big[\partial_\mu \varphi(s,\mathbb{P}^{W^0}_{(X_s^{t_0,\xi,\alpha},\alpha_s)})(X_s^{t_0,\xi,\alpha},\alpha_s)\cdot b(s,X_s^{t_0,\xi,\alpha},\mathbb{P}^{W^0}_{X_s^{t_0,\xi,\alpha}},\alpha_s)\big]\\
			&\h{30pt}+\frac{1}{2}\mathbb{E}^1\Big\{\operatorname{tr}\pig[\nabla_x\partial_\mu \varphi(s,\mathbb{P}^{W^0}_{(X_s^{t_0,\xi,\alpha},\alpha_s)})(X_s^{t_0,\xi,\alpha},\alpha_s)\sigma(s,X_s^{t_0,\xi,\alpha},\alpha_s)\sigma^\top(s,X_s^{t_0,\xi,\alpha},\alpha_s)\pig]\Big\}\\
&\h{30pt} +\frac{1}{2}\mathbb{E}^1\Big\{\operatorname{tr}\pig[\nabla_x\partial_\mu \varphi(s,\mathbb{P}^{W^0}_{(X_s^{t_0,\xi,\alpha},\alpha_s)})(X_s^{t_0,\xi,\alpha},\alpha_s)\sigma^0(s,X_s^{t_0,\xi,\alpha})\sigma^{0;\top}(s,X_s^{t_0,\xi,\alpha})\pig]\Big\}\\
&\h{30pt} +\frac{1}{2}\mathbb{E}^1\widecheck{\mathbb{E}}^1\Big\{\operatorname{tr}\Big[\partial_\mu^2 \varphi(s,\mathbb{P}^{W^0}_{(X_s^{t_0,\xi,\alpha},\alpha_s)})(X_s^{t_0,\xi,\alpha},\alpha_s,\widecheck{X}_s^{t_0,\xi,\alpha},\widecheck{\alpha}_s)\sigma^0(s,X_s^{t_0,\xi,\alpha})\sigma
^{0;\top}(s,\widecheck{X}_s^{t_0,\xi,\alpha})\Big]\Big\}ds\Bigg],
\end{align*}
where the operators $\partial_\mu$, $\partial_\mu^2$ are defined as in Definition \ref{vis_def}, through the projection. The process $(\widecheck{X}_s^{t_0,\xi,\alpha},\widecheck{\alpha}_s)$ is an independent copy of $(X_s^{t_0,\xi,\alpha},\alpha_s)$. For any $\varepsilon >0$, it holds that
\begin{equation}
\label{continuity_measure_flow}
\begin{split}
\mathbb{P}^0\left(\sup_{t\in[s,s+h]}\mathcal{W}_2(\mathbb{P}^{W^0}_{X_{t}^{t_0,\xi,\alpha}},\mathbb{P}^{W^0}_{X_{s}^{t_0,\xi,\alpha}})> \varepsilon\right)
\leq \,&\frac{\mathbb{E}^0\displaystyle\sup_{t\in[s,s+h]}\pig[\mathcal{W}_2(\mathbb{P}^{W^0}_{X_{t}^{t_0,\xi,\alpha}},\mathbb{P}^{W^0}_{X_{s}^{t_0,\xi,\alpha}})\pigr]^2}{\varepsilon^2}\\
\leq\,&\frac{\mathbb{E}\displaystyle\sup_{t\in[s,s+h]}|X_{t}^{t_0,\xi,\alpha}-X_{s}^{t_0,\xi,\alpha}|^2}{\varepsilon^2}\\
\leq\,&\frac{Ch}{\varepsilon^2},
\end{split}
\end{equation}
where in the last inequality we have used Proposition \ref{prop. property of X}. The above term goes to $0$ as $h\to 0^+$, thus the flow of measure  $s\mapsto \mathbb{P}_{X_s^{t_0,\xi,\alpha}}^{W^0}$ is continuous on $[t_0,s_0]$ for $\mathbb{P}$-a.s. $\omega^0 \in \Omega^0$, so is $s\mapsto \mathbb{P}_{(X_s^{t_0,\xi,\alpha},\alpha_s)}^{W^0}$ for $\mathbb{P}$-a.s. $\omega^0 \in \Omega^0$. From the regularity of the coefficients in Assumption \ref{assume:A}, the regularity of the test function $\varphi$, the continuity of the processes $s\mapsto X_s^{t_0,\xi,\alpha}$, $s\mapsto \mathbb{P}_{(X_s^{t_0,\xi,\alpha},\alpha_s)}^{W^0}$ for $\mathbb{P}$-a.s. $\omega^0 \in \Omega^0$ and the dominated convergence theorem, we have as $h\to 0^+$,
\begin{equation}
\label{supersolution}
\begin{split}
&\partial_t\varphi(t_0,\nu_0)+\mathbb{E}\widecheck{\mathbb{E}}\bigg\{f(t_0,\xi,\mu_0,\alpha') + \big[\partial_\mu \varphi(t_0,\nu_0)(\xi,\alpha')\cdot b(t_0,\xi,\mu_0,\alpha')\big]\\
			&+\frac{1}{2}\operatorname{tr}\Big[\nabla_x\partial_\mu \varphi(t_0,\nu_0)(\xi,\alpha')\sigma(t_0,\xi,\alpha')[\sigma(t_0,\xi,\alpha')]^\top\Big]+\frac{1}{2}\operatorname{tr}\Big[\nabla_x\partial_\mu \varphi(t_0,\nu_0)(\xi,\alpha')\sigma^0(t_0,\xi)[\sigma^0(t_0,\xi)]^\top\Big]\\
   &+\frac{1}{2}\operatorname{tr} \pig[\partial_\mu^2 \varphi(t_0,\nu_0)(\xi,\alpha',\widecheck{\xi},\widecheck{\alpha'})
   \sigma^0(t_0,\xi)[\sigma^0(t_0,\widecheck{\xi})]^\top\pig]\bigg\}\leq 0.
\end{split}		
  \end{equation}
\textbf{Part 2. $v$ is a viscosity subsolution:} For any $\varphi \in {C}^{1,2}([0,T]\times \mathcal{P}_2(\mathbb{R}^d))$, we assume that $v-\varphi$ attains a maximum at $(t_0,\mu_0)\in [0,T)\times\mathcal{P}_2(\mathbb{R}^d)$ with a value of $0$. Let $\varepsilon>0$ and $\xi \in L^2(\Omega^1,\mathcal{F}^1,\mathbb{P}^1;\mathbb{R}^d)$ satisfying $\mathcal{L}(\xi) = \mu_0$. There is an $\alpha^{\varepsilon} \in \mathcal{A}_t$ such that for any $h\in(0,T-t_0]$, we have
\begin{align*}
			v(t_0,\mu_0)-\varepsilon
            \leq\, &\mathbb{E}\Bigg[\int_{t_0}^{t_0+h} f(r,X_r^{t_0,\xi,\alpha^{\varepsilon}},\mathbb{P}^{W^0}_{X_r^{t_0,\xi,\alpha^{\varepsilon}}},\alpha_r^{\varepsilon})dr+v(t_0+h,\mathbb{P}^{W^0}_{X^{t_0,\xi,\alpha^{\varepsilon}}_{t_0+h}})\Bigg],
		\end{align*}
		where $(X_s^{t_0,\xi,\alpha^{\varepsilon}})_{s\in[t_0,T]}$ solves the dynamic \eqref{dynamics} with initial time $t_0$, initial data $\xi$ and control $\alpha^\varepsilon$. Then Theorem \ref{standard_ito} implies
		\begin{align}
			0\leq\,& \frac{1}{h}\mathbb{E}\Big[(v-\varphi)(t_0,\mu_0) - (v-\varphi)(t_0+h,\mathbb{P}^{W^0}_{X_{t_0+h}^{t_0,\xi,\alpha^{\varepsilon}}})\Big]\nonumber\\
			<\, &\frac{1}{h}\mathbb{E}\Bigg\{\int_{t_0}^{t_0+h}f(s,X_s^{t_0,\xi,\alpha^{\varepsilon}},\mathbb{P}^{W^0}_{X_s^{t_0,\xi,\alpha^{\varepsilon}}},\alpha_s^{\varepsilon}) + \partial_t\varphi(s,\mathbb{P}_{X_s^{t_0,\xi,\alpha^{\varepsilon}}}^{W^0})\nonumber\\
   &\h{30pt}+ \mathbb{E}^1\Big[\partial_\mu \varphi(s,\mathbb{P}^{W^0}_{X_s^{t_0,\xi,\alpha^{\varepsilon}}})(X_s^{t_0,\xi,\alpha^{\varepsilon}})\cdot b(s,X_s^{t_0,\xi,\alpha^{\varepsilon}},\mathbb{P}^{W^0}_{X_s^{t_0,\xi,\alpha^{\varepsilon}}},\alpha_s^{\varepsilon})\Big]\nonumber\\
			&\h{30pt}+ \frac{1}{2}\mathbb{E}^1\Big\{\operatorname{tr}\pig[\nabla_x\partial_\mu \varphi(s,\mathbb{P}^{W^0}_{X_s^{t_0,\xi,\alpha^{\varepsilon}}})(X_s^{t_0,\xi,\alpha^{\varepsilon}})\sigma(s,X_s^{t_0,\xi,\alpha^{\varepsilon}},\alpha_s^{\varepsilon})\sigma^\top(s,X_s^{t_0,\xi,\alpha^{\varepsilon}},\alpha_s^{\varepsilon})\pig]\Big\}\nonumber\\
   &\h{30pt}+ \frac{1}{2}\mathbb{E}^1\Big\{\operatorname{tr}\pig[\nabla_x\partial_\mu \varphi(s,\mathbb{P}^{W^0}_{X_s^{t_0,\xi,\alpha^{\varepsilon}}})(X_s^{t_0,\xi,\alpha^{\varepsilon}})\sigma^0(s,X_s^{t_0,\xi,\alpha^{\varepsilon}})\sigma^{0;\top}(s,X_s^{t_0,\xi,\alpha^{\varepsilon}})\pig]\Big\}\nonumber\\
   &\h{30pt}+ \frac{1}{2}\mathbb{E}^1\widecheck{\mathbb{E}}^1\Big\{\operatorname{tr}\Big[\partial_\mu^2 \varphi(s,\mathbb{P}^{W^0}_{X_s^{t_0,\xi,\alpha^\varepsilon}})(X_s^{t_0,\xi,\alpha^\varepsilon},\widecheck{X}_s^{t_0,\xi,\alpha^\varepsilon})\sigma^0(s,X_s^{t_0,\xi,\alpha^\varepsilon})\sigma
^{0;\top}(s,\widecheck{X}_s^{t_0,\xi,\alpha^\varepsilon})\Big]\Big\}ds\Bigg\}+\varepsilon\nonumber\\
			\leq &\frac{1}{h}\mathbb{E}\Bigg\{\int_{t_0}^{t_0+h}\sup_{a\in A}\bigg(f(s,X_s^{t_0,\xi,\alpha^\varepsilon},\mathbb{P}^{W^0}_{X_s^{t_0,\xi,\alpha^\varepsilon}},a)   
   + \partial_t\varphi(s,\mathbb{P}_{X_s^{t_0,\xi,\alpha^{\varepsilon}}}^{W^0}) \nonumber\\
   &\h{80pt}+ \mathbb{E}^1\big[\partial_\mu \varphi(s,\mathbb{P}^{W^0}_{X_s^{t_0,\xi,\alpha^\varepsilon}})(X_s^{t_0,\xi,\alpha^\varepsilon})\cdot b(s,X_s^{t_0,\xi,\alpha^\varepsilon},\mathbb{P}^{W^0}_{X_s^{t_0,\xi,\alpha^\varepsilon}},a)\big]\nonumber\\
			&\h{80pt}+ \frac{1}{2}\mathbb{E}^1\pig\{\operatorname{tr}\pig[\nabla_x\partial_\mu \varphi(s,\mathbb{P}^{W^0}_{X_s^{t_0,\xi,\alpha^\varepsilon}})(X_s^{t_0,\xi,\alpha^\varepsilon})\sigma(s,X_s^{t_0,\xi,\alpha^\varepsilon},a)\sigma^\top(s,X_s^{t_0,\xi,\alpha^\varepsilon},a)\pig]\Big\}\Bigg)\nonumber\\
   &\hspace{30pt}+ \frac{1}{2}\mathbb{E}^1\Big\{\operatorname{tr}\pig[\nabla_x\partial_\mu \varphi(s,\mathbb{P}^{W^0}_{X_s^{t_0,\xi,\alpha^{\varepsilon}}})(X_s^{t_0,\xi,\alpha^{\varepsilon}})\sigma^0(s,X_s^{t_0,\xi,\alpha^{\varepsilon}})\sigma^{0;\top}(s,X_s^{t_0,\xi,\alpha^{\varepsilon}})\pig]\Big\}\nonumber\\
   &\h{30pt}+ \frac{1}{2}\mathbb{E}^1\widecheck{\mathbb{E}}^1\Big\{\operatorname{tr}\Big[\partial_\mu^2 \varphi(s,\mathbb{P}^{W^0}_{X_s^{t_0,\xi,\alpha^\varepsilon}})(X_s^{t_0,\xi,\alpha^\varepsilon},\widecheck{X}_s^{t_0,\xi,\alpha^\varepsilon})\sigma^0(s,X_s^{t_0,\xi,\alpha^\varepsilon})\sigma
^{0;\top}(s,\widecheck{X}_s^{t_0,\xi,\alpha^\varepsilon})\Big]\Big\}\Bigg\}ds
+\varepsilon.
\label{759}
   \end{align}
  Arguing as in \eqref{supersolution},  we have
   \begin{align}
       &\frac{1}{h}\mathbb{E}\int_{t_0}^{t_0+h} \Bigg\{\partial_t \varphi(s,\mathbb{P}_{X_s^{t_0,\xi,\alpha^{\varepsilon}}}^{W^0})+ \frac{1}{2}\mathbb{E}^1\Big\{\operatorname{tr}\pig[\nabla_x\partial_\mu \varphi(s,\mathbb{P}^{W^0}_{X_s^{t_0,\xi,\alpha^{\varepsilon}}})(X_s^{t_0,\xi,\alpha^{\varepsilon}})\sigma^0(s,X_s^{t_0,\xi,\alpha^{\varepsilon}})\sigma^{0;\top}(s,X_s^{t_0,\xi,\alpha^{\varepsilon}})\pig]\Big\}\nonumber\\
   &\h{60pt}+ \frac{1}{2}\mathbb{E}^1\widecheck{\mathbb{E}}^1\Big\{\operatorname{tr}\Big[\partial_\mu^2 \varphi(s,\mathbb{P}^{W^0}_{X_s^{t_0,\xi,\alpha^\varepsilon}})(X_s^{t_0,\xi,\alpha^\varepsilon},\widecheck{X}_s^{t_0,\xi,\alpha^\varepsilon})\sigma^0(s,X_s^{t_0,\xi,\alpha^\varepsilon})\sigma
^{0;\top}(s,\widecheck{X}_s^{t_0,\xi,\alpha^\varepsilon})\Big]\Big\}\Bigg\}ds\nonumber\\
       &\longrightarrow \partial_t\varphi(t_0,\mu_0)+\frac{1}{2}\mathbb{E}\widecheck{\mathbb{E}}\Bigg\{\operatorname{tr}\Big[\nabla_x\partial_\mu \varphi(t_0,\mu_0)(\xi)\sigma^0(t_0,\xi)[\sigma^0(t_0,\xi)]^\top\Big]+\operatorname{tr} \pig[\partial_\mu^2 \varphi(t_0,\mu_0)(\xi,\widecheck{\xi})
   \sigma^0(t_0,\xi)[\sigma^0(t_0,\widecheck{\xi})]^\top\pig]\Bigg\},
   \label{767}
   \end{align}
as $h\to 0^+$. We define the following terms:
  \begin{align*}
      \RN{1}:=&\frac{1}{h}\mathbb{E}\int_{t_0}^{t_0+h}\sup_{a\in A}\Big\{f(s,X_s^{t_0,\xi,\alpha^\varepsilon},\mathbb{P}^{W^0}_{X_s^{t_0,\xi,\alpha^\varepsilon}},a)-f(t_0,\xi,\mu_0,a)\Big\}ds;\\
      \RN{2}:=&\frac{1}{h}\mathbb{E}\int_{t_0}^{t_0+h}\sup_{a\in A}\Big\{\mathbb{E}^1\pig[\partial_\mu \varphi(s,\mathbb{P}^{W^0}_{X_s^{t_0,\xi,\alpha^\varepsilon}})(X_s^{t_0,\xi,\alpha^\varepsilon})\cdot b(s,X_s^{t_0,\xi,\alpha^\varepsilon},\mathbb{P}^{W^0}_{X_s^{t_0,\xi,\alpha^\varepsilon}},a)\\
      &\h{85pt}-\partial_\mu \varphi(t_0,\mu_0)(\xi)\cdot b(t_0,\xi,\mu_0,a)\pig]\Big\}ds;\\
      \RN{3}:=&\frac{1}{h}\mathbb{E}\int_{t_0}^{t_0+h}\sup_{a\in A}\bigg\{\frac{1}{2}\mathbb{E}^1\Big\{\operatorname{tr}\pig[\nabla_x\partial_\mu \varphi(s,\mathbb{P}^{W^0}_{X_s^{t_0,\xi,\alpha^\varepsilon}})(X_s^{t_0,\xi,\alpha^\varepsilon})\sigma(s,X_s^{t_0,\xi,\alpha^\varepsilon},a)\sigma^\top(s,X_s^{t_0,\xi,\alpha^\varepsilon},a)\pig]\Big\}\\
  &\h{85pt}- \frac{1}{2}\mathbb{E}^1\Big\{\operatorname{tr}\pig[\nabla_x\partial_\mu \varphi(t_0,\mu_0)(\xi)\sigma(t_0,\xi,a)\sigma^\top(t_0,\xi,a)\pig]\Big\}\bigg\}ds.
  \end{align*}
  Then the above gives 
  \begin{align}
  &\frac{1}{h}\mathbb{E}\Bigg\{\int_{t_0}^{t_0+h}\sup_{a\in A}\bigg(f(s,X_s^{t_0,\xi,\alpha^\varepsilon},\mathbb{P}^{W^0}_{X_s^{t_0,\xi,\alpha^\varepsilon}},a)
  +\mathbb{E}^1\Big[\partial_\mu \varphi(s,\mathbb{P}^{W^0}_{X_s^{t_0,\xi,\alpha^\varepsilon}})(X_s^{t_0,\xi,\alpha^\varepsilon})\cdot b(s,X_s^{t_0,\xi,\alpha^\varepsilon},\mathbb{P}^{W^0}_{X_s^{t_0,\xi,\alpha^\varepsilon}},a)\Big]\nonumber\\
			&\h{80pt}+ \frac{1}{2}\mathbb{E}^1\Big\{\operatorname{tr}\pig[\nabla_x\partial_\mu \varphi(s,\mathbb{P}^{W^0}_{X_s^{t_0,\xi,\alpha^\varepsilon}})(X_s^{t_0,\xi,\alpha^\varepsilon})\sigma(s,X_s^{t_0,\xi,\alpha^\varepsilon},a)\sigma^\top(s,X_s^{t_0,\xi,\alpha^\varepsilon},a)\pig]\Big\}\bigg)ds\Bigg\}\nonumber\\ 
  &\leq\RN{1}+\RN{2}+\RN{3}+\mathbb{E}\sup_{a\in A}\Big\{f(t_0,\xi,\mu_0,a)+\partial_\mu \varphi(t_0,\mu_0)(\xi)\cdot b(t_0,\xi,\mu_0,a)\nonumber\\
  &\h{110pt}+\frac{1}{2}\operatorname{tr}\pig[\nabla_x\partial_\mu \varphi(t_0,\mu_0)(\xi)\sigma(t_0,\xi,a)\sigma^\top(t_0,\xi,a)\pig]\Big\}.\label{782}
  \end{align}
  We now show that all these three terms $\RN{1}$, $\RN{2}$, $\RN{3} \to 0$ as $h\to 0^+$. We first investigate $\RN{1}$ by
  \begin{align*}
   \RN{1}=\,&\frac{1}{h}\mathbb{E}\int_{t_0}^{t_0+h}\sup_{a\in A}\Big\{f(s,X_s^{t_0,\xi,\alpha^\varepsilon},\mathbb{P}^{W^0}_{X_s^{t_0,\xi,\alpha^\varepsilon}},a)-f(t_0,\xi,\mu_0,a)\Big\}ds\\ 
  \leq\,& \frac{1}{h}\mathbb{E}\int_{t_0}^{t_0+h}K\Big\{|X_s^{t_0,\xi,\alpha^\varepsilon}-\xi|+\mathcal{W}_1(\mathbb{P}_{X_s^{t_0,\xi,\alpha^\varepsilon}}^{W^0},\mu_0)+|s-t_0|^{\beta}\Big\}ds\\
  &\h{-10pt}\longrightarrow  0,
\end{align*}
as $h\to 0^+$ by Proposition \ref{prop. property of X} and Assumption \ref{assume:A}. For $\RN{2}$, we have
\begin{align}
    \RN{2}=\,&\frac{1}{h}\mathbb{E}\int_{t_0}^{t_0+h}\sup_{a\in A}\Big\{\mathbb{E}^1\pig[\partial_\mu \varphi(s,\mathbb{P}^{W^0}_{X_s^{t_0,\xi,\alpha^\varepsilon}})(X_s^{t_0,\xi,\alpha^\varepsilon})\cdot b(s,X_s^{t_0,\xi,\alpha^\varepsilon},\mathbb{P}^{W^0}_{X_s^{t_0,\xi,\alpha^\varepsilon}},a)\nonumber\\
    &\h{70pt}-\partial_\mu \varphi(t_0,\mu_0)(\xi)\cdot b(t_0,\xi,\mu_0,a)\pig]\Big\}ds\nonumber\\
    \leq\, &\frac{1}{h}\mathbb{E}\int_{t_0}^{t_0+h} \sup_{a\in  A}\bigg\{\mathbb{E}^1\Big[\partial_\mu \varphi(s,\mathbb{P}^{W^0}_{X_s^{t_0,\xi,\alpha^\varepsilon}})(X_s^{t_0,\xi,\alpha^\varepsilon})\cdot b(s,X_s^{t_0,\xi,\alpha^\varepsilon},\mathbb{P}^{W^0}_{X_s^{t_0,\xi,\alpha^\varepsilon}},a)\Big]\nonumber\\
    &\,\h{70pt}-\mathbb{E}^1\Big[\partial_\mu \varphi(t_0,\mu_0)(\xi)\cdot b(s,X_s^{t_0,\xi,\alpha^\varepsilon},\mathbb{P}^{W^0}_{X_s^{t_0,\xi,\alpha^\varepsilon}},a)\Big]\bigg\}ds\nonumber\\
    &+\frac{1}{h}\mathbb{E}\int_{t_0}^{t_0+h} \sup_{a\in  A}\mathbb{E}^1\Big[\partial_\mu \varphi(t_0,\mu_0)(\xi)\cdot b(s,X_s^{t_0,\xi,\alpha^\varepsilon},\mathbb{P}^{W^0}_{X_s^{t_0,\xi,\alpha^\varepsilon}},a)-\partial_\mu \varphi(t_0,\mu_0)(\xi)\cdot b(t_0,\xi,\mu_0,a)\Big]ds\nonumber\\
    \leq\, &\frac{1}{h}\mathbb{E}\int_{t_0}^{t_0+h} \sup_{a\in A}\mathbb{E}^1\Big[\pig|\partial_\mu \varphi(s,\mathbb{P}^{W^0}_{X_s^{t_0,\xi,\alpha^\varepsilon}})(X_s^{t_0,\xi,\alpha^\varepsilon})-\partial_\mu \varphi(t_0,\mu_0)(\xi)\pig|\pig|b(s,X_s^{t_0,\xi,\alpha^\varepsilon},\mathbb{P}^{W^0}_{X_s^{t_0,\xi,\alpha^\varepsilon}},a)\pig|\Big]ds\nonumber\\
    &+\frac{K}{h}\mathbb{E}\int_{t_0}^{t_0+h} \mathbb{E}^1\Big[\pig|\partial_\mu \varphi(t_0,\mu_0)(\xi)\pig|\pig(|s-t_0|^\beta + |X_s^{t_0,\xi,\alpha^\varepsilon}-\xi|+\mathcal{W}_1(\mathbb{P}^{W^0}_{X_s^{t_0,\xi,\alpha^\varepsilon}},\mu_0)\pig)\Big]ds\nonumber\\
    \leq\,&\frac{K}{h}\mathbb{E}\int_{t_0}^{t_0+h} \mathbb{E}^1\Big[\pig|\partial_\mu \varphi(s,\mathbb{P}^{W^0}_{X_s^{t_0,\xi,\alpha^\varepsilon}})(X_s^{t_0,\xi,\alpha^\varepsilon})-\partial_\mu \varphi(t_0,\mu_0)(\xi)\pig|\Big(|1+|X_s^{t_0,\xi,\alpha^\varepsilon}|^\rho\Big)\Big]ds\label{1314}\\
    &+\frac{\sqrt{3}K}{h}\sqrt{\mathbb{E}\mathbb{E}^1|\partial_\mu \varphi(t_0,\mu_0)(\xi)|^2}\int_{t_0}^{t_0+h} \sqrt{\mathbb{E}\mathbb{E}^1\Big(|s-t_0|^{2\beta} + |X_s^{t_0,\xi,\alpha^\varepsilon}-\xi|^2+\pig[\mathcal{W}_1(\mathbb{P}^{W^0}_{X_s^{t_0,\xi,\alpha^\varepsilon}},\mu_0)\pigr]^2\Big)}ds.\nonumber
\end{align}
The convergence of the term in \eqref{1314} is due to the continuity as proved in \eqref{continuity_measure_flow}, Proposition \ref{prop. property of X}, the regularity of the test function $\varphi$ and the dominated convergence theorem. Thus, the term II converges to zero as $h \to 0^+$ with the aid of Proposition \ref{prop. property of X}. Similar estimate holds for the term $\RN{3}$, thus putting \eqref{767} and \eqref{782} into \eqref{759} and then passing $h\to 0^+$, we obtain that 
\begin{align*}
			0\leq \,&\partial_t\varphi(t_0,\mu_0)+\mathbb{E}\sup_{a\in A}\bigg\{f(t_0,\xi,\mu_0,a) + \big[\partial_\mu \varphi(t_0,\mu_0)(\xi)\cdot b(t_0,\xi,a,\mu_0)\big]\\
			&+\frac{1}{2}\operatorname{tr}\pig[\nabla_x\partial_\mu \varphi(t_0,\mu_0)(\xi)\sigma(t_0,\xi,a)\sigma^\top(t_0,\xi,a)\pig]\bigg\}+\mathbb{E}\widecheck{\mathbb{E}}\Bigg\{\frac{1}{2}\operatorname{tr}\Big[\nabla_x\partial_\mu \varphi(t_0,\mu_0)(\xi)\sigma^0(t_0,\xi)[\sigma^0(t_0,\xi)]^\top\Big]\\&+\frac{1}{2}\operatorname{tr} \pig[\partial_\mu^2 \varphi(t_0,\mu_0)(\xi,\widecheck{\xi})
   \sigma^0(t_0,\xi)[\sigma^0(t_0,\widecheck{\xi})]^\top\pig]\Bigg\}+\varepsilon. 
		\end{align*}
  As $\varepsilon>0$ is arbitrary, the claim is proved.
	\end{proof}
\subsection{Uniqueness}
This subsection demonstrates that any viscosity subsolution is less than or equal to any viscosity supersolution, leading to the uniqueness of viscosity solution which is the main result of this article.
\begin{theorem}\label{thm:main}
		Suppose that Assumptions \ref{assume:A}-\ref{assume:B} hold. Let  $u_1$, $u_2 : [0,T]\times \mathcal{P}_2(\mathbb{R}^d) \to \mathbb{R}$ be the viscosity subsolution and supersolution (in the sense of Definition \ref{def. of vis sol}) of equation \eqref{HJB_intro} respectively. Then it holds that $u_1 \leq u_2$ on $[0,T]\times \mathcal{P}_2(\mathbb{R}^d)$. Hence, the viscosity solution of equation \eqref{HJB_intro} is unique.
		\label{thm compar}
	\end{theorem}

\begin{proof}
Recalling the function $v_0$ defined in \eqref{def. v_e=V_e mfc} with $\varepsilon=0$, we shall prove that $u_1\leq v_{0}$ and $v_{0} \leq u_2$ on $[0,T]\times \mathcal{P}_2(\mathbb{R}^d)$. We prove the cases for subsolution and supersolution separately.

\medskip

\noindent {\bf Part 1. Proof of  $u_1\leq v_{0}$:} Let $u_1$ be a bounded viscosity subsolution of the equation \eqref{HJB_intro}. To prove $u_1\leq v_{0}$ by contradiction, we assume  that there exists $(t_0,\widetilde{\mu}_0)\in [0,T] \times \mathcal{P}_2(\mathbb{R}^d)$ such that
		\begin{align*}
			(u_1-v_{0})(t_0,\widetilde{\mu}_0) >0.
		\end{align*}
		Let $\xi\in L^2(\Omega,\mathcal{F},\mathbb{P};\mathbb{R}^d)$ such that $\mathcal{L}(\xi)=\widetilde{\mu}_0$. For any $k\in \mathbb{N}$, we let $\mu_0^k \in \mathcal{P}_2(\mathbb{R}^d)$ be the law of $\xi\mathds{1}_{\{|\xi|\leq k\}}$. We see that $\mu_0^k \in \mathcal{P}_q(\mathbb{R}^d)$ for any $q\geq1$ and
		$$ \pig[\mathcal{W}_1(\mu_0^k,\widetilde{\mu}_0)\pigr]^2\leq\pig[\mathcal{W}_2(\mu_0^k,\widetilde{\mu}_0)\pigr]^2\leq \mathbb{E}\pig[|\xi\mathds{1}_{\{|\xi|\leq k\}}-\xi|^2\pig] =\int_{|x|>k}|x|^2\widetilde{\mu}_0(dx) \longrightarrow 0$$
		as $k \to \infty$. Therefore, as both $u_1$ and $v_{0}$ are continuous on $[0,T] \times \mathcal{P}_2(\mathbb{R}^d)$, we can find a $k \in \mathbb{N}$ large enough such that $\mu_0 := \mu_0^k\in \mathcal{P}_q(\mathbb{R}^d)$ for any $q\geq1$ and 
		\begin{align}
			(u_1-v_{0})(t_0, \mu_0) >0.
			\label{ineq. u_1-v_0(t_0,mu_0)>0}
		\end{align}
  	 \noindent {\bf Step 1A: Construction of the comparison function:} Recalling the approximations defined in \eqref{def. approx of f}, \eqref{def. approx of g}, \eqref{def. tilde v_e,n,m} and \eqref{eq. v_e,n,m=int bar of v_e,n,m}, we define $\widecheck{u}_1(t,x):=e^{t-t_0}u_1(t,x)$ and similarly for $\widecheck{\overline{v}}_{\varepsilon,n,m}$, $\widecheck{v}_{\varepsilon,n,m}$, $\widecheck{f}^i_{n,m}$, $\widecheck{f}$ from $\overline{v}_{\varepsilon,n,m}$, $v_{\varepsilon,n,m}$, $f^i_{n,m}$, $f$ respectively. We also define $\widecheck{g}:=e^{T-t_0}g$ and $\widecheck{g}^i_{n,m}:=e^{T-t_0}g^i_{n,m}$.  By direct computation, we see that $\widecheck{u}_1$ is a viscosity subsolution of the equation
  \begin{equation}\label{eq. vis subsol of check u1}
			\left\{\begin{aligned}
				&\partial_t u(t,\mu)
				+\int_{\mathbb{R}^d}
				\sup_{a\in A} \Bigg\{\widecheck{f}(t,x,\mu,a)+b(t,x,\mu,a)\cdot \partial_\mu u(t,\mu)(x) \\
				&\h{90pt}+ \dfrac{1}{2}\text{tr}\Big[\pig(\sigma(t,x,a)\big[\sigma(t,x,a)\big]^\top+\sigma^0(t,x)[\sigma^0(t,x)]^\top\pig)\nabla_x\partial_\mu u(t,\mu)(x)\Big]\Bigg\}\mu(dx) \\
				&+ \dfrac{1}{2} \int_{\mathbb{R}^{2d}}\text{tr}\Big[\sigma^0(t,x)[\sigma^0(t,y)]^\top\partial_\mu^2
    u(t,\mu)(x,y)\Big] \, \mu^{\otimes 2}(dx,dy)-u(t,\mu)\\
    &=0\h{5pt} \text{for $(t,\mu) \in [0,T) \times\mathcal{P}_2(\mathbb{R}^d)$};\\
				&u(T,\mu)=\int_{\mathbb{R}^d}\widecheck{g}(x,\mu)\mu(dx) \h{10pt}\text{for $\mu\in \mathcal{P}_2(\mathbb{R}^d)$},
			\end{aligned}\right.
		\end{equation}
in the sense of Definition \ref{def. of vis sol}. By Theorem \ref{thm v_e,n,m}, we obtain that $\widecheck{v}_{\varepsilon,n,m}$ solves the following equation in the classical sense:
		\begin{equation}
			\left\{
			\begin{aligned}
				&\partial_t u(t,\mu)+\int_{\mathbb{R}^{dn}}
				\sup_{\overline{a} \in A^n}\Bigg\{\dfrac{1}{n}\sum^n_{i=1}\widecheck{f}^i_{n,m}(t,\overline{x},a^i)
    +\sum^n_{i=1}\Big\langle b^i_{n,m}(t,\overline{x},a^i)
				,\nabla_{x^i}\widecheck{\overline{v}}_{\varepsilon,n,m}(t,\overline{x})\Big\rangle\\
				&\h{100pt}+\dfrac{1}{2}
				\sum^n_{i=1}\textup{tr}\left[\Big((\sigma\sigma^\top)(t,x^i,a^i)+(\sigma^0\sigma^{0;\top})(t,x^i)+\varepsilon^2I_d\Big)\nabla_{x^ix^i}^2
				\widecheck{\overline{v}}_{\varepsilon,n,m}
				(t,\overline{x})\right]\\
				&\h{100pt}+\dfrac{1}{2}
				\sum^n_{i,j=1,i\neq j}\textup{tr}\left[\sigma^0(t,x^i)\sigma^{0;\top}(t,x^j)\nabla_{x^ix^j}^2
				\widecheck{\overline{v}}_{\varepsilon,n,m}
		(t,\overline{x})\right]\Bigg\}\bbotimes_{k=1}^n\mu(dx^k)
				-u (t,\mu)\\
				&=0\h{5pt} \text{for $(t,\mu) \in [0,T) \times\mathcal{P}_2(\mathbb{R}^{d})$};\\
				& u(T,\mu)
				=\dfrac{1}{n}\sum^n_{i=1}\int_{\mathbb{R}^{dn}}\widecheck{g}^i_{n,m}(\overline{x})\bbotimes_{k=1}^n\mu(dx^k)
				\h{5pt} \text{for $\mu \in \mathcal{P}_2(\mathbb{R}^{d})$},
			\end{aligned}
			\right.    
			\label{eq. vis sol of check v_e,n,m}
		\end{equation}
		where we write 	$\displaystyle\bbotimes_{k=1}^n\mu(dx^k):=\mu(dx^1)\otimes\ldots\otimes\mu(dx^n)$ and $\overline{a}=(a^1,\ldots,a^n)$ for each $a^i \in A$ with $i=1,2,\ldots,n$.
  
Let $l_0:=(u_1-v_{0})(t_0, \mu_0)>0$. Referring to \eqref{moment_V_def} for the definition of $M_{2}(\mu_0)$, we choose a sufficiently small $\delta>0$, depending on $M_{2}(\mu_0)$ and $l_0$ only, such that $u_1(t_0,\mu_0)- v_0(t_0,\mu_0)-\delta M_{2}(\mu_0)\geq l_0/2$ and thus $u_1(t_0,\mu_0)- v_{\varepsilon, n,m}(t_0,\mu_0)-\delta M_{2}(\mu_0)\geq l_0/3$, for small enough $\varepsilon>0$ and large enough $n$, $m\in \mathbb{N}$ depending on $l_0$, by Lemmas \ref{lem. |v_e-v_0|<C_5 e} and \ref{lem. conv. of v_e,n,m to v_e}. Thus, it holds that 
\begin{equation}\label{eq:cpt}
\displaystyle\sup_{(t,\mu)\in [0,T] \times \mathcal{P}_2(\mathbb{R}^d)} \widecheck{u}_1(t,\mu)- \widecheck{v}_{\varepsilon, n,m}(t,\mu)-\delta M_{2}(\mu)\geq l_0/3,
\end{equation}
for small enough $\varepsilon>0$ and large enough $n$, $m\in \mathbb{N}$ depending on $l_0$. Define the set 
\begin{align*}
U^2_{\delta,\varepsilon, n,m}:=\pig\{(t,\mu) \in [0,T]\times\mathcal{P}_2(\mathbb{R}^d):\widecheck{u}_1(t,\mu)- \widecheck{v}_{\varepsilon, n,m}(t,\mu)-\delta M_{2}(\mu)\geq l_0/3\pig\},
\end{align*}
which is non-empty as $u_1(t_0,\mu_0)- v_{\varepsilon, n,m}(t_0,\mu_0)-\delta M_{2}(\mu_0)\geq l_0/3$. It is obvious that
\begin{align*}
    \displaystyle\sup_{(t,\mu)\in [0,T] \times \mathcal{P}_2(\mathbb{R}^d)} \widecheck{u}_1(t,\mu)- \widecheck{v}_{\varepsilon, n,m}(t,\mu)-\delta M_{2}(\mu)
    =\displaystyle\sup_{(t,\mu)\in U^2_{\delta,\varepsilon, n,m}} \widecheck{u}_1(t,\mu)- \widecheck{v}_{\epsilon, n,m}(t,\mu)-\delta M_{2}(\mu).
\end{align*}
Since $\widecheck{u}_1$ and $\widecheck{v}_{\varepsilon,n,m}$ are bounded independent of $\varepsilon,n,m$ by Theorem \ref{thm v_e,n,m}, it yields that for any $\mu \in U^2_{\delta,\varepsilon, n,m}$,
\begin{align}
\label{bound_V}
\delta M_2(\mu)
\leq l_0/3+e^T\lVert u_1\rVert_{\infty} + e^T\lVert v_{\varepsilon,n,m} \rVert_{\infty}
\leq l_0/3+e^T(\lVert u_1\rVert_{\infty} +\ell_2)
< +\infty,
\end{align}
where $\ell_2$ is the constant given in Theorem \ref{thm v_e,n,m}. Then we see that
\begin{align*}
U^2_{\delta,\varepsilon, n,m}\subset\Big\{(t,\mu) \in [0,T]\times\mathcal{P}_2(\mathbb{R}^d):
M_{2}(\mu)\leq \dfrac{1}{\delta}\pig[l_0/3+e^T(\lVert u_1\rVert_{\infty} +\ell_2)\pig]\Big\}=:U^2_\delta.
\end{align*}
The set $U^2_\delta$ is compact in $\big([0,T],|\cdot|\big) \times (\mathcal{P}_{2}(\mathbb{R}^d),\mathcal{W}_1)$ due to Lemma \ref{lem. compact of V^p_K}. Note that $\widecheck{u}_1(t,\mu)$ is $\big([0,T],|\cdot|\big) \times (\mathcal{P}_{2}(\mathbb{R}^d),\mathcal{W}_1)$-continuous by our definition of viscosity solution, $\widecheck{v}_{\varepsilon, n,m}(t,\mu)$ is $\big([0,T],|\cdot|\big) \times (\mathcal{P}_{2}(\mathbb{R}^d),\mathcal{W}_1)$-continuous by Lemma \ref{v_W1_continuous}, and $-\delta M_2(\mu)$ is $\mathcal{W}_1$-upper semicontinuous from Remark \ref{V_lower_semi}, therefore $\widecheck{u}_1-\widecheck{v}_{\varepsilon,n,m}-\delta M_2$ is upper semicontinuous in $\big([0,T],|\cdot|\big) \times (\mathcal{P}_{2}(\mathbb{R}^d),\mathcal{W}_1)$. It implies that the set $U^2_{\delta,\varepsilon, n,m}$ is a closed subset of $U^2_\delta$ under $\big([0,T],|\cdot|\big) \times (\mathcal{P}_{2}(\mathbb{R}^d),\mathcal{W}_1)$. Therefore, $U^2_{\delta,\varepsilon, n,m}$ is compact in $\big([0,T],|\cdot|\big) \times (\mathcal{P}_{2}(\mathbb{R}^d),\mathcal{W}_1)$. The same arguments imply that  $U^2_{\delta,\varepsilon, n,m}$ is compact in $\big([0,T],|\cdot|\big) \times (\mathcal{P}_{1}(\mathbb{R}^d),\mathcal{W}_1)$ as well. Let $\{(t_k,\mu_k)\}_{k\in\mathbb{N}}$ be a sequence in $U^2_{\delta,\varepsilon, n,m}$ such that 
\begin{align}
\label{about_to_limsup}
    \widecheck{u}_1(t_k,\mu_k)- \widecheck{v}_{\varepsilon, n,m}(t_k,\mu_k)-\delta M_{2}(\mu_k)>\displaystyle\sup_{(t,\mu)\in U^2_{\delta,\epsilon, n,m}} \widecheck{u}_1(t,\mu)- \widecheck{v}_{\varepsilon, n,m}(t,\mu)-\delta M_{2}(\mu)-\frac{1}{k}.
\end{align}
By the compactness of $U^2_{\delta,\varepsilon, n,m}$ and the upper semicontinuity of $\widecheck{u}_1-\widecheck{v}_{\varepsilon,n,m}-\delta M_2$, there exists $(\widetilde{t},\widetilde{\mu}) \in [0,T] \times U^2_{\delta,\varepsilon, n,m}$ such that $(t_k,\mu_k)\to(\widetilde{t},\widetilde{\mu})$ in $\big([0,T],|\cdot|\big) \times (\mathcal{P}_{2}(\mathbb{R}^d),\mathcal{W}_1)$. From \eqref{about_to_limsup}, we conclude that 
\begin{align*}
    \widecheck{u}_1(\widetilde{t},\widetilde{\mu})- \widecheck{v}_{\varepsilon, n,m}(\widetilde{t},\widetilde{\mu})-\delta M_{2}(\widetilde{\mu})\geq& \limsup_k \pig[\widecheck{u}_1(t_k,\mu_k)- \widecheck{v}_{\varepsilon, n,m}(t_k,\mu_k)-\delta M_{2}(\mu_k)\pig]\\
    \geq&\sup_{(t,\mu)\in U^2_{\delta,\varepsilon, n,m}} \widecheck{u}_1(t,\mu)- \widecheck{v}_{\varepsilon, n,m}(t,\mu)-\delta M_{2}(\mu),
\end{align*}
and the maximum is attained at $(\widetilde{t},\widetilde{\mu})$. We note that this maximum point depends on $\delta,\varepsilon, n,m$.\\
\hfill\\
 \noindent{\bf Step 1B. Proof of $\widetilde{t}<T$:} In this step, we aim to prove that $\widetilde{t}<T$. Suppose, on the contrary, that  $\widetilde{t}=T$. The definition of the point $(\widetilde{t},\widetilde{\mu})=(T,\widetilde{\mu})$ implies that $u_1(t_0,\mu_0) - v_{\varepsilon,n,m}(t_0,\mu_0)-\delta M_{2}(\mu_0)\leq  \widecheck{u}_1(T,\widetilde{\mu}) - \widecheck{v}_{\varepsilon,n,m}(T,\widetilde{\mu})-\delta M_{2}(\widetilde{\mu})\leq  \widecheck{u}_1(T,\widetilde{\mu}) - \widecheck{v}_{\varepsilon,n,m}(T,\widetilde{\mu})$. Thus, for $\widehat{\mu}^{n,\overline{x}}:=\dfrac{1}{n}\displaystyle\sum^n_{j=1} \delta_{x^j}$, we can use the terminal conditions in equations \eqref{eq. vis subsol of check u1} and \eqref{eq. vis sol of check v_e,n,m} to obtain that 
		\begin{align*}
			&\h{-10pt}u_1(t_0,\mu_0) - v_{\varepsilon,n,m}(t_0,\mu_0)\\
			\leq\,& \dfrac{e^{T-t_0}}{n}
			\sum^n_{i=1}\left[
			\int_{\mathbb{R}^{dn}}
			\left(g(x^i,\widetilde{\mu})
			-g^i_{n,m}(x^1,\ldots,x^n)\right)
			\bbotimes_{k=1}^n\widetilde{\mu}(dx^k)
			\right]+\delta M_{2}(\mu_0)\\
			=\,& \dfrac{e^{T-t_0}}{n}
			\sum^n_{i=1}\left[
			\int_{\mathbb{R}^{dn}}
			\left(g(x^i,\widetilde{\mu})
			-g(x^i,\widehat{\mu}^{n,\overline{x}})\right)
			\bbotimes_{k=1}^n\widetilde{\mu}(dx^k)
			\right]\\
			&+\dfrac{e^{T-t_0}}{n}
			\sum^n_{i=1}\left[
			\int_{\mathbb{R}^{dn}}
			\left(g(x^i,\widehat{\mu}^{n,\overline{x}})
			-g^i_{n,m}(x^1,\ldots,x^n)\right)
			\bbotimes_{k=1}^n\widetilde{\mu}(dx^k)
			\right]
   +\delta M_{2}(\mu_0).
		\end{align*}
		Using the Lipschitz property of $g$ in Assumption \ref{assume:A} and (2) of Lemma \ref{lem estimate of b^i_n,m...}, we further have
		\begin{align}
			&\h{-10pt}u_1(t_0,\mu_0) - v_{\varepsilon,n,m}(t_0,\mu_0)\nonumber\\
			\leq\,&Ke^{T-t_0}
			\left[
			\int_{\mathbb{R}^{dn}}
			\mathcal{W}_1(\widetilde{\mu},\widehat{\mu}^{n,\overline{x}})
			\bbotimes_{k=1}^n\widetilde{\mu}(dx^k)
			\right]+\dfrac{2Ke^{T-t_0}}{n}
			\left[
			m^{dn}\int_{\mathbb{R}^{dn}}
			\left(\sum^n_{i=1}|y^i|\right) \prod^n_{j=1}\Phi(my^j)dy^j
			\right]\nonumber\\
   &+\delta M_{2}(\mu_0).
			\label{1139}
		\end{align}
  From \cite[Theorem 1]{FG15}, there is a constant $C_d>0$ depending on $d$ only and a sequence $\{h_n\}_{n \in \mathbb{N}} \subset \mathbb{R}$ such that 
  \begin{align}
			\int_{\mathbb{R}^{dn}}
			\mathcal{W}_1(\widetilde{\mu},\widehat{\mu}^{n,\overline{x}})
			\bbotimes_{k=1}^n\widetilde{\mu}(dx^k)
			\leq C_{d} h_n\left[\int_{\mathbb{R}^{d}}|x|^{q_0}\widetilde{\mu}(dx)\right]^{1/q_0}.
   \label{def. h_n}
		\end{align}
The sequence $\{h_n\}_{n \in \mathbb{N}}$ and the number $q_0$ are given by
  \[
h_n=
\left\{
\begin{array}{ll}
n^{-1/2} + n^{-(q_0-1)/q_0} & \text{if } d=1; \\
n^{-1/2} \log(1+n) + n^{-(q_0-1)/q_0} & \text{if }  d=2; \\
n^{-1/d} + n^{-(q_0-1)/q_0} & \text{if } d>2,
\end{array}
\right.
\quad \text{with}\quad
q_0=
\left\{
\begin{array}{ll}
3/2 & \text{if } d=1,2; \\
5/3 & \text{if } d>2,
\end{array}
\right.
\]
where $h_n$ satisfies $\lim_{n\to \infty}h_n=0$. Inequalities \eqref{bound_V} and \eqref{def. h_n} further imply
		\begin{align}
			\dfrac{1}{C_dh_n}\int_{\mathbb{R}^{dn}}
			\mathcal{W}_1(\widetilde{\mu},\widehat{\mu}^{n,\overline{x}})
			\bbotimes_{k=1}^n\widetilde{\mu}(dx^k)
			\leq
			\left[\int_{\mathbb{R}^{d}}
			|x|^2
			\widetilde{\mu}(dx)\right]^{1/2} 
   \leq\dfrac{1}{\delta^{1/2}}\pig[ l_0/3+e^T(\lVert u_1\rVert_{\infty} +\ell_2)\pigr]^{1/2}.
			\label{ineq. int W_2(tilde mu- hat mu^n,x)}
		\end{align}
		Hence, from \eqref{1139} and \eqref{ineq. int W_2(tilde mu- hat mu^n,x)}, we have,
		\begin{align*}
			&\h{-10pt}u_1(t_0,\mu_0) - v_{\varepsilon,n,m}(t_0,\mu_0)\\
   \leq\,& \dfrac{K C_dh_ne^{T-t_0}}{\delta^{1/2}}\pig[ l_0/3+e^T(\lVert u_1\rVert_{\infty} +\ell_2)\pigr]^{1/2}
	+\dfrac{2Ke^{T-t_0}m^{dn}}{n}
			\left[
			\int_{\mathbb{R}^{dn}}
			\left(\sum^n_{i=1}|y^i|\right) \prod^n_{j=1}\Phi(my^j)dy^j
			\right]\\
   &+\delta M_{2}(\mu_0).
		\end{align*}
		Passing $m \to \infty$ and then $n \to \infty$ subsequently, we use the fact that $h_n \to 0$ to yield that 
		\begin{align*}
			u_1(t_0,\mu_0) - \lim_{n \to \infty}\lim_{m \to \infty} v_{\varepsilon,n,m}(t_0,\mu_0)
			\leq \delta M_{2}(\mu_0).
		\end{align*}
		Finally, using Lemmas \ref{lem. |v_e-v_0|<C_5 e} and \ref{lem. conv. of v_e,n,m to v_e}, we pass $\varepsilon \to 0^+$ then $\delta\to 0^+$ to conclude that $u_1(t_0,\mu_0)-v_0(t_0,\mu_0)\leq 0$, which contradicts \eqref{ineq. u_1-v_0(t_0,mu_0)>0} and thus $\widetilde{t}<T$.\\
  \hfill\\
   \noindent{\bf Step 1C. Estimate of  $u_1-v_{\varepsilon,n,m}$:}  We assume the maximum value of $\widecheck{u}_1(t,\mu)- \widecheck{v}_{\varepsilon, n,m}(t,\mu)-\delta M_{2}(\mu)$ attained at $(\widetilde{t},\widetilde{\mu})$ over $[0,T]\times \mathcal{P}_2(\mathbb{R}^d)$ is $M^* \in \mathbb{R}$. As $\widecheck{v}_{\varepsilon, n,m}(t,\mu)+\delta M_{2}(\mu) \in C^{1,2}([0,T]\times \mathcal{P}_2(\mathbb{R}^d))$, we use the fact that $\widecheck{u}_1$ is the viscosity subsolution of \eqref{eq. vis subsol of check u1} to see that
\begin{align*}
			0\leq\,&\partial_t (\widecheck{v}_{\varepsilon,n,m}+\delta  M_2) (\widetilde{t},\widetilde{\mu})
			-(\widecheck{v}_{\varepsilon,n,m}+\delta  M_2) (\widetilde{t},\widetilde{\mu})-M^*\\
			&+\int_{\mathbb{R}^d}\sup_{a \in A}\Bigg\{
			\widecheck{f}(\widetilde{t},x,\widetilde{\mu},a)+ \Big\langle b (\widetilde{t},x,\widetilde{\mu},a),\p_\mu (\widecheck{v}_{\varepsilon,n,m}+\delta  M_2) (\widetilde{t},\widetilde{\mu})(x)\Big\rangle
			\\
			&\h{55pt}
			+\dfrac{1}{2} \textup{tr}\Big\{\pig[\sigma(\widetilde{t},x,a)\big[\sigma(\widetilde{t},x,a)\big]^\top+\sigma^0(\widetilde{t},x)[\sigma^0(\widetilde{t},x)]^\top\pig]
			\nabla_x\p_\mu(\widecheck{v}_{\varepsilon,n,m}+\delta  M_2) (\widetilde{t},\widetilde{\mu})(x)\Big\}\Bigg\}\widetilde{\mu}(dx)\\
   				&+ \dfrac{1}{2} \int_{\mathbb{R}^{2d}}\text{tr}\Big[\sigma^0(\widetilde{t},x)[\sigma^0(\widetilde{t},y)]^\top\partial_\mu^2
    (\widecheck{v}_{\varepsilon,n,m}+\delta  M_2)(\widetilde{t},\widetilde{\mu})(x,y)\Big] \, \widetilde{\mu}^{\otimes 2}(dx,dy).
		\end{align*}
		Therefore, as $\widecheck{u}_1(\widetilde{t},\widetilde{\mu}) - \widecheck{v}_{\varepsilon,n,m}(\widetilde{t},\widetilde{\mu})-\delta  M_2(\widetilde{\mu})\leq M^*$ and $\widecheck{v}_{\varepsilon,n,m}$ solves \eqref{eq. vis sol of check v_e,n,m}, we further have
  \begin{align}
			&\h{-10pt}(\widecheck{u}_1-\widecheck{v}_{\varepsilon,n,m})(\widetilde{t},\widetilde{\mu})\nonumber\\
			\leq\,&
			\int_{\mathbb{R}^d}\sup_{a \in A}\Bigg\{
			\widecheck{f}(\widetilde{t},x,\widetilde{\mu},a)+ \Big\langle b (\widetilde{t},x,\widetilde{\mu},a),\p_\mu (\widecheck{v}_{\varepsilon,n,m}+\delta  M_2) (\widetilde{t},\widetilde{\mu})(x)\Big\rangle
			\nonumber\\
			&\h{45pt}
			+\dfrac{1}{2} \textup{tr}\Big\{\pig[\sigma(\widetilde{t},x,a)\big[\sigma(\widetilde{t},x,a)\big]^\top+\sigma^0(\widetilde{t},x)[\sigma^0(\widetilde{t},x)]^\top\pig]
			\nabla_x\p_\mu(\widecheck{v}_{\varepsilon,n,m}+\delta  M_2) (\widetilde{t},\widetilde{\mu})(x)\Big\}\Bigg\}\widetilde{\mu}(dx)\nonumber\\
   				&+ \dfrac{1}{2} \int_{\mathbb{R}^{2d}}\text{tr}\Big[\sigma^0(\widetilde{t},x)[\sigma^0(\widetilde{t},y)]^\top\partial_\mu^2
    (\widecheck{v}_{\varepsilon,n,m}+\delta  M_2)(\widetilde{t},\widetilde{\mu})(x,y)\Big] \, \widetilde{\mu}^{\otimes 2}(dx,dy)\nonumber\\
			&-\int_{\mathbb{R}^{dn}}
			\sup_{\overline{a} \in A^n}\Bigg\{\dfrac{1}{n}\sum^n_{i=1}\widecheck{f}^i_{n,m}(\widetilde{t},\overline{x},a^i)       
            +\sum^n_{i=1}\Big\langle b^i_{n,m}(\widetilde{t},\overline{x},a^i),\nabla_{x^i}\widecheck{\overline{v}}_{\varepsilon,n,m}(\widetilde{t},\overline{x})\Big\rangle\nonumber\\
&\h{60pt}+\dfrac{1}{2}
\sum^n_{i=1}\textup{tr}\left[\Big((\sigma\sigma^\top)(\widetilde{t},x^i,a^i)+(\sigma^0\sigma^{0;\top})(\widetilde{t}\,,x^i)+\varepsilon^2I_d\Big)\nabla_{x^ix^i}^2
			\widecheck{\overline{v}}_{\varepsilon,n,m}
		(\widetilde{t},\overline{x})\right]\nonumber\\
&\h{60pt}+\dfrac{1}{2}
			\sum^n_{i,j=1,i\neq j}\textup{tr}\left[\sigma^0(\widetilde{t}\,,x^i)\sigma^{0;\top}(\widetilde{t}\,,x^j)\nabla_{x^ix^j}^2
			\widecheck{\overline{v}}_{\varepsilon,n,m}
			(\widetilde{t},\overline{x})\right]
\Bigg\}\bbotimes_{k=1}^n\widetilde{\mu}(dx^k).
			\label{ineq. check u1-check v_e,n,m}
		\end{align}
		We divide the estimate into three parts: the part involving $ M_2$, the part involving $\widecheck{\overline{v}}_{\varepsilon,n,m}$ and the term $(\widecheck{u}_1-\widecheck{v}_{\varepsilon,n,m})(\widetilde{t},\widetilde{\mu})$. First, by direct computation, we obtain $\partial_\mu M_2(\mu)(x) = 2x$, $\partial_\mu^2 M_2(\mu)(x,y) = 0$ and $\nabla_x\partial_\mu M_2(\mu)(x) = 2I_d$. Assumption \ref{assume:A} and \eqref{ineq. int W_2(tilde mu- hat mu^n,x)} tell us that
\begin{align}
			&\h{-10pt}\int_{\mathbb{R}^d}\sup_{a \in A}\bigg\{
			\left\langle b (\widetilde{t},x,\widetilde{\mu},a),\p_\mu  M_2(\widetilde{\mu})(x)\right\rangle
   +\dfrac{1}{2} \textup{tr}\Big\{\pig[(\sigma\sigma^\top)(\widetilde{t},x,a)+\sigma^0(\widetilde{t},x)[\sigma^0(\widetilde{t},x)]^\top
			\pig]
			\nabla_x\p_\mu  M_2 (\widetilde{\mu})(x)\Big\}\bigg\}\widetilde{\mu}(dx)\nonumber\\
+\,& \dfrac{1}{2} \int_{\mathbb{R}^{2d}}\text{tr}\Big[\sigma^0(\widetilde{t},x)[\sigma^0(\widetilde{t},y)]^\top\partial^2_\mu M_2(\widetilde{\mu})(x,y)\Big] \, \widetilde{\mu}^{\otimes 2}(dx,dy)\nonumber\\
			=\,&\int_{\mathbb{R}^d}\sup_{a \in A}\bigg\{
			\left\langle b (\widetilde{t},x,\widetilde{\mu},a),2x\right\rangle
   +\textup{tr}\pig[(\sigma\sigma^\top)(\widetilde{t},x,a)
			+\sigma^0(t,x)[\sigma^0(t,x)]^\top\pig]
			\bigg\}\widetilde{\mu}(dx)\nonumber\\
   \leq\,& \int_{\mathbb{R}^d}\Big[2K(1+|x|^{\rho})|x| + 4K^2(1+|x|^{2\rho}) \Big]\widetilde{\mu}(dx) \nonumber\\
   \leq\,& C_K\int_{\mathbb{R}^d}(1+|x|^{\rho+1})\widetilde{\mu}(dx)\nonumber\\
   \leq\,&C_K+C_K\big[M_2(\widetilde{\mu})\big]^{(1+\rho)/2}\nonumber\\
   \leq \,&C_K+\frac{C_K}{\delta^{(1+\rho)/2}}\pig[ l_0/3+e^T(\lVert u_1\rVert_{\infty} +\ell_2)\pigr]^{(1+\rho)/2},
   \label{1045}
		\end{align}
for some fixed constant $C_K$ depending only on $K$. Second, we recall the representation of $\widecheck{v}_{\varepsilon,n,m}(t,\mu)
		=e^{t-t_0}v_{\varepsilon,n,m}(t,\mu)$ from \eqref{eq. v_e,n,m=int bar of v_e,n,m}, as well as (1) and (2) of Theorem \ref{thm v_e,n,m}. Using inequalities \eqref{ineq.|D v_e,n,m|} and \eqref{ineq.|D^2 v_e,n,m|}, we can directly compute that
		\begin{align}
			\p_\mu\widecheck{v}_{\varepsilon,n,m}(t,\mu)(x)
			=\sum^n_{i=1}
			\int_{\mathbb{R}^{d(n-1)}}
			\nabla_{x^i}\widecheck{\overline{v}}_{\varepsilon,n,m}
			(t,\overline{x})\Big|_{x^i=x}
			\bbotimes^n_{k=1,k\neq i}\mu(dx^k);
		\end{align}
		and use (2) in Theorem \ref{thm v_e,n,m} to yield that
\begin{align*}
	\p^2_\mu\widecheck{v}_{\varepsilon,n,m}(t,\mu)(x,y)
   =\sum^n_{i=1}\sum^n_{j=1,j\neq i}
			\int_{\mathbb{R}^{d(n-2)}}
			\nabla_{x^ix^j}^2\widecheck{\overline{v}}_{\varepsilon,n,m}
			(t,\overline{x})\Big|_{x^i=x,x^j=y}
			\bbotimes^n_{k=1,k\neq i,j}\mu(dx^k).
\end{align*}
		Hence, we estimate the term 
	\begin{align}
			&\int_{\mathbb{R}^d}\sup_{a \in A}\bigg\{\widecheck{f}(\widetilde{t},y,\widetilde{\mu},a)
			+\dfrac{1}{2} \textup{tr}\Big\{\pig[(\sigma\sigma^\top)(\widetilde{t},y,a)+(\sigma^0\sigma^{0;\top})(\widetilde{t},y)\pig]
			\nabla_x\p_\mu\widecheck{v}_{\varepsilon,n,m}(\widetilde{t},\widetilde{\mu})(y)\Big\}\nonumber\\
			&\h{40pt}
			+ \Big\langle b (\widetilde{t},y,\widetilde{\mu},a),\p_\mu \widecheck{v}_{\varepsilon,n,m} (\widetilde{t},\widetilde{\mu})(y)\Big\rangle\Bigg\}\widetilde{\mu}(dy)\nonumber\\
   &+ \dfrac{1}{2} \int_{\mathbb{R}^{2d}}\text{tr}\Big[\sigma^0(\widetilde{t},z)[\sigma^0(\widetilde{t},y)]^\top\partial_\mu^2
    \widecheck{v}_{\varepsilon,n,m}(\widetilde{t},\widetilde{\mu})(z,y)\Big] \, \widetilde{\mu}^{\otimes 2}(dz,dy)\nonumber\\
      &=\int_{\mathbb{R}^d}\sup_{a \in A}\Bigg\{\sum^n_{j=1}\int_{\mathbb{R}^{d(n-1)}}\dfrac{1}{n}\widecheck{f}(\widetilde{t},y,\widetilde{\mu},a)\bbotimes^n_{k=1,k\neq j}\widetilde{\mu}(dx^k)\nonumber\\
			&\h{60pt}+\dfrac{1}{2} \textup{tr}\bigg\{\pig[(\sigma\sigma^\top)(\widetilde{t},y,a)+(\sigma^0\sigma^{0;\top})(\widetilde{t},y)\pig]
			\sum^n_{j=1}
			\int_{\mathbb{R}^{d(n-1)}}
			\nabla_{x^jx^j}^2\widecheck{\overline{v}}_{\varepsilon,n,m}
			(t,\overline{x})\Big|_{x^j=y}
			\bbotimes^n_{k=1,k\neq j}\widetilde{\mu}(dx^k)\bigg\}\nonumber\\
			&\h{60pt}+ \Big\langle b (\widetilde{t},y,\widetilde{\mu},a),\sum^n_{j=1}
			\int_{\mathbb{R}^{d(n-1)}}
			\nabla_{x^j}\widecheck{\overline{v}}_{\varepsilon,n,m}
			(t,\overline{x})\Big|_{x^j=y}
			\bbotimes^n_{k=1,k\neq j}\widetilde{\mu}(dx^k)\Big\rangle\Bigg\}\widetilde{\mu}(dy)\nonumber\\
   &\h{10pt}+\dfrac{1}{2} \int_{\mathbb{R}^{2d}}\textup{tr}\left\{\pig[\sigma^0(\widetilde{t},z)\sigma^{0;\top}(\widetilde{t},y)\pig]
			\sum^n_{j=1}
			\sum^n_{l=1,l\neq j}
			\int_{\mathbb{R}^{d(n-2)}}
			\nabla_{x^jx^l}^2\widecheck{\overline{v}}_{\varepsilon,n,m}
			(t,\overline{x})\Big|_{x^j=z,x^l=y}
			\bbotimes^n_{k=1,k\neq j,l}\widetilde{\mu}(dx^k)\right\}\widetilde{\mu}^{\otimes 2}(dz,dy)\nonumber\\
         &\leq\sum^n_{j=1}\int_{\mathbb{R}^{dn}}\sup_{a^j \in A}\Bigg\{\dfrac{1}{n}\widecheck{f}(\widetilde{t},x^j,\widetilde{\mu},a^j)
			+\dfrac{1}{2} \textup{tr}\bigg\{\pig[(\sigma\sigma^\top)(\widetilde{t},x^j,a^j)+(\sigma^0\sigma^{0;\top})(\widetilde{t},x^j)\pig]
			\nabla_{x^jx^j}^2\widecheck{\overline{v}}_{\varepsilon,n,m}
			(t,\overline{x})\bigg\}\nonumber\\
			&\h{80pt}+ \Big\langle b (\widetilde{t},x^j,\widetilde{\mu},a^j),
			\nabla_{x^j}\widecheck{\overline{v}}_{\varepsilon,n,m}
			(t,\overline{x})
			\Big\rangle\Bigg\}\bbotimes^n_{k=1}\widetilde{\mu}(dx^k)\nonumber\\
  &\h{10pt}+\dfrac{1}{2}\sum^n_{j=1}
			\sum^n_{l=1,l\neq j} \int_{\mathbb{R}^{dn}}\textup{tr}\left\{\pig[\sigma^0(\widetilde{t},x^j)\sigma^{0;\top}(\widetilde{t},x^l)\pig]
			\nabla_{x^jx^l}^2\widecheck{\overline{v}}_{\varepsilon,n,m}
			(t,\overline{x})
			\right\}\bbotimes^n_{k=1}\widetilde{\mu}(dx^k).
			\label{1941}
		\end{align}
  Third, as $\widecheck{u}_1- \widecheck{v}_{\varepsilon,n,m}-\delta  M_2$ attains its maximum over the space $[0,T] \times \mathcal{P}_2(\mathbb{R}^d)$ at $(\widetilde{t},\widetilde{\mu})$, it holds that
\begin{align}
	(u_1- v_{\varepsilon,n,m}-\delta  M_2)(t_0,\mu_0) = (\widecheck{u}_1- \widecheck{v}_{\varepsilon,n,m}-\delta  M_2)(t_0,\mu_0)
			&\leq \widecheck{u}_1(\widetilde{t},\widetilde{\mu}) - \widecheck{v}_{\varepsilon,n,m}(\widetilde{t},\widetilde{\mu})
			-\delta  M_2(\widetilde{\mu})\nonumber\\
   &\leq \widecheck{u}_1(\widetilde{t},\widetilde{\mu}) - \widecheck{v}_{\varepsilon,n,m}(\widetilde{t},\widetilde{\mu}).
			\label{check u_1 - check v_e,n,m at t_0< at tilde t }
		\end{align}		
After substituting \eqref{1045}, \eqref{1941} and \eqref{check u_1 - check v_e,n,m at t_0< at tilde t } into \eqref{ineq. check u1-check v_e,n,m}, we make use of \eqref{ineq.|D v_e,n,m|} and \eqref{ineq.|D^2 v_e,n,m|} to deduce that
\begin{align}
			&\h{-10pt}(u_1- v_{\varepsilon,n,m})(t_0,\mu_0)
   -\delta M_{2}(\mu_0)\nonumber\\
			\leq\,&
			\delta\left[\frac{C_K}{\delta^{(1+\rho)/2}}\pig[ l_0/3+e^T(\lVert u_1\rVert_{\infty} +\ell_2)\pigr]^{(1+\rho)/2} +C_K\right]\nonumber\\
&+\int_{\mathbb{R}^{dn}}
			\sum^n_{i=1}\sup_{a^i \in A}
			\Bigg\{\dfrac{1}{n}
			\widecheck{f}(\widetilde{t},x^i,\widetilde{\mu},a^i)
			-\dfrac{1}{n}\widecheck{f}^i_{n,m}(\widetilde{t},\overline{x},a^i)
			-\dfrac{\varepsilon^2}{2}\textup{tr}\nabla_{x^ix^i}^2
			\widecheck{\overline{v}}_{\varepsilon,n,m}
			(\widetilde{t},\overline{x})\nonumber\\
			&\h{130pt}+\left\langle b (\widetilde{t},x^i,\widetilde{\mu},a^i)-b^i_{n,m}(\widetilde{t},\overline{x},a^i),\nabla_{x^i}\widecheck{\overline{v}}_{\varepsilon,n,m}(\widetilde{t},\overline{x})\right\rangle
			\Bigg\}\bbotimes^n_{k=1}\widetilde{\mu}(dx^k)\nonumber\\
			\leq\,&
			\delta\left[\frac{C_K}{\delta^{(1+\rho)/2}}\pig[ l_0/3+e^T(\lVert u_1\rVert_{\infty} +\ell_2)\pigr]^{(1+\rho)/2} + C_K\right]
			\nonumber\\
			&
			+\int_{\mathbb{R}^{dn}}
			\sum^n_{i=1}\sup_{a^i \in A}
			\Bigg\{\dfrac{e^{\widetilde{t}-t_0}}{n}
			\pig|f(\widetilde{t},x^i,\widetilde{\mu},a^i)
			-f^i_{n,m}(\widetilde{t},\overline{x},a^i)\pig|
			-\dfrac{\varepsilon^2}{2}\textup{tr}\nabla_{x^ix^i}^2
			\widecheck{\overline{v}}_{\varepsilon,n,m}
			(\widetilde{t},\overline{x})\nonumber\\
			&\h{130pt}+\dfrac{C_4e^{\widetilde{t}-t_0}}{n}\pig|b (\widetilde{t},x^i,\widetilde{\mu},a^i)-b^i_{n,m}(\widetilde{t},\overline{x},a^i)\pig|
			\Bigg\}\bbotimes^n_{k=1}\widetilde{\mu}(dx^k).
			\label{1271}
		\end{align}
		We use Assumption \ref{assume:A} and (2) of Lemma \ref{lem estimate of b^i_n,m...} to estimate the following term:
		\begin{align}
			&\h{-10pt}\pig|f(\widetilde{t},x^i,\widetilde{\mu},a^i)
			-f^i_{n,m}(\widetilde{t},\overline{x},a^i)\pig|
			+C_4\pig|b (\widetilde{t},x^i,\widetilde{\mu},a^i)-b^i_{n,m}(\widetilde{t},\overline{x},a^i)\pig|\nonumber\\
			\leq\,&
			\pig|f(\widetilde{t},x^i,\widetilde{\mu},a^i)
			-f (\widetilde{t},x^i,\widehat{\mu}^{n,\overline{x}},a^i)\pig|+\pig|f (\widetilde{t},x^i,\widehat{\mu}^{n,\overline{x}},a^i)
			-f^i_{n,m}(\widetilde{t},\overline{x},a^i)\pig|\nonumber\\
			&+C_4\pig|b(\widetilde{t},x^i,\widetilde{\mu},a^i)
			-b (\widetilde{t},x^i,\widehat{\mu}^{n,\overline{x}},a^i)\pig|
			+C_4\pig|b (\widetilde{t},x^i,\widehat{\mu}^{n,\overline{x}},a^i)
			-b^i_{n,m}(\widetilde{t},\overline{x},a^i)\pig|\nonumber\\
			\leq\,&
			K(1+C_4)\mathcal{W}_1(\widetilde{\mu},\widehat{\mu}^{n,\overline{x}})
   +K(1+C_4)m\int_{\mathbb{R}}\left|\widetilde{t}-\pig[T\wedge(\widetilde{t}-s)^+\pig]\right|^\beta \phi(ms)ds
			\nonumber\\
			&+K(1+C_4)m^{dn}\int_{\mathbb{R}^{dn}}
			\left(|y^i|+\dfrac{1}{n}\sum^n_{j=1}|y^j|\right) \prod^n_{k=1}\Phi(my^k)dy^k.
			\label{1291}
		\end{align}
		Putting \eqref{1291} and \eqref{ineq. int W_2(tilde mu- hat mu^n,x)} into \eqref{1271}, we see that
\begin{align*}
			&\h{-10pt}(u_1-v_{\varepsilon,n,m}-\delta  M_2)(t_0,\mu_0)\nonumber\\
			\leq\,&
			\delta\left[\frac{C_K}{\delta^{(1+\rho)/2}}\pig[ l_0/3+e^T(\lVert u_1\rVert_{\infty} +\ell_2)\pigr]^{(1+\rho)/2} + C_K\right]
			\nonumber\\
   &-\int_{\mathbb{R}^{dn}}
			\dfrac{\varepsilon^2}{2}
			\sum^n_{i=1}\textup{tr}\nabla_{x^ix^i}^2
			\widecheck{\overline{v}}_{\varepsilon,n,m}
			(\widetilde{t},\overline{x})\bbotimes^n_{k=1}\widetilde{\mu}(dx^k)
			+\dfrac{e^{T-t_0}C_dh_nK(1+C_4)}{\delta^{1/2}}\pig[ l_0/3+e^T(\lVert u_1\rVert_{\infty} +\ell_2)\pigr]^{1/2}\nonumber\\
			&+e^{T-t_0}K(1+C_4)m\int_{\mathbb{R}}\left|\widetilde{t}-\pig[T\wedge(\widetilde{t}-s)^+\pig]\right|^\beta \phi(ms)ds\\
			&+\dfrac{2e^{T-t_0}K(1+C_4)m^{dn}}{n}\int_{\mathbb{R}^{dn}}
			\left(\sum^n_{j=1}|y^j|\right) \prod^n_{k=1}\Phi(my^k)dy^k.
		\end{align*}
		Using Lemmas \ref{lem. classical sol. of smooth approx.} and \ref{lem. |v_e,n,m-v_0,n,m| < C_6e}, we first take $\varepsilon \to 0^+$ and then $m \to \infty$ to obtain that
\begin{align*}
			&\h{-10pt}(u_1-\lim_{m\to \infty}v_{0,n,m}-\delta  M_2)(t_0,\mu_0)\nonumber\\
			\leq\,&
		\delta\left[\frac{C_K}{\delta^{(1+\rho)/2}}\pig[ l_0/3+e^T(\lVert u_1\rVert_{\infty} +\ell_2)\pigr]^{(1+\rho)/2} + C_K\right]
			\nonumber\\
			&+\dfrac{e^{T-t_0}C_dh_nK(1+C_4)}{\delta^{1/2}}\pig[ l_0/3+e^T(\lVert u_1\rVert_{\infty} +\ell_2)\pigr]^{1/2}.
		\end{align*}
		By \eqref{def. h_n}, \eqref{ineq. int W_2(tilde mu- hat mu^n,x)}, Lemmas \ref{lem. |v_e-v_0|<C_5 e} and \ref{lem. conv. of v_e,n,m to v_e}, we take $n \to \infty$ and then $\delta \to 0^+$ to obtain that
		\begin{align*}
			(u_1-v_{0})(t_0,\mu_0)
			=\left(u_1-\lim_{n\to \infty}\lim_{m\to \infty}v_{0,n,m}\right)(t_0,\mu_0)
			\leq 0,
		\end{align*}
		which contradicts \eqref{ineq. u_1-v_0(t_0,mu_0)>0}.\\
  \hfill\\
 \noindent{\bf Part 2. Proof of  $u_2\geq v_{0}$:} Following the arguments of the first part of Step II of the proof of \cite[Theorem 5.1]{cosso_master_2022},  we can assume without loss of generality that $u_2(s,\cdot)$ is $\mathcal{W}_1$-Lipschitz continuous for every $s\in [0,T]$; and showing $u_2 \geq v_0$ is equivalent to showing 
\begin{equation}
    u_2(t, \mu) \geq v^s(t, \nu):= \mathbb{E} \left[ \int_t^s f \left( r, X_r^{t, \xi, \mathfrak{a}}, \mathbb{P}_{X_r^{t, \xi, \mathfrak{a}}}^{W^0}, \mathfrak{a} \right) dr \right] + \mathbb{E}u_2 \left( s, \mathbb{P}_{X_s^{t, \xi, \mathfrak{a}}}^{W^0} \right),
    \label{1446}
\end{equation}
for every $(t, \mu) \in [0, T] \times \mathcal{P}_2(\mathbb{R}^d)$, $s \in (t, T]$, $\xi \in L^2(\Omega, \mathcal{F}_t, \mathbb{P}; \mathbb{R}^d)$ with $\mathcal{L}(\xi) = \mu$, and $\mathfrak{a} \in \mathcal{M}_t$, where $\mathcal{M}_t$ denotes the set of $\mathcal{F}_t^t$-measurable random variables $\alpha: \Omega \to A$ and $\nu:=\mathcal{L}(\xi,\mathfrak{a})$. Suppose, for contradiction, that \eqref{1446} does not hold. Then, there exist $t_0 \in [0, T)$, $s_0 \in (t_0, T]$, $\mu_0 \in \mathcal{P}_2(\mathbb{R}^d)$ and $\nu_0 \in \mathcal{P}_2(\mathbb{R}^d \times A)$, with $\mu_0$ being the marginal of $\nu_0$ on $\mathbb{R}^d$, such that
\begin{align}
v^{s_0}(t_0, \nu_0) > u_2(t_0, \mu_0).
\label{1450}
\end{align}
Following the approach outlined at the beginning of Part 1 of this proof, we assume that there exists $q > 2$ such that $\nu_0 \in \mathcal{P}_q(\mathbb{R}^d\times A)$. The function $\widecheck{u}_2(t, \mu) := e^{t-t_0} u_2(t, \mu)$ is a viscosity supersolution of the following equation:
\begin{align*}
			\left\{\begin{aligned}
				&\partial_t u(t,\mu)
				+\int_{\mathbb{R}^d}
				\sup_{a\in A} \Bigg\{\widecheck{f}(t,x,\mu,a)+b(t,x,\mu,a)\cdot \partial_\mu u(t,\mu)(x) \\
				&\h{90pt}+ \dfrac{1}{2}\textup{tr}\Big((\sigma(t,x,a)\big[\sigma(t,x,a)\big]^\top+\sigma^0(t,x)[\sigma^0(t,x)]^\top)\nabla_x\partial_\mu u(t,\mu)(x)\Big)\Bigg\}\mu(dx)\\
				&+ \dfrac{1}{2} \int_{\mathbb{R}^{2d}}\text{tr}\Big[\sigma^0(t,x)[\sigma^0(t,y)]^\top\partial_\mu^2
    u(t,\mu)(x,y)\Big] \, \mu^{\otimes 2}(dx,dy)-u(t,\mu)=0\h{5pt} \text{for $(t,\mu) \in [0,T) \times\mathcal{P}_2(\mathbb{R}^{d})$};\\	&u(T,\mu)=\int_{\mathbb{R}^d}\widecheck{g}(x,\mu)\mu(dx) \h{10pt}\text{for $\mu\in \mathcal{P}_2(\mathbb{R}^d)$},
			\end{aligned}\right.
		\end{align*}
where $\widecheck{f}(t,x,\mu,a):=e^{t-t_0}f(t,x,\mu,a)$ and $\widecheck{g}(x,\mu):=e^{T-t_0}g(x,\mu)$. That is, for any $s_1\in (0,T]$ and $\varphi \in C^{1,2}([0,s_1]\times\mathcal{P}_2(\mathbb{R}^d\times A))$ such that $\widecheck{u}_2-\varphi$ attains a minimum with a value of $0$ at $(t^*,\nu^*) \in [0,s_1)\times \mathcal{P}_2(\mathbb{R}^d\times A)$, then the following inequality holds:
\begin{align}
    0\geq\,&\partial_t \varphi(t^*,\nu^*)+\int_{\mathbb{R}^d\times A}
				 \Bigg\{\widecheck{f}(t^*,x,\mu^*,a)+b(t^*,x,\mu^*,a)\cdot \partial_\mu \varphi(t^*,\nu^*)(x,a) \nonumber\\
				&+ \dfrac{1}{2}\textup{tr}\Big[(\sigma(t^*,x,a)\big[\sigma(t^*,x,a)\big]^\top+\sigma^0(t^*,x)[\sigma^0(t^*,x)]^\top)\nabla_x\partial_\mu \varphi(t^*,\nu^*)(x,a)\Big]\Bigg\}\nu^*(dx,da)\nonumber\\
    &+ \dfrac{1}{2} \int_{\mathbb{R}^d \times A \times \mathbb{R}^d \times A}\text{tr}\Big[\sigma^0(t^*,x)[\sigma^0(t^*,y)]^\top\partial_\mu^2
    \varphi(t^*,\nu^*)(x,a,y,\alpha)\Big] \, (\nu^*)^{\otimes 2}(dx,da,dy,d\alpha)
    -\varphi(t^*,\nu^*),
    \label{2214}
		\end{align}
  where $\mu^*$ is the marginal of $\nu^*$ on $\mathbb{R}^d$, and the operators $\partial_\mu$, $\partial_\mu^2$ are defined as in Definition \ref{vis_def}, through the projection.
  
  We regularize the coefficients with respect to the control variable. Let $\Psi:\mathbb{R}^d \to \mathbb{R}^+$ be a compactly supported smooth function satisfying $\int_{\mathbb{R}^{d}}\Psi(y)dy=1$. We extend \(b\) and \(f\) to the space \([0, T] \times \mathbb{R}^{d}\times\mathcal{P}_2(\mathbb{R}^d)\times \mathbb{R}^d\) by setting $b(t,x,\mu,a)=0$ and $f(t,x,\mu,a)=0$ when $a\in\mathbb{R}^d$ is not in $A$. For simplicity, we continue to denote these extensions by \(b\) and \(f\). We further define the functions \(\widetilde{b}^{\h{.7pt}i}_{n,m}\) and \(\widetilde{f}^{\h{.7pt}i}_{n,m}\) by
\[
\widetilde{b}^{\h{.7pt}i}_{n,m}(t, \overline{x}, a) := m^d \int_{\mathbb{R}^d} b^i_{n,m}(t, \overline{x}, a - a') \Psi(m a') \, da',
\]
\[
\widetilde{f}^{\h{.7pt}i}_{n,m}(t, \overline{x}, a) := m^d \int_{\mathbb{R}^d} f^i_{n,m}(t, \overline{x}, a - a') \Psi(m a') \, da',
\]
for any $n, m \in \mathbb{N}$, $i = 1,2, \ldots, n$, $\overline{x} = (x^1,x^2, \ldots, x^n) \in \mathbb{R}^{dn}$ and $(t, a) \in [0, T] \times A$. Here $b_{n,m}^i$ and $f_{n,m}^i$ are as defined in \eqref{def. approx of b} and \eqref{def. approx of f}. Recalling the compactly supported smooth function $\Phi$ defined in Section \ref{Phi_defined}, we also define 
\begin{equation*}
u_{n,m}(t, \overline{x}) := m^{dn} \int_{\mathbb{R}^{dn}} u_2 \left(t, \frac{1}{n} \sum_{j=1}^{n} \delta_{x^j - y^j} \right) \prod_{j=1}^{n} \Phi(my^j) \, dy^j.
\end{equation*} We now introduce
\begin{align}
\label{v^s_n,m}
v^{s_0}_{n,m}(t, \nu) :=& \frac{1}{n} \sum_{i=1}^{n} \mathbb{E} \bigg[ \int_{t}^{s_0} \widetilde{f}^{\h{.7pt}i}_{n,m} \left( r, \overline{\widetilde{X}}^{1,m,t,\overline{\xi},\overline{\mathfrak{a}}_0}_{r}, \ldots, \overline{\widetilde{X}}^{n,m,t,\overline{\xi},\overline{\mathfrak{a}}_0}_{r}, \mathfrak{a}_0^i \right) dr\nonumber\\
&\h{110pt}+ u_{n,m} \left( s_0, \overline{\widetilde{X}}^{1,m,t,\overline{\xi},\overline{\mathfrak{a}}_0}_{s_0}, \ldots, \overline{\widetilde{X}}^{n,m,t,\overline{\xi},\overline{\mathfrak{a}}_0}_{s_0} \right) \bigg],
\end{align}
for any $t \in [0,s_0]$ and $\nu \in \mathcal{P}_2(\mathbb{R}^d \times A)$, where $\overline{\xi}=(\xi^1,\xi^2,\ldots,\xi^n) \in L^2(\Omega, \mathcal{F}_t, \mathbb{P}; \mathbb{R}^{dn})$, $\overline{\mathfrak{a}}_0=(\mathfrak{a}_0^1,\mathfrak{a}_0^2,\ldots,\mathfrak{a}_0^n) \in (\mathcal{M}_t)^n$ such that $\mathcal{L}(\overline{\xi},\overline{\mathfrak{a}}_0)=\nu \otimes \cdots \otimes \nu$ and $\overline{\widetilde{X}}^{m,t,\overline{\xi},\overline{\mathfrak{a}}_0}_s=\Big(\overline{\widetilde{X}}^{1,m,t,\overline{\xi},\overline{\mathfrak{a}}_0}_s,\ldots,\overline{\widetilde{X}}^{n,m,t,\overline{\xi},\overline{\mathfrak{a}}_0}_s\Big)$ is the solution to \eqref{eq. state perturbed by e BM and smoothing} on $[t,s_0]$ with $\overline{\alpha}=\overline{\mathfrak{a}}_0$, $\varepsilon=0$ and $b$ replaced by $\widetilde{b}^{\h{.7pt}i}_{n,m}$. For every $n, m \in \mathbb{N}$, $(t, \nu) \in [0, s_0] \times \mathcal{P}_2(\mathbb{R}^d\times A)$, we define $\widecheck{v}_{n,m}^{s_0}:=e^{t-t_0} v_{n,m}^{s_0}$ and also similarly define  $\widecheck{\widetilde{f}^{\h{.7pt}i}}_{n,m}$ and $\widecheck{u}_{n,m}$ from $\widetilde{f}_{n,m}^i$ and $u_{n,m}$, respectively. Let
\begin{align*}
\overline{v}_{n,m}^{s_0}(t,\overline{x},\overline{a})
:=\widetilde{v}_{n,m}^{s_0}
(t,\delta_{(x^1,a^1)}\otimes\cdots\otimes\delta_{(x^n,a^n)})
		\end{align*}
 for any $\overline{x}=\left(x^1, \ldots, x^n\right) \in \mathbb{R}^{dn}$ and $\overline{a}=\left(a^1, \ldots, a^n\right) \in A^n$, where
\begin{align}
		\widetilde{v}_{n,m}^{s_0}(t,\overline{\nu})
		:=& \dfrac{1}{n}\sum^n_{i=1}\mathbb{E}\Bigg[\int_t^{s_0} \widetilde{f}^{\h{.7pt}i}_{n,m}\left(s,\overline{\widetilde{X}}^{1,m,t,\overline{\xi},\overline{\mathfrak{a}}_0}_s,\ldots,
		\overline{\widetilde{X}}^{n,m,t,\overline{\xi},\overline{\mathfrak{a}}_0}_s,
		\mathfrak{a}_0^i\right)ds\nonumber \\
		&\h{110pt}+ u_{n,m}\left(\overline{\widetilde{X}}^{1,m,t,\overline{\xi},\overline{\mathfrak{a}}_0}_{s_0},\ldots,
		\overline{\widetilde{X}}^{n,m,t,\overline{\xi},\overline{\mathfrak{a}}_0}_{s_0}\right)\Bigg],
		\label{def. tilde v_e,n,m, with fixed control}
	\end{align}
 for any $t\in[0,s_0]$ and $\overline{\nu} \in \mathcal{P}_2\left(\mathbb{R}^{dn} \times A^n\right)$, where $\overline{\xi}$ and $\overline{\mathfrak{a}}_0$ satisfy $\mathcal{L}(\overline{\xi},\overline{\mathfrak{a}}_0)=\overline{\nu}$. Moreover, by \cite[Theorem A.8]{cosso_master_2022}, we deduce that $\widecheck{v}_{n,m}^{s_0}$ 
can be represented by 
\begin{align*}
\widecheck{v}_{n,m}^{s_0}(t,\nu)
=e^{t-t_0}\int_{\mathbb{R}^{dn}\times A^n}
\overline{v}_{n,m}^{s_0}(t,\overline{x},\overline{a})\nu(dx^1,da^1)\otimes\cdots\otimes
\nu(dx^n,da^n),
\end{align*}
and it could be shown by following the proofs of Lemma \ref{lem. classical sol. of smooth approx.}, Theorem \ref{thm v_e,n,m} and \cite[Theorem A.8]{cosso_master_2022} that 
\begin{enumerate}[(1).]
    \item $\overline{v}_{n, m}^{s_0} \in C^{1,2}\left(\left[0, s_0\right] \times (\mathbb{R}^{dn} \times A^n) \right)$ and $v_{n, m}^{s_0} \in C^{1,2}\left(\left[0, s_0\right] \times \mathcal{P}_2\left(\mathbb{R}^d \times A\right)\right)$;
    \item for any $i=1, \ldots, n$ and $(t,\overline{x},\overline{a}) \in\left[0, s_0\right] \times\mathbb{R}^{dn} \times A^n$, it holds that
$$
\left|\nabla_{x^i} \overline{v}_{n, m}^{s_0}(t,\overline{x},\overline{a})\right| \leq \frac{C_K}{n},
$$
where the constant $C_K \geq 0$ depends on $d$, $K$, $T$, but independent of $n, m$;
\item if $t \in [0,s_0]$ and $\nu \in \mathcal{P}_q\left(\mathbb{R}^d \times A\right)$ for some $q>2$,
then
\begin{align}
\label{v^s_n,m conv}
\lim _{n \rightarrow+\infty} \lim _{m \rightarrow+\infty} v_{n, m}^{s_0}(t, \nu)=v^{s_0}(t, \nu);
\end{align}
\item the function $v_{n,m}^{s_0}(t,\nu)$ solves the following equation classically:
\[
\left\{
\begin{aligned}
    &\partial_t u(t, \nu) + \overline{\mathbb{E}} \Bigg[ \sum_{i=1}^n \bigg\{ \frac{1}{n} \widetilde{f}^{\h{.7pt}i}_{n,m}(t, \overline{\xi}, \mathfrak{a}_0^i) 
    + \Big\langle\widetilde{b}_{n,m}^i (t, \overline{\xi}, \mathfrak{a}_0^i), \nabla_{x^i} \overline{v}_{n,m}^{s_0}(t, \overline{\xi}, \overline{\mathfrak{a}}_0)\Big\rangle  \\
    &\h{90pt} + \frac{1}{2} \text{tr} \left(\left[ (\sigma \sigma^\top) (t, \xi^i, \mathfrak{a}_0^i)+(\sigma^0\sigma^{0;\top})(t,\xi^i)\right]  \nabla_{x^ix^i}^2 \overline{v}_{n,m}^{s_0}(t, \overline{\xi}, \overline{\mathfrak{a}}_0)\right) \bigg\}\\
    &\h{60pt}+\dfrac{1}{2}
				\sum^n_{i,j=1,i\neq j}\textup{tr}\Big[\sigma^0(t,\xi^i)\sigma^{0;\top}(t,\xi^j)\nabla_{x^ix^j}^2 \overline{v}_{n,m}^{s_0}(t, \overline{\xi}, \overline{\mathfrak{a}}_0)\Big]\Bigg] =0; \\
    &u(s_0, \nu) = \overline{\mathbb{E}}  [u_{n,m}(s_0, \overline{\xi})],
\end{aligned}
\right.
\]
for any $t \in [0,s_0)$ and $\nu \in \mathcal{P}_2(\mathbb{R}^d \times A)$, where $\overline{\xi}=(\xi^1,\xi^2,\ldots,\xi^n) \in L^2(\Omega, \mathcal{F}_t, \mathbb{P}; \mathbb{R}^{dn})$, $\overline{\mathfrak{a}}_0=(\mathfrak{a}_0^1,\mathfrak{a}_0^2,\ldots,\mathfrak{a}_0^n) \in (\mathcal{M}_t)^n$ such that $\mathcal{L}(\overline{\xi},\overline{\mathfrak{a}}_0)=\nu \otimes \ldots \otimes \nu$.
\end{enumerate}
We now return to the hypothesis \eqref{1450}, with $\mu_0$ and $\nu_0$ mentioned therein. Let $l_0:=v^{s_0}(t_0,\nu_0)-u_2(t_0,\mu_0)>0$. For small enough $\delta>0$ depending on $M_{2}(\nu_0)$ and $l_0$ only, we have $v^{s_0}(t_0,\nu_0)-u_2(t_0,\mu_0)-\delta M_{2}(\nu_0)\geq l_0/2$ and thus $v^{s_0}_{n,m}(t_0,\nu_0)-u_2(t_0,\mu_0)-\delta M_{2}(\nu_0)\geq l_0/3$, where $v^{s_0}_{n,m}$ is defined in \eqref{v^s_n,m}, and this holds for large enough $n$, $m\in \mathbb{N}$ depending on $\delta$, $M_{2}(\nu_0)$ and $l_0$ only, by \eqref{v^s_n,m conv}. Thus, it holds that 
\begin{equation}\label{eq:cpt1}
\displaystyle\sup_{(t,\nu)\in [0,T] \times \mathcal{P}_2(\mathbb{R}^d\times A)} \widecheck{v}^{s_0}_{n,m}(t,\nu)-\widecheck{u}_2(t,\mu)-\delta M_{2}(\nu)\geq l_0/3,
\end{equation}
where $\mu$ is the marginal of $\nu$ on $\mathbb{R}^d$, for large enough $n$, $m\in \mathbb{N}$ depending on $\delta$, $M_{2}(\nu_0)$ and $l_0$ only. By the compactness of \begin{align*}
\widetilde{U}^2_{\delta, n,m}:=\pig\{(t,\nu) \in [0,T]\times\mathcal{P}_2(\mathbb{R}^d \times A):\widecheck{v}^{s_0}_{n,m}(t,\nu)-\widecheck{u}_2(t,\mu)-\delta M_{2}(\nu) \geq l_0/3\pig\},
\end{align*}
and the upper semicontinuity of $\widecheck{v}^{s_0}_{n,m}-\widecheck{u}_2-\delta M_{2}$, we argue as in Part 1A to deduce that there exists $(\widetilde{t},\widetilde{\nu}) \in [0,T] \times \widetilde{U}^2_{\delta, n,m}$ such that the maximum of $\widecheck{v}^{s_0}_{n,m}-\widecheck{u}_2-\delta M_{2}$ is attained at $(\widetilde{t},\widetilde{\nu})$. We note that this maximum point depends on $\delta, n,m$. If $\widetilde{t} = s_0=T$, then we proceed as in Step 1B to get a contradiction. If $\widetilde{t} = s_0<T$, then \begin{align*}
	(v^{s_0}_{n,m}-u_2-\delta  M_2)(t_0,\nu_0) = (\widecheck{v}^{s_0}_{n,m}-\widecheck{u}_2-\delta  M_2)(t_0,\nu_0)
			&\leq  \widecheck{v}^{s_0}_{n,m}(s_0,\widetilde{\nu})-\widecheck{u}_2(s_0,\widetilde{\nu}) 
			-\delta  M_2(\widetilde{\nu})\nonumber\\
   &\leq \widecheck{v}^{s_0}_{n,m}(s_0,\widetilde{\nu})-\widecheck{u}_2(s_0,\widetilde{\nu})\to 0,
\end{align*}
as $m \to \infty$ then $n\to \infty$, by using the definition of $\widecheck{v}_{n,m}^{s_0}:=e^{t-t_0} v_{n,m}^{s_0}$ and that of $v_{n,m}^{s_0}$ in \eqref{v^s_n,m}. If $\widetilde{t}<s_0$, we apply Definition \ref{def. of vis sol} of supersolution and put $\varphi=\widecheck{v}_{n,m}^{s_0}-\delta M_2-M_*$ in \eqref{2214}, with $M_* \in \mathbb{R}$ such that $ \widecheck{v}_{n,m}^{s_0}-\widecheck{u}_2 -\delta M_2-M_*$ attains the maximum with a value of $0$ at $(\widetilde{t},\widetilde{\nu})$. Then we proceed as in Step 1C and utilize items (1)-(4) in the above to draw a contradiction and conclude the proof.
\end{proof}

\appendix
\section{Technical Proofs in Section \ref{approximation}}
\label{Appendix_proof_Section_3}
\begin{proof}[\textup{\bf Proof of Lemma \ref{lem estimate of b^i_n,m...}: }]
The bounds of $f^i_{n,m}$ and $g^i_{n,m}$ in assertion (1) are obvious by the definitions in \eqref{def. approx of b}-\eqref{def. approx of g} and Assumption \ref{assume:A}. For the bound of $b^i_{n,m}$ in assertion (1), we have
\begin{align*}
|b^i_{n,m}(t,\overline{x},a)|&\leq Km^{dn}\int_{\mathbb{R}^{dn}}
[1+|x^i|^\rho+|y^i|^\rho]\prod^n_{k=1}\Phi(my^k)dy^k\\
&\leq K\left(1+|x^i|^\rho+m^{-\rho}\int_{\mathbb{R}^{d}}
|y^i|^\rho\Phi(y^i)dy^i\right)\\
&\leq  K\pig(1+C_{\Phi,\rho}m^{-\rho}+|x^i|^\rho\pig).
\end{align*}

We prove assertion (2) for $g$ by considering
\begin{align*}
\pig|g(x^i, \widehat{\mu}^{n,\overline{x}}) - g^i_{n,m} (\overline{x})\pig|
\leq m^{dn} \int_{\mathbb{R}^{dn}} \bigg| g(x^i, \widehat{\mu}^{n,\overline{x}}) 
- g \bigg( x^i -y^i, \frac{1}{n} \sum_{j=1}^{n} \delta_{x^j-y^j} \bigg) \bigg| \prod^n_{k=1}\Phi(my^k) dy^k.
\end{align*}
Using the Lipschitz continuity of $g$ in Assumption \ref{assume:A} and the fact that
\begin{align}
\mathcal{W}_1 \bigg( \widehat{\mu}^{n,\overline{x}},  \frac{1}{n} \sum_{j=1}^{n} \delta_{x^j-y^j} \bigg) 
&= \mathcal{W}_1  \bigg( \frac{1}{n} \sum_{j=1}^{n} \delta_{x^j}, \frac{1}{n} \sum_{j=1}^{n} \delta_{x^j - y^j} \bigg) \nonumber\\
&\leq \int_{\mathbb{R}^d\times \mathbb{R}^d} |x-y|
\left[\frac{1}{n} \sum_{j=1}^{n} \delta_{(x^j,x^j - y^j)}(dx,dy)\right] \nonumber\\
&= \frac{1}{n} \sum_{j=1}^{n} |y^j|,
\label{1701}
\end{align}
we obtain that
\begin{align*}
\pig|g(x^i, \widehat{\mu}^{n,\overline{x}}) - g^i_{n,m} (\overline{x})\pig|
\leq K m^{dn} \int_{\mathbb{R}^{dn}} \bigg( |y^i| + \frac{1}{n} \sum_{j=1}^{n} |y^j| \bigg) \prod_{k=1}^{n} \Phi(my^k) dy^k.
\end{align*}
We prove assertion (2) for $b$ by considering (the proof for $f$ is exactly the same)
\begin{align*}
&|b(t, x^i, \widehat{\mu}^{n,\overline{x}}, a) - b_{n,m}^i (t, \overline{x}, a)| \\
&\leq m^{dn+1} \int_{\mathbb{R}^{dn+1}} \left| b(t, x^i, \widehat{\mu}^{n,\overline{x}}, a) - b \bigg( T \wedge (t-s)^+, x^i - y^i, \frac{1}{n} \sum_{j=1}^{n} \delta_{x^j - y^j}, a \bigg) \right| \phi(ms) \prod_{k=1}^{n} \Phi(my^k) dy^k ds\\
&\leq m \int_{\mathbb{R}} \left| b(t, x^i, \widehat{\mu}^{n,\overline{x}}, a) 
- b \left( T \wedge (t-s)^+, x^i, \widehat{\mu}^{n,\overline{x}}, a \right) \right| \phi(ms) ds \\
&\h{10pt}+ m^{dn+1} \int_{\mathbb{R}^{dn+1}} \left| b \left( T \wedge (t-s)^+, x^i, \widehat{\mu}^{n,\overline{x}}, a \right) - b \bigg( T \wedge (t-s)^+, x^i - y^i, \frac{1}{n} \sum_{j=1}^{n} \delta_{x^j - y^j}, a \bigg) \right|\cdot\\
&\h{335pt}\phi(ms) \prod_{j=1}^{n} \Phi(my^j) dy^j ds.
\end{align*}
The inequality in \eqref{1701} and Assumption \ref{assume:A} imply that
\begin{align*}
&|b(t, x^i, \widehat{\mu}^{n,\overline{x}}, a) - b_{n,m}^i (t, \overline{x}, a)| \\
&\leq Km \int_{\mathbb{R}} \left|t- (T \wedge (t-s)^+) ) \right|^\beta \phi(ms) ds 
+ K m^{dn} \int_{\mathbb{R}^{dn}} \bigg( |y^i| + \frac{1}{n} \sum_{j=1}^{n} |y^j| \bigg) \prod_{k=1}^{n} \Phi(my^k) dy^k.
\end{align*}
For the proof of assertion (3) for $g$ (the proofs for $f$ and $b$ are exactly the same), we let $\overline{x}$, $\overline{z} \in \mathbb{R}^{dn}$ and estimate
\begin{align*}
&|g_{n,m}^i (\overline{x}) - g_{n,m}^i (\overline{z})| \\
&\leq m^{dn} \int_{\mathbb{R}^{dn}} \bigg| g \bigg( x^i-y^i, \frac{1}{n} \sum_{j=1}^{n} \delta_{x^j-y^j} \bigg) 
- g \bigg( z^i-y^i, \frac{1}{n} \sum_{j=1}^{n} \delta_{z^j-y^j} \bigg) \bigg|  \prod_{k=1}^{n} \Phi(my^k) dy^k.
\end{align*}
Then the inequality in \eqref{1701}  yields that
\begin{align*}
|g_{n,m}^i (\overline{x}) - g_{n,m}^i (\overline{z})| 
&\leq K m^{dn} \int_{\mathbb{R}^{dn}}\bigg[ |x^i - z^i| + \frac{1}{n} \sum_{j=1}^{n} |x^j - z^j| \bigg]\prod_{k=1}^{n} \Phi(my^k) dy^k\\
&= K \bigg[ |x^i - z^i| + \frac{1}{n} \sum_{j=1}^{n} |x^j - z^j| \bigg].
\end{align*}
Finally, assertion (4) follows immediately from assertion (2). 
\end{proof}
\begin{proof}[\textup{\bf Proof of Lemma \ref{finite_derivative_estimate}: }]
    \noindent{\bf Step 1. Lipschitz continuity of $\overline{v}_{\varepsilon,n,m}(t,\overline{x})$ in $\overline{x}$:} 
    Note that the identity \cite[(A.21)]{cosso_master_2022} may contain potential typographical errors, so we reproduce the proof of the Lipschitz continuity property here for the case involving common noise. We aim to establish that
\begin{align*}
\left| \overline{v}_{\varepsilon,n,m}(t, \overline{x}) - \overline{v}_{\varepsilon,n,m}(t, \overline{z}) \right| \leq \frac{C_4}{n} \left| \overline{x} - \overline{z} \right|,
\end{align*}
when the components of $\overline{x} = (x^1, \ldots, x^n)$ and $\overline{z} = (z^1, \ldots, z^n)$ are all equal, apart from one component $x^k \neq z^k$ for some $k=1,2,\ldots,n$. Recalling the definition in \eqref{def. tilde v_e,n,m}, we use the continuity in (3) of Lemma \ref{lem estimate of b^i_n,m...} to yield that
\begin{align}
&\left| \overline{v}_{\varepsilon,n,m}(t, \overline{x}) - \overline{v}_{\varepsilon,n,m}(t, \overline{z}) \right| \nonumber\\
&\leq 2K \sup_{\overline{\alpha} \in \overline{\mathcal{A}}^n_t} \frac{1}{n} \sum_{i=1}^{n} \mathbb{E} \left[ \int_{t}^{T} \left| \overline{X}_{s}^{i,m,\varepsilon,t,\overline{x},\overline{\alpha}} - \overline{X}_{s}^{i,m,\varepsilon,t,\overline{z},\overline{\alpha}} \right| ds + \left| \overline{X}_{T}^{i,m,\varepsilon,t,\overline{x},\overline{\alpha}} - \overline{X}_{T}^{i,m,\varepsilon,t,\overline{z},\overline{\alpha}} \right| \right].
\label{1897}
\end{align}
Suppose that $\overline{x}$ and $\overline{z}$ differ only for the first component $x^1 \neq z^1$. For $i = 1,2, \ldots, n$, the $\mathbb R^d$-valued process $\overline{X}^i := (\overline{X}_{s}^{i,m,\varepsilon,t,\overline{x},\overline{\alpha}})_{s\in[t,T]} $ solves the following equation on $[t, T]$:

\[
\overline{X}_{s}^i = x^i 
+ \int_{t}^{s} b_{n,m}^i (r, \overline{X}_{r}^1, \ldots, \overline{X}_{r}^n, \overline{\alpha}_{r}^i) dr 
+ \int_{t}^{s} \sigma (r, \overline{X}_{r}^i, \overline{\alpha}_{r}^i) d\overline{W}^i_{r} 
+ \int_{t}^{s} \sigma^0 (r,\overline{X}_{r}^i) d\overline{W}^0_{r}
+ \varepsilon (\overline{B}^i_{s} - \overline{B}^i_{t}).
\]
As the coefficients of the above equation are regular enough and have bounded continuous derivatives by Lemma \ref{lem estimate of b^i_n,m...}, the process $\Delta \overline{X}_{s}^i:=\overline{X}_{s}^{i,m,\varepsilon,t,\overline{z},\overline{\alpha}}-\overline{X}_{s}^{i,m,\varepsilon,t,\overline{x},\overline{\alpha}}$ satisfies
\begin{align*}
\Delta \overline{X}_{s}^i =\,&(z^1-x^1)\delta_{1i} 
+ \int_{t}^{s} \sum_{j=1}^{n} \widetilde{b}_{n,m,r}^{i,j}
\Delta \overline{X}_{r}^j dr
+ \int_{t}^{s}\sum_{k=1}^{d}\big(\Delta \overline{X}_{r}^i\big)_k\widetilde{\sigma}_{r}^{i,k}  d\overline{W}^i_{r},\\
&+ \int_{t}^{s}\sum_{k=1}^{d}\big(\Delta \overline{X}_{r}^i\big)_k\widetilde{\sigma}_{r}^{0,i,k} d\overline{W}^0_{r},
\end{align*}
where $\widetilde{b}_{n,m,r}^{i,j}:=\int^1_0\nabla_{x^j} b_{n,m}^i (r, \overline{X}_{r}^1+\theta \Delta \overline{X}_{r}^1, \ldots, \overline{X}_{r}^n+\theta \Delta \overline{X}_{r}^n, \overline{\alpha}_{r}^i) d\theta $, $\widetilde{\sigma}_{r}^{i,k}:=\int^1_0\p_{x_k} \sigma (r, \overline{X}_{r}^i+\theta \Delta \overline{X}_{r}^i,
\overline{\alpha}_{r}^i) d\theta $ 
and $\widetilde{\sigma}_{r}^{0,i,k}:=\int^1_0\p_{x_k} \sigma^0 (r, \overline{X}_{r}^i+\theta \Delta \overline{X}_{r}^i) d\theta $. The $\mathbb{R}^{dn}$-valued continuous process $\Delta \overline{X}_{s} := (\Delta \overline{X}_{s}^1, \ldots, \Delta \overline{X}_{s}^n)^\top$ is the unique solution to the above system of linear stochastic equations such that $\mathbb{E}\left[\sup_{s\in [t,T]}|\Delta \overline{X}_{s} |^2\right]<\infty$. Next, we provide an estimate of $\sup_{s\in [t,T]}\mathbb{E}\pig[\sum_{i=1}^n|\Delta \overline{X}_{s}^i |\pig]$ with a method akin to the proof of Tanaka's formula. Letting $\vartheta >0$, we consider the function $u_\vartheta:\mathbb{R}^d \to \mathbb{R}$ defined by 
\begin{align*}
    u_\vartheta(y):=\sqrt{|y|^2+\vartheta^2}.
\end{align*}
Direct calculation gives 
\begin{align*}
    \nabla_y u_\vartheta(y) = \frac{y}{u_\vartheta(y)},\quad \nabla_{yy}^2 u_\vartheta(x) = \frac{1}{u_\vartheta(y)}I_d - \frac{1}{\big[u_\vartheta(y)\big]^3}yy^\top.
\end{align*}
Applying It\^o's formula to $u_\vartheta\left(\Delta \overline{X}_{s}^i \right)$ gives 
\begin{align*}
&\h{-10pt}du_\vartheta\left(\Delta \overline{X}_{s}^i \right)\\
=\,& \pig\langle \nabla_y u_\vartheta(\Delta \overline{X}_{s}^i ), d\Delta \overline{X}_{s}^i \pig\rangle 
+ \frac{1}{2}\tr\left\{\left[\sum_{k=1}^{d}\big(\Delta \overline{X}_{s}^i\big)_k\widetilde{\sigma}_{s}^{i,k} \right]^\top \nabla_{yy}^2 u_\vartheta\left(\Delta \overline{X}_{s}^i \right)
\left[\sum_{k=1}^{d}\big(\Delta \overline{X}_{s}^i\big)_k\widetilde{\sigma}_{s}^{i,k} \right]\right\}ds\\
&+ \frac{1}{2}\tr\left\{\left[\sum_{k=1}^{d}\big(\Delta \overline{X}_{s}^i\big)_k\widetilde{\sigma}_{s}^{0,i,k} \right]^\top \nabla_{yy}^2 u_\vartheta\left(\Delta \overline{X}_{s}^i \right)
\left[\sum_{k=1}^{d}\big(\Delta \overline{X}_{s}^i\big)_k\widetilde{\sigma}_{s}^{0,i,k} \right]\right\}ds\\
=\,&\Bigg\langle \frac{\Delta \overline{X}_{s}^i }{u_\vartheta(\Delta \overline{X}_{s}^i )},\sum_{j=1}^{n} \widetilde{b}_{n,m,s}^{i,j}
\Delta \overline{X}_{s}^j ds
+ \left[\sum_{k=1}^{d}\big(\Delta \overline{X}_{s}^i\big)_k\widetilde{\sigma}_{s}^{i,k} \right]d\overline{W}^i_{s}
+ \left[\sum_{k=1}^{d}\big(\Delta \overline{X}_{s}^i\big)_k\widetilde{\sigma}_{s}^{0,i,k} \right] d\overline{W}^0_{s}\Bigg\rangle\\
&+\frac{1}{2}\tr\left\{\left[\sum_{k=1}^{d}\big(\Delta \overline{X}_{s}^i\big)_k\widetilde{\sigma}_{s}^{i,k} \right]^\top 
\left(\frac{1}{u_\vartheta(\Delta \overline{X}_{s}^i )}I_d - \frac{\pig(\Delta \overline{X}_{s}^i\pig) 
\pig(\Delta \overline{X}_{s}^i\pigr)^\top}{\pig[u_\vartheta(\Delta \overline{X}_{s}^i)\pigr]^3}\right)
\left[\sum_{k=1}^{d}\big(\Delta \overline{X}_{s}^i\big)_k\widetilde{\sigma}_{s}^{i,k} \right]\right\}ds\\
&+\frac{1}{2}\tr\left\{\left[\sum_{k=1}^{d}\big(\Delta \overline{X}_{s}^i\big)_k\widetilde{\sigma}_{s}^{0,i,k} \right]^\top 
\left(\frac{1}{u_\vartheta(\Delta \overline{X}_{s}^i)}I_d - \frac{\pig(\Delta \overline{X}_{s}^i\pig) 
\pig(\Delta \overline{X}_{s}^i\pigr)^\top}{\pig[u_\vartheta(\Delta \overline{X}_{s}^i)\pigr]^3}\right)
\left[\sum_{k=1}^{d}\big(\Delta \overline{X}_{s}^i\big)_k\widetilde{\sigma}_{s}^{0,i,k} \right]\right\}ds.
\end{align*}
It is clear that
\begin{align*}
\frac{\Delta \overline{X}_{s}^i}{u_\vartheta(\Delta \overline{X}_{s}^i)}\longrightarrow \frac{\Delta \overline{X}_{s}^i}{\pig| \Delta \overline{X}_{s}^i\pigr|}\mathds{1}_{\big\{\Delta \overline{X}_{s}^i\neq 0\big\}},\quad\mathbb{P}\text{-a.s. as $\vartheta\to 0$.} 
\end{align*}
Moreover,
\begin{align*}
    &\h{-10pt}\frac{1}{2}\tr\left\{\left[\sum_{k=1}^{d}\big(\Delta \overline{X}_{s}^i\big)_k\widetilde{\sigma}_{s}^{i,k} \right]^\top \left(\frac{1}{u_\vartheta(\Delta \overline{X}_{s}^i)}I_d\right)
    \left[\sum_{k=1}^{d}\big(\Delta \overline{X}_{s}^i\big)_k\widetilde{\sigma}_{s}^{i,k} \right]\right\}\\
    \longrightarrow&\frac{1}{2}\left(\frac{1}{\pig| \Delta \overline{X}_{s}^i\pigr|}\mathds{1}_{\big\{\Delta \overline{X}_{s}^i\neq 0\big\}}\right)
    \tr\left\{\left[\sum_{k=1}^{d}\big(\Delta \overline{X}_{s}^i\big)_k\widetilde{\sigma}_{s}^{i,k} \right]^\top 
    \left[\sum_{k=1}^{d}\big(\Delta \overline{X}_{s}^i\big)_k\widetilde{\sigma}_{s}^{i,k} \right]
    \right\}
\end{align*}
$\mathbb{P}$-a.s. as $\vartheta\to 0$. Similarly,
\begin{align*}
    &\h{-10pt}\frac{1}{2}\tr\left\{\left[\sum_{k=1}^{d}\big(\Delta \overline{X}_{s}^i\big)_k\widetilde{\sigma}_{s}^{i,k} \right]^\top 
    \left( \frac{\Delta \overline{X}_{s}^i \pig(\Delta \overline{X}_{s}^i\pigr)^\top}{\pig[u_\vartheta(\Delta \overline{X}_{s}^i)\pigr]^3}\right)
    \left[\sum_{k=1}^{d}\big(\Delta \overline{X}_{s}^i\big)_k\widetilde{\sigma}_{s}^{i,k} \right]\right\}\\
    \longrightarrow &\frac{1}{2}\left(\frac{1}{\pig| \Delta \overline{X}_{s}^i\pigr|^3}\mathds{1}_{\big\{\Delta \overline{X}_{s}^i\neq 0\big\}}\right)
    \tr\left\{\left[\sum_{k=1}^{d}\big(\Delta \overline{X}_{s}^i\big)_k\widetilde{\sigma}_{s}^{i,k} \right]^\top  \left(\Delta \overline{X}_{s}^i \pig(\Delta \overline{X}_{s}^i\pigr)^\top\right)
    \left[\sum_{k=1}^{d}\big(\Delta \overline{X}_{s}^i\big)_k\widetilde{\sigma}_{s}^{i,k} \right]\right\}
\end{align*}
$\mathbb{P}$-a.s. as $\vartheta\to 0$. The terms involving $\widetilde{\sigma}_{s}^{0,i,k}$ exhibit similar convergences as established in the preceding two results. Therefore, by taking the expectation and applying the dominated convergence theorem as $\vartheta \to 0$, we conclude that
\begin{align}
\label{dX_before_exp}
    &\h{-10pt}\mathbb{E}\pig[|\Delta \overline{X}_{s}^i|\pig]
    -|z^1-x^1|\delta_{1i} \nonumber\\  
    =&\mathbb{E}\left[\int^s_t\left\langle \frac{\Delta \overline{X}_{r}^i}{\pig|\Delta \overline{X}_{r}^i\pigr|}
    ,\sum_{j=1}^{n} \widetilde{b}_{n,m,r}^{i,j}
\Delta \overline{X}_{r}^j  \right\rangle 
    \mathds{1}_{\big\{\Delta \overline{X}_{r}^i\neq 0\big\}}dr\right]\nonumber\\
&+\frac{1}{2}\mathbb{E}\left\{\int^s_t\left(\frac{1}{\pig| \Delta \overline{X}_{r}^i\pigr|}\mathds{1}_{\big\{\Delta \overline{X}_{r}^i\neq 0\big\}}\right)
    \tr\left\{\left[\sum_{k=1}^{d}\big(\Delta \overline{X}_{r}^i\big)_k\widetilde{\sigma}_{r}^{i,k} \right]^\top 
    \left[\sum_{k=1}^{d}\big(\Delta \overline{X}_{r}^i\big)_k\widetilde{\sigma}_{r}^{i,k} \right]
    \right\}dr\right\}\nonumber\\
&-\frac{1}{2}\mathbb{E}\Bigg\{\int^s_t\left(\frac{1}{\pig| \Delta \overline{X}_{r}^i\pigr|^3}\mathds{1}_{\big\{\Delta \overline{X}_{r}^i\neq 0\big\}}\right)
    \tr\left\{\left[\sum_{k=1}^{d}\big(\Delta \overline{X}_{r}^i\big)_k\widetilde{\sigma}_{r}^{i,k} \right]^\top  \left(\Delta \overline{X}_{r}^i \pig(\Delta \overline{X}_{r}^i\pigr)^\top\right)
    \left[\sum_{k=1}^{d}\big(\Delta \overline{X}_{r}^i\big)_k\widetilde{\sigma}_{r}^{i,k} \right]\right\}dr\Bigg\}\nonumber\\
&+\frac{1}{2}\mathbb{E}\left\{\int^s_t\left(\frac{1}{\pig| \Delta \overline{X}_{r}^i\pigr|}\mathds{1}_{\big\{\Delta \overline{X}_{r}^i\neq 0\big\}}\right)
    \tr\left\{\left[\sum_{k=1}^{d}\big(\Delta \overline{X}_{r}^i\big)_k\widetilde{\sigma}_{r}^{0,i,k} \right]^\top 
    \left[\sum_{k=1}^{d}\big(\Delta \overline{X}_{r}^i\big)_k\widetilde{\sigma}_{r}^{0,i,k} \right]
    \right\}dr\right\}\nonumber\\
&-\frac{1}{2}\mathbb{E}\Bigg\{\int^s_t\left(\frac{1}{\pig| \Delta \overline{X}_{r}^i\pigr|^3}\mathds{1}_{\big\{\Delta \overline{X}_{r}^i\neq 0\big\}}\right)
    \tr\left\{\left[\sum_{k=1}^{d}\big(\Delta \overline{X}_{r}^i\big)_k\widetilde{\sigma}_{r}^{0,i,k} \right]^\top  \left(\Delta \overline{X}_{r}^i \pig(\Delta \overline{X}_{r}^i\pigr)^\top\right)
    \left[\sum_{k=1}^{d}\big(\Delta \overline{X}_{r}^i\big)_k\widetilde{\sigma}_{r}^{0,i,k} \right]\right\}dr\Bigg\}.
\end{align}
The term in the third line of \eqref{dX_before_exp} can be estimated by the Cauchy–Schwarz inequality:
\begin{align*}
    &\h{-10pt}\left(\frac{1}{\pig| \Delta \overline{X}_{r}^i\pigr|}\mathds{1}_{\big\{\Delta \overline{X}_{r}^i\neq 0\big\}}\right)
    \tr\left\{\left[\sum_{k=1}^{d}\big(\Delta \overline{X}_{r}^i\big)_k\widetilde{\sigma}_{r}^{i,k} \right]^\top 
    \left[\sum_{k=1}^{d}\big(\Delta \overline{X}_{r}^i\big)_k\widetilde{\sigma}_{r}^{i,k} \right]
    \right\}\nonumber\\
    =\,& \frac{1}{\pig| \Delta \overline{X}_{r}^i\pigr|}
    \sum_{p,q=1}^d\left|\sum_{k=1}^d\pig( \Delta \overline{X}_{r}^i\pigr)_k
    \pig(\widetilde{\sigma}_{r}^{i,k}\pigr)_{pq}\right|^2
    \mathds{1}_{\big\{\Delta \overline{X}_{r}^i\neq 0\big\}}\\
    \leq\,& K^2\pig| \Delta \overline{X}_{r}^i\pigr|.
\end{align*}
Similarly, the term in the forth line of \eqref{dX_before_exp} can be estimated by
\begin{align*}
    & \h{-10pt}\left(\frac{1}{\pig| \Delta \overline{X}_{r}^i\pigr|^3}\mathds{1}_{\big\{\Delta \overline{X}_{r}^i\neq 0\big\}}\right)
    \tr\left\{\left[\sum_{k=1}^{d}\big(\Delta \overline{X}_{r}^i\big)_k\widetilde{\sigma}_{r}^{i,k} \right]^\top  \left(\Delta \overline{X}_{r}^i \pig(\Delta \overline{X}_{r}^i\pigr)^\top\right)
    \left[\sum_{k=1}^{d}\big(\Delta \overline{X}_{r}^i\big)_k\widetilde{\sigma}_{r}^{i,k} \right]\right\}\\
    \leq &\frac{1}{\pig| \Delta \overline{X}_{r}^i\pigr|^3} \pig| \Delta \overline{X}_{r}^i\pigr|^4
    \sum_{k=1}^{d}\left|\widetilde{\sigma}_{r}^{i,k}\right|^2
\mathds{1}_{\big\{\Delta \overline{X}_{r}^i\neq 0\big\}}\\
    \leq &  K^2 \pig| \partial_{x^1_k} \overline{X}_{r}^i\pigr|.
\end{align*}
The terms involving $\widetilde{\sigma}_{s}^{0,i,k}$ can be estimated in a manner similar to the preceding results. Therefore, \eqref{dX_before_exp} reduces to
\begin{align*}
  \mathbb{E}\pig[|\Delta \overline{X}_{s}^i|\pig]
    &\leq |z^1-x^1|\delta_{1i}
    +\mathbb{E}\left[\int^s_t\sum_{j=1}^n\pig| \widetilde{b}_{n,m,r}^{i,j}
\Delta \overline{X}_{r}^j\pig| 
    +2K^2\pig|\Delta \overline{X}_{r}^i\pigr|dr\right].
\end{align*}
Summing over $i=1,2,\ldots,n$, we have
\begin{align*}
\mathbb{E}\left[\sum^n_{i=1}|\Delta \overline{X}_{s}^i| \right]
\leq \,&|z^1-x^1|+\int^s_t \mathbb{E}\left[\sum_{j=1}^{n}\left(\sum_{i=1}^{n}\big|\widetilde{b}_{n,m,r}^{i,j}\big|\right)
\big| \Delta \overline{X}_{r}^j \big|
+ 2K^2 \sum^n_{i=1}\pig| \Delta \overline{X}_{r}^i\pigr|\right]dr\\
\leq \,&|z^1-x^1|+\int^s_t \mathbb{E}\left[\max_{1\leq \ell\leq n}\left(\sum_{i=1}^{n}\big|\widetilde{b}_{n,m,r}^{i,\ell}\big|\right)
\sum^n_{j=1}\big| \Delta \overline{X}_{r}^j \big| 
+2K^2 \sum^n_{j=1}|\Delta \overline{X}_{r}^j|\right]dr.
\end{align*}
 The Lipschitz continuity estimate for $b_{n,m}^i$ in Lemma \ref{lem estimate of b^i_n,m...} deduces that
\begin{align*}
    \max_{1 \leq \ell \leq n} \sum_{i=1}^{n} \left| \widetilde{b}_{n,m,s}^{i,\ell} \right|
    &\leq \max_{1 \leq \ell \leq n} \sum_{i=1}^{n}\sup_{\overline{x}\in \mathbb{R}^{dn}} \left| \nabla_{x^\ell} b_{n,m}^i (s,\overline{x}, \overline{\alpha}_{s}^i) \right|\\
    &= \max_{1 \leq \ell \leq n}
    \left( \sup_{\overline{x}\in \mathbb{R}^{dn}}\left| \nabla_{x^\ell} b_{n,m}^\ell (s, \overline{x}, \overline{\alpha}_{s}^\ell) \right| 
    + \sum_{i=1, i \neq \ell}^{n}\sup_{\overline{x}\in \mathbb{R}^{dn}} \left| \nabla_{x^\ell} b_{n,m}^i (s, \overline{x}, \overline{\alpha}_{s}^i) \right| \right)\\
    &\leq \sqrt{d}\left[K \left( 1 + \frac{1}{n} \right) + K\frac{n-1}{n} \right]\\
    &= 2\sqrt{d}K.
\end{align*}
It gives us that $
    \mathbb{E}\left[\sum^n_{i=1}|\Delta \overline{X}_{s}^i| \right]
\leq|z^1-x^1| + 2(\sqrt{d}K+ K^2) \int_{t}^{s} \mathbb{E}\left[\sum^n_{i=1}\big|\Delta \overline{X}_{r}^i\big| \right] \, dr
    $. Gr\"onwall's inequality yields
\begin{align*}
\mathbb{E}\left[\sum^n_{i=1}\big|\Delta \overline{X}_{s}^i\big| \right] \leq |z^1-x^1|e^{2(\sqrt{d}K+K^2)T}, \quad \text{for every $s\in [t,T]$}.
\end{align*}
Thus, we obtain from \eqref{1897} that
\begin{align}
&\left| \overline{v}_{\varepsilon,n,m}(t, \overline{x}) - \overline{v}_{\varepsilon,n,m}(t, \overline{z}) \right| \nonumber\\
&\leq 2K \sup_{\overline{\alpha} \in \overline{\mathcal{A}}^n_t} \frac{1}{n} \sum_{i=1}^{n} \mathbb{E} \left[ \int_{t}^{T} \left| \overline{X}_{s}^{i,m,\varepsilon,t,\overline{x},\overline{\alpha}} - \overline{X}_{s}^{i,m,\varepsilon,t,\overline{z},\overline{\alpha}} \right| ds + \left| \overline{X}_{T}^{i,m,\varepsilon,t,\overline{x},\overline{\alpha}} - \overline{X}_{T}^{i,m,\varepsilon,t,\overline{z},\overline{\alpha}} \right| \right]\nonumber\\
&\leq\dfrac{C_{d,K,T}}{n}|z^1-x^1|.
\label{1606}
\end{align}

\noindent{\bf Step 2. $\overline{v}_{\varepsilon,n,m}(t,\overline{x})$ is the unique classical solution of \eqref{eq. bellman bar v_e,n,m}:} In this step, the definition of viscosity solution is referred to the usual Crandall-Lions' definition as the equation is on $[0,T] \times \mathbb{R}^{dn}$ instead of the Wasserstein space. First note that the volatility term (the second-order term) of equation \eqref{eq. bellman bar v_e,n,m} can be written as $\frac{1}{2}\tr\big[Q \nabla_{\bar{x}\bar{x}}^2 v\big]$, with 
\begin{align*}
    Q &:= \begin{pmatrix}
        \sigma\sigma^\top(t,x^1,a^1) &0 &\ldots &0\\
        0 &\sigma\sigma^\top(t,x^2,a^2) &\ldots &0\\
        \vdots &\vdots&\ldots&\vdots   \\
        0 &0 &\ldots &\sigma\sigma^\top(t,x^n,a^n)
    \end{pmatrix}
    +\varepsilon^2I_{dn}+\begin{pmatrix}
        \sigma^0(t,x^1)\\
        \sigma^0(t,x^2)\\
        \vdots\\
        \sigma^0(t,x^n)
    \end{pmatrix}\begin{pmatrix}
        \sigma^0(t,x^1)\\
        \sigma^0(t,x^2)\\
        \vdots\\
        \sigma^0(t,x^n)
    \end{pmatrix}^\top.
\end{align*} It can be easily shown that $Q$ is positive definite. Consequently, we have $Q \nabla_{\bar{x}\bar{x}}^2 \overline{v}_{\varepsilon,n,m} = \Sigma\Sigma^\top \nabla_{\bar{x}\bar{x}}^2\overline{v}_{\varepsilon,n,m}$, where $\Sigma$ satisfies $\Sigma\Sigma^\top \geq \varepsilon^2 I_{dn}$. Using \cite[Chatper 4.6, Theorem 6.2]{yong1999stochastic}, we conclude that $\overline{v}_{\varepsilon,n,m}$ is the unique viscosity solution of the Bellman equation \eqref{eq. bellman bar v_e,n,m}. Due to the $1/2$-H\"older continuity of $\overline{v}_{\varepsilon,n,m}(t,\overline{x})$ in $t$, Lipschitz continuity in $\overline{x}$ as established in \eqref{1606} and the boundedness of the non-homogeneous term $f^i_{n,m}$ of \eqref{eq. bellman bar v_e,n,m}, we consider the equation \eqref{eq. bellman bar v_e,n,m} on $[0,T] \times B_R \subset [0,T] \times \mathbb{R}^{dn}$ with the parabolic boundary value of $\overline{v}_{\varepsilon,n,m}(t,\overline{x})$. For this localized equation, we can apply \cite[Theorem 8.4]{CKS00} to obtain a unique strong solution $\overline{v}^\dagger_{\varepsilon,n,m}(t,\overline{x})$ which lies in $C^{\frac{1+\alpha}{2},1+\alpha}_{\textup{loc}}([0,T) \times B_R) \cap C([0,T] \times \overline{B_R})$. We see that $\overline{v}^\dagger_{\varepsilon,n,m}(t,\overline{x})$ is also the viscosity solution to \eqref{eq. bellman bar v_e,n,m} in $[0,T] \times B_R$ under the usual Crandall-Lions' definition, by \cite[Proposition 2.10]{CKS00}. Therefore, as the viscosity solution to \eqref{eq. bellman bar v_e,n,m} on $[0,T] \times \mathbb{R}^{dn}$ is unique, then we see that $\overline{v}_{\varepsilon,n,m}=\overline{v}^\dagger_{\varepsilon,n,m} \mathds{1}_{[0,T]\times B_R}
+\overline{v}_{\varepsilon,n,m} \mathds{1}_{[0,T]\times (\mathbb{R}^{dn}\setminus B_R)}$ for any $R>0$. In other words, $\overline{v}_{\varepsilon,n,m}$ is continuously differentiable with respect to $\overline{x}$ on $\mathbb{R}^{dn} \times[0,T)$. We consider another equation
\begin{equation}
\left\{
\begin{aligned}
				&\partial_t u(t,\overline{x})
				+\sup_{\overline{a} \in A^n}\Bigg\{\dfrac{1}{n}\sum^n_{i=1}f^i_{n,m}(t,\overline{x},a^i)
				+\dfrac{1}{2}\sum^n_{i=1}\textup{tr}\left[\Big((\sigma\sigma^\top)(t,x^i,a^i)+\sigma^0\sigma^{0;\top}(t,x^i)+\varepsilon^2I_d\Big)\nabla_{x^ix^i}^2
				u(t,\overline{x})\right]\\ 
				&\h{80pt}+\sum^n_{i=1}\left\langle b^i_{n,m}(t,\overline{x},a^i)
				,\nabla_{x^i}\overline{v}_{\varepsilon,n,m}(t,\overline{x})\right\rangle
				+\frac{1}{2}\sum_{i,j=1,i\neq j}^n\textup{tr}\Big[\sigma^0(t,x^i)\sigma^{0;\top}(t,x^j)\nabla_{x^ix^j}^2u(t,\overline{x})\Big]\Bigg\}\\
				&=0\h{5pt} \text{in $[0,T) \times\mathbb{R}^{dn}$};\\
				&u(T,\overline{x})
				=\dfrac{1}{n}\sum^n_{i=1}g^i_{n,m}(\overline{x})
				\h{5pt} \text{ in $\mathbb{R}^{dn}$}.
			\end{aligned}
			\right.   
			\label{eq. bellman bar v_e,n,m wit fixed term}
\end{equation}
In the above equation, we note the unknown $u$ is only involved in the time differentiation and the second-order spatial differentiation, the remaining terms are all in $C^{0,\alpha}([0,T)\times B_{R'})$ for any $R'>0$ and some $\alpha>0$. As $\overline{v}_{\varepsilon,n,m}$ is the unique viscosity solution to \eqref{eq. bellman bar v_e,n,m} on $[0,T] \times \mathbb{R}^{dn}$ under the usual Crandall-Lions' definition, it is also a viscosity solution to \eqref{eq. bellman bar v_e,n,m wit fixed term} under the usual Crandall-Lions' definition as it is continuously differentiable on $\overline{x}$. By considering the equation in \eqref{eq. bellman bar v_e,n,m wit fixed term} in a local region with the parabolic boundary value $\overline{v}_{\varepsilon,n,m}$, we can apply \cite[Theorem 5.2]{M19} to find a classical solution $\overline{v}^*_{\varepsilon,n,m}(t,\overline{x})$ to \eqref{eq. bellman bar v_e,n,m wit fixed term} on that local region. As the viscosity solution of \eqref{eq. bellman bar v_e,n,m wit fixed term} on $[0,T] \times \mathbb{R}^{dn}$ is unique by \cite[Chatper 4.6, Theorem 6.2]{yong1999stochastic}, we see that $\overline{v}_{\varepsilon,n,m}(t,\overline{x})=\overline{v}^*_{\varepsilon,n,m}(t,\overline{x})$ is actually classical.\\
\hfill\\
\noindent {\bf Step 3. Bounds of derivatives of $\overline{v}_{\varepsilon,n,m}(t,\overline{x})$:} Note that $\overline{v}_{\varepsilon,n,m}(t,\overline{x})\in C^{1,2}([0,T]\times\mathbb{R}^{dn})$, the boundedness of the first-order derivative with respect to $\overline{x}$ follows from \eqref{1606}, while the boundedness of the second-order derivative is established using \cite[Chapter 4.7, Theorem 4]{K08}.
\end{proof}


\bibliography{Bib}
\bibliographystyle{abbrv} 
\end{document}